\documentclass{article}

\usepackage{arxiv}
\usepackage[pdftex]{graphicx}
\usepackage[utf8]{inputenc} 
\usepackage[T1]{fontenc}    
\usepackage{hyperref}       
\usepackage{url}            
\usepackage{booktabs}       
\usepackage{amsfonts}       
\usepackage{nicefrac}       
\usepackage{microtype}      
\usepackage{lipsum}		
\usepackage{graphicx}

\usepackage{doi}

\usepackage{amsmath, amssymb, url}
\usepackage{subcaption}
\usepackage{wrapfig}
\usepackage[pdftex]{graphicx}
\usepackage{hyperref}
\usepackage{appendix}
\usepackage{xcolor}
\usepackage{amsthm}

\usepackage{amsfonts}
\usepackage{soul}
\usepackage{todonotes}

\usepackage{mathtools}
\usepackage{soul}
\usepackage{enumitem}

\newtheorem{lemma}{Lemma}[section]
\newtheorem{theorem}{Theorem}
\newtheorem{proposition}{Proposition}
\newtheorem{corollary}{Corollary}
\newtheorem{assump}{Assumption}
\newtheorem{example}{Example}
\newtheorem{remark}{Remark}

\newcommand{\Sml}{\mathcal{S}_{\mu}^{L}(\mathbb{R}^d)} 
\newcommand{\bbrho}{\bar{\bar \rho}^{(\varphi)}}
\newcommand{\htv}{\hat{t}^{(\varphi)}}
\newcommand{\hhrho}{\hat{\hat \rho}^{(\varphi)}}
\newcommand{\mtl}{l_{\theta}}
\newcommand{\mtv}{\mathtt{v}}
\newcommand{\mtk}{k_{\theta}}
\newcommand{\mtc}{\mathtt{c}}

\newcommand{\mE}{\mathbb{E}}
\newcommand{\mF}{\mathcal{F}}
\newcommand{\ttheta}{\tilde{\theta}}

\newcommand{\tmp}{\tilde{p}}

\newcommand{\veps}{w}
\newcommand{\mV}{\mathcal{V}}
\newcommand{\mr}{\mathbf{r}}

\newcommand{\tn}{\tilde{\nabla}}

\newcommand{\beq}{\begin{equation}}
\newcommand{\eeq}{\end{equation}}

\newcommand{\bc}[1]{{\color{blue} #1}}

\newcommand{\bcred}[1]{{\color{blue} #1}}
\newcommand{\mg}[1]{{\color{purple} #1}}

\usepackage{soul}

\newcommand{\bcgreen}[1]{{\color{orange} #1}}
\renewcommand{\bc}[1]{{\color{black} #1}}
\renewcommand{\bcred}[1]{{\color{black} #1}}
\renewcommand{\mg}[1]{{\color{black} #1}}
\renewcommand{\bcgreen}[1]{\color{black} #1}

\title{Entropic Risk-Averse Generalized Momentum Methods}

\author{Bugra Can \\
    Department of Management Science and Information Systems\\ Rutgers University\\
    Piscataway, NJ-08854, USA. \\
     \texttt{bc600@scarletmail.rutgers.edu} 
	\And
	Mert G\"urb\"uzbalaban\\
  Department of Management Science Information Systems and\\ Department of Electrical and Computer Engineering\\
  Rutgers University\\
  Piscataway, NJ-08854, USA.\\
  \texttt{mert.gurbuzbalaban@rutgers.edu}   
}

\renewcommand{\shorttitle}{Entropic Risk-Averse Generalized Momentum Methods}

\hypersetup{
pdftitle={A template for the arxiv style},
pdfsubject={q-bio.NC, q-bio.QM},
pdfauthor={David S.~Hippocampus, Elias D.~Striatum},
pdfkeywords={First keyword, Second keyword, More},
}

\begin{document}
\maketitle

\begin{abstract}
In the context of first-order algorithms subject to random gradient noise, we study the trade-offs between the convergence rate (which quantifies how fast the initial conditions are forgotten) and the ``risk” of suboptimality, i.e. deviations from the expected suboptimality. We focus on a general class of momentum methods (GMM) which recover popular methods such as gradient descent (GD), accelerated gradient descent (AGD), and the heavy-ball (HB) method as special cases depending on the choice of GMM parameters. We use well-known risk measures ``entropic risk" and ``entropic value at risk" to quantify the risk of suboptimality. For strongly convex smooth minimization, we first obtain new convergence rate results for GMM with a unified theory that is also applicable to both AGD and HB, improving some of the existing results for HB. We then provide explicit bounds on the entropic risk and entropic value at risk of suboptimality at a given iterate which also provides direct bounds on the probability that the suboptimality exceeds a given threshold based on Chernoff’s inequality. Our results unveil fundamental trade-offs between the convergence rate and the risk of suboptimality. We then plug the entropic risk and convergence rate estimates we obtained in a computationally tractable optimization framework and propose entropic risk-averse GMM (RA-GMM) and entropic risk-averse AGD (RA-AGD) methods which can select the GMM parameters to systematically trade-off the entropic value at risk with the convergence rate. We show that RA-AGD and RA-GMM lead to improved performance on quadratic optimization and logistic regression problems compared to the standard choice of parameters. To our knowledge, our work is the first to resort to coherent risk measures to design the parameters of momentum methods in a systematic manner.
\end{abstract}


\section{Introduction}\label{sec: introduction}

First-order methods based on gradient information are frequently used for solving large-scale convex optimization and machine learning problems due to their favorable scalability properties to large datasets and to high dimensions \cite{bottou2018optimization}. Classic analysis of first-order methods such as gradient descent or its accelerated versions based on momentum averaging for convex optimization assumes access to the \bcgreen{actual} gradients of the objective and provides convergence rate results to a global optimum. However, in many applications, it is often the case that we do not have access to the gradients but rather we possess a noisy (stochastic) unbiased estimate of the gradient. Examples would include but are not limited to stochastic gradient methods that estimate the gradient from randomly sampled subset of data points \cite{robbins1951stochastic} for large-scale optimization or machine learning problems or privacy-preserving first-order optimization algorithms that add independent identically distributed (i.i.d.) Gaussian (or Laplacian) noise to the gradients for preserving the privacy of user data \cite{kuru2020differentially,bassily2014private}. In such applications,  because the gradients are subject to persistent stochastic noise, the iterates \bcgreen{can} oscillate around the global optimum without converging to it. Consequently, suboptimality of the iterates could often fluctuate around a limiting non-zero mean value \cite{aybat2019robust,dieuleveut2020bridging} or could even diverge \cite{flammarion2015averaging}. Therefore, controlling the statistical properties of suboptimality $S_k$ at a given iterate $k$, such as the deviations of the suboptimality from its mean value or the probability that the suboptimality $S_k$ exceeds a certain given threshold $t$, becomes key performance metrics to consider in addition to the convergence rates traditionally considered in deterministic optimization \cite{pmlr-v99-jain19a,harvey2019tight}. 


When gradients are exactly known without any noise, it is well-known that momentum methods such as Nesterov's accelerated gradient descent (AGD) method and Polyak's heavy-ball (HB) method can significantly improve the convergence rate in terms of its dependency to the condition number \bcgreen{of the objective function} \cite{nesterov2003introductory,polyak1987introduction}. In fact, AGD method achieves the lower bounds for strongly convex smooth objectives in terms of its dependence to the condition number \cite{nesterov2003introductory}. However, in the presence of stochastic gradient noise, the situation is different. It has been empirically observed that momentum-based methods amplify the noise in the gradients \cite{Hardt-blog} and are less robust to noise in the sense that they need more accurate gradient information than gradient descent (GD) to achieve the same expected accuracy with the standard choice of parameters 
\cite{devolder2014first,aspremontSmooth08,flammarion2015averaging,Schmidt11InexactProx,Hardt-blog,devolder2013thesis}. In fact, for momentum-methods, it is known that there are fundamental trade-offs between the convergence rate and the asymptotic suboptimality \cite{aybat2019robust} \bcgreen{and the performance of momentum methods is quite sensitive to the choice of parameters}. If we view the suboptimality $S_k$ at step $k$ of a first-order algorithm as a random variable, then the existing guarantees for AGD and HB provides bounds on the expected suboptimality $\mathbb{E}\bcgreen{[S_k]}$ for every $k$ as a function of the stepsize and the momentum parameter (see e.g \cite{aybat2019robust,can2019accelerated,aybat2019universally,gadat-stoc-heavy-ball}) and provide conditions under which a limiting distribution exists for the iterates and the suboptimality \cite{can2019accelerated}. However, the expectation of a random variable is a point estimate and contains limited information about the underlying distribution (in the sense that two random variables can have the same expectation but can have arbitrarily different standard deviations or can have very different tail behavior). Therefore, existing approaches do not allow to optimize the parameters of momentum methods to control the tail probabilities $\mathbb{P}\bcgreen{\{S_k \geq t\}}$ for some \bcgreen{given} threshold level $t$ or to control the ``risk" of suboptimality, i.e. deviations from the expected suboptimality. 
%


In this paper, our purpose is to investigate the trade-offs between the convergence rate and ``risk" of suboptimality in designing the parameters of a first-order algorithm and devise a computationally tractable strategy to achieve a particular trade-off in a systematic manner. 
For this purpose, we focus on the \bcgreen{following} general class of \bcgreen{generalized} momentum methods (\bcgreen{GMM}) for minimizing a strongly convex smooth function $f:\mathbb{R}^d\to\mathbb{R}$, i.e. for solving the unconstrained optimization problem 
\bc{
\begin{equation}\label{opt_problem}
\min_{x\in\mathbb{R}^d} f(x),
\end{equation}
}
when noisy estimates of the gradients are available\footnote{\bcgreen{These methods were also called modified \textit{modified accelerated methods} in \cite{hu2017dissipativity}; however, to emphasize that they generalize AGD, HB; we use the name GMM.}}. The iterations consist of
\begin{subequations}
\begin{eqnarray}
x_{k+1}&=&(1+\beta)x_{k}-\beta x_{k-1}-\alpha \tilde \nabla f( y_k ), \label{TMM: Iter11}\\
y_{k}&=&(1+\gamma)x_k-\gamma x_{k-1} \label{TMM: Iter21}
\end{eqnarray}
\end{subequations} 
starting from the initialization $x_0, x_{-1}\in \mathbb{R}^d$, where $\tilde \nabla f( y_k )$ is an unbiased stochastic estimate of the actual gradient $\nabla f(y_k)$ \bcgreen{(see Sections \ref{sec: Prelim-and-background} and \ref{sec: Subgaussian} for further details)}. \bcgreen{GMM} has three parameters where $\alpha>0$ is the stepsize and $\beta, \gamma \geq 0$ are called momentum parameters. Depending on how the parameters are chosen, \bcgreen{GMM} recovers
AGD, GD, HB, \bcgreen{triple momentum methods \cite{gannot2021frequency,scoy-tmm}}\footnote{The original version of TMM from \cite{scoy-tmm} has four parameters in it but \cite{gannot2021frequency} considers the version with three parameters as in our setting.} and \mg{robust momentum methods \cite{vanscoy2021speedrobustness}} as special cases. 
To quantify the risk of suboptimality $S_k$; we use \emph{entropic risk} which is a well-known convex risk measure that depends on the risk-aversion of the user (measured by the risk-averseness parameter $\theta> 0$) through the exponential utility function \bc{ \cite{ruszczynski2013advances,Javid}}. Our choice of entropic risk measure is motivated by several facts. First, entropic risk of suboptimality $S_k$ is relevant to optimization practice because through Chernoff inequality, it implies bounds on the tail probability $\mathbb{P}\bcgreen{\{S_k \geq t\}}$ for any threshold $t$ (as will be discussed further in Section \ref{subsec: EVAR}). Second, we found that it was relatively easier to compute explicit upper bounds for it as opposed to other risk measures such as conditional value at risk (CV@R) which often require estimating the quantiles numerically \bc{\cite{Javid}}. 

We start with asking the following questions: $(i)$ ``What is the effect of parameters on the convergence rate and on  ``the entropic risk of suboptimality?", $(ii)$ ``Can we design the parameters of \bcgreen{GMM} to achieve particular trade-offs between the convergence rate and the entropic risk?". In this work, we provide answers to these questions in the context of strongly convex smooth minimization problems. As a high-level summary, our first contribution is to obtain explicit non-asymptotic characterizations of the entropic risk of suboptimality $S_k$ (Proposition \ref{prop: quad-risk-meas-gauss-noise}, Proposition \ref{prop: risk-meas-bound-gaussian}, Proposition \ref{prop: entr-risk-meas-str-cnvx-subgaus}) as well as to obtain new non-asymptotic convergence results for 
\bcgreen{GMM} as a function of the parameters (Lemma \ref{lem: non_asym_conv_quad_obj}, Theorem \ref{thm: TMM-MI-solution}). 
This allows us to get performance bounds depending on the desired level of risk-averseness \bcgreen{measured by the parameter $\theta>0$}. Our second contribution is that by optimizing our entropic risk bounds over the choice of the risk-averseness parameter $\theta$; we obtain bounds on the entropic value at risk (EV@R) of suboptimality (Theorem \ref{thm: quad-evar-bound}, Theorem \ref{thm: Evar-TMM-str-cnvx-bound}, Theorem \ref{thm: evar-bound-str-cnvx-subgaus}) where EV@R is a coherent risk measure with desirable properties. In particular, EV@R of suboptimality $S_k$ at a confidence level $\zeta \in (0,1)$ provides the tightest threshold $t_*$ (that can be obtained through the Chernoff inequality applied to suboptimality) such that $\mathbb{P}\bcgreen{\{S_k \geq t_* \}} \leq \zeta$ \cite{Javid}. Therefore, our results provide tail bounds as a function of the \bcgreen{GMM} parameters. Furthermore, our results demonstrate that for constant choice of stepsize and momentum parameters, EV@R \bcgreen{converges to an interval around zero} (as the iterations grow) with a linear rate that we can explicitly characterize (Theorems \ref{thm: quad-evar-bound}, \ref{thm: Evar-TMM-str-cnvx-bound}, and \ref{thm: evar-bound-str-cnvx-subgaus}). 
\bcgreen{In addition, our results show that there are fundamental trade-offs between the convergence rate (which determines the decay of the bias due to intialization) and the risk of suboptimality measured in entropic risk and EV@R.}
Our third contribution is that we insert our EV@R and convergence rate bounds into a computationally tractable optimization framework which allows us to choose the parameters to obtain systematic trade-offs between the convergence rate and EV@R in an (approximately) Pareto optimal fashion. Consequently, we can design the parameters of \bcgreen{GMM} to minimize the asymptotic EV@R of suboptimality subject to rate constraints. This allows us to control the tail probabilities  $\mathbb{P}\bcgreen{\{S_k \geq t\}}$ as a function of the parameters, where we show that our proposed design methodology for the parameters leads to an improved performance in practice compared to previously considered standard choice of parameters for \bcgreen{GMM} (Section \ref{sec: num-exp}). \mg{Having provided a high-level summary of our results, in the following few paragraphs, we expand upon the technical details of our contributions.}

\mg{First, for illustrative purposes, we focus on the special case when the objective is a strongly convex quadratic. In this case, existing guarantees for deterministic \bcgreen{GMM} \cite{gitman2019understanding} are asymptotic where the asymptotic rate is given by the spectral radius $\rho(A_Q)$ of the transition matrix $A_Q$ which is a quantity that depends on the Hessian $Q$ of the objective. We obtain non-asymptotic convergence guarantees (Lemma \ref{lem: non_asym_conv_quad_obj}); this is achieved by a careful analysis of the Jordan decomposition of the (typically) non-symmetric matrix $A_Q$. Then, we characterize the set of parameters for which the entropic risk is finite explicitly when the gradient noise is an isotropic Gaussian. This set is decreasing in the risk-averseness  parameter $\theta$, we show that as $\theta\to 0$ it recovers the set of parameters that leads to global linear convergence which we can also explicitly characterize (\bcgreen{see equation \eqref{eq: lim_F_sq} and Lemma \ref{lem: feas-set-stab-set}}). We show that limit of the entropic risk of suboptimality at step $k$ converges to a value (linearly with the convergence rate $\rho^2(A_Q)$) which we call infinite-horizon entropic risk \bcgreen{(Proposition \ref{prop: quad_qisk_meas_convergence})}. This result shows that convergence rate and the entropic risk are both
relevant for designing the parameters of GMM subject to random gradient noise and provides explicit characterizations of the risk of suboptimality at any iterate. Using these explicit characterizations, we show that there are fundamental trade-offs between the rate and the infinite-horizon entropic risk. Furthermore, we can show that a stationary distribution exists. We also characterize the entropic value at risk at stationarity for a given confidence level $\zeta$.\looseness=-1 
}

\mg{Second, we consider general strongly convex smooth objectives (that are not necessarily quadratics). In this setting, for GMM methods subject to \emph{random gradient noise}, the existing results on the bounds for expected suboptimality  are only local (except for particular choice of parameters which correspond to the special cases of GD, HB and AGD) \cite{gitman2019understanding}, i.e. holds only if the initial point is sufficiently close to the optimum and the existing analyses do not have all the constants explicitly available.\footnote{More specifically, Gitman \emph{et al.} \cite[Thm. 3]{gitman2019understanding} obtain asymptotic convergence rate results for quasi-hyperbolic momentum (QHM) methods for quadratic functions where the constants $\epsilon_k$ in this theorem are not explicitly available. This result can be used to provide local convergence results (but not global) for strongly convex functions and provides local convergence results for GMM as well if QHM parameters are appropriately scaled.} In the noiseless case, to our knowledge, explicit convergence rate results are available for only robust momentum methods which use particular choice of parameters (where all the parameters are expressed as a particular fixed function of the convergence rate) \cite[Theorem 21]{vanscoy2021speedrobustness} and also for TMM which corresponds to a particular (unique) choice of the GMM parameters \cite{scoy-tmm,vanscoy2021speedrobustness}. To our knowledge, rate estimates for more general GMM parameters can only be obtained numerically. More specifically, \cite[Lemma 5]{hu2017dissipativity} shows that if for any given $\rho^2$, if there exists a positive semi-definite matrix \bcgreen{$\tilde{P}$} satisfying a particular $3 \times 3$ matrix inequality (that depends on $\rho^2$) then GMM admits linear convergence with rate $\rho^2$. That being said, both the rate $\rho^2$ and the matrix $\tilde{P}$ are unfortunately not explicitly available for general choice of parameters but rather estimated numerically by solving this particular matrix inequality. 
Furthermore, existing results do not provide explicit \bcgreen{suboptimality} bounds at a given iterate $k$ \bcgreen{in the presence of gradient noise}. 
Another issue is that characterizing the trade-offs between the convergence rate and entropic risk explicitly and accurately as a function of the parameters would require developing new convergence rate guarantees for a wider choice of parameters than those available in the literature. In order to address these points, we identify a family of parameters $(\alpha,\beta,\gamma)$ that are parametrized by two free variables $(\vartheta,\psi)$. We provide explicit convergence rate results for \bcgreen{GMM} and performance bound for the noisy \bcgreen{GMM} iterations provided that $\vartheta$ and $ \psi$ satisfy an inequality we precise. The proof is based on extending the deterministic suboptimality analysis of \bcgreen{GMM} from \cite{hu2017dissipativity} to the noisy case and constructing an explicit choice of the Lyapunov matrix $\tilde{P}$ that serves as a solution to the matrix inequality provided in \cite[Lemma 5]{hu2017dissipativity}. These allow us to obtain explicit convergence rate results for a wider choice of parameters (Theorem \ref{thm: TMM-MI-solution}) that covers AGD and HB as a special cases. In particular, our convergence rate results for HB scales like $1 - \Theta(\sqrt{\alpha \mu)}$ for $\alpha$ small enough which is comparable to the convergence rate of AGD. When $\alpha$ is sufficiently small, our result shows an accelerated behavior, i.e. square root dependency to $\alpha$, and improves existing results from \cite{liu2020improved,ghadimi-heavy-ball} 
which obtained $1 - \Theta(\alpha \mu)$ rate for HB (see Remark \ref{remark-hb}). We then obtain explicit bounds on the finite-horizon entropic risk and infinite-horizon entropic risk depending on the parameters (Proposition \ref{prop: risk-meas-bound-gaussian}). Optimizing these bounds over $\theta$, we obtain characterization of the EVAR of $f(x_k) - f(x_*)$ at any $k$ which highlights the trade-offs between the choice of parameters. These results allow us to obtain bounds on the tail probability $\mathbb{P}\{f(x_k) - f(x_*) \geq t\}$ as a function of $t$. Moreover, by inserting the explicit bounds on EV@R and convergence rate in a computationally tractable optimization framework, we are able to select the stepsize and momentum parameters of \bcgreen{GMM} to achieve a desired trade-off between the convergence rate and the entropic value at risk. Since EV@R provides a bound on the tail probabilities, our framework allows us to optimize a bound on the tail probabilities as well.}

\mg{Finally, in Section \ref{sec: num-exp}, we provide numerical experiments that illustrate our results and practical benefits of our framework on quadratic optimization and logistic regression problems. While for simplicity of the discussion, we assumed that the gradient noise is isotropic Gaussian until the end of Section \ref{sec: num-exp}; we discuss in Section \ref{sec: Subgaussian} and Remark \ref{remark: general-params} that similar results hold for more general noise with (sub-Gaussian) light tails and for more general choice of parameters.}
\mg{
\subsection{Other Related Work}
Michalowsky and Ebenbauer \cite{michalowsky2014multidimensional} formulated the design of deterministic gradient algorithms as a state feedback problem and used robust control techniques and linear matrix inequalities to study their properties. Design and analysis of deterministic first-order optimization algorithms building on integral quadratic constraints from robust control theory have been studied in \cite{lessard2016analysis}. More recently, there has been a growing literature which uses dynamical system representations and control-theoretic techniques to study the convergence and robustness of optimization algorithms including \cite{hu2017approxSG,fazlyab2017analysis,fallah2019robust,can2019accelerated}. Design of the parameters of momentum-methods to achieve trade-offs in the bias and variance of the iterates or expected suboptimality have also been studied in \cite{aybat2019robust,vanscoy2021speedrobustness,mohammadi2020robustness,zhang2021robust,flammarion2015averaging,michalowsky2021robust,mohammadi2018variance}. However, all these research works do neither study the ``risk" of suboptimality nor the tail probabilities associated to suboptimality. Furthermore, the performance guarantees apply mainly to AGD, HB (and sometimes to their proximal variants) but not they do not apply to GMM with the more general choice of parameters we study in this paper. Scoy \emph{et al.} \cite{scoy-tmm} introduced the triple momentum method \bcgreen{(TMM)} and showed that it improves upon the convergence rate of AGD for deterministic optimization. Gannot \cite{gannot2021frequency} studied robustness of TMM and gradient descent when subject to deterministic relative gradient error. TMM corresponds to a particular choice of the stepsize and momentum parameters in GMM (see \cite{scoy-tmm,gannot2021frequency}). 
As opposed to these papers, we consider stochastic noise in the iterations and we consider an alternative general family of parameters for which we can develop convergence rate and risk estimates. 

Lan \cite{lan2012optimal} introduced the AC-SA method for stochastic composite optimization, where the objective function is given as the sum of general nonsmooth and smooth stochastic components. It is shown that the AC-SA algorithm is an optimal method for smooth, nonsmooth, and stochastic convex optimization. Ghadimi and Lan~\cite{ghadimi2012optimal,ghadimi-lan} studied AC-SA further for solving strongly convex composite optimization problems, developed tail bounds for it and showed that it can achieve a performance in expected suboptimality matching the lower complexity bounds in stochastic optimization. Another multi-stage AGD method with optimal complexity for expected suboptimality is developed in \cite{aybat2019universally}. Here, we focus on a more general class of methods (GMM) than AGD and focus on the design of parameters to achieve systematic trade-offs between the risk of suboptimality and the convergence rate instead.\looseness=-1 

For HB methods subject to random noise, there are many references in the literature that are not possible to list all. Among these most related to our paper, other than the references we mentioned above, Can \emph{et al.} \cite{can2019accelerated} \bcgreen{develops} tight accelerated convergence results for quadratic problems for optimization and Kuru \emph{et al.} \cite{kuru2020differentially} provides privacy guarantees for the user data. For quadratic objectives, Loizou \emph{et al.} \cite{Loizou-shb} shows accelerated linear rate in expected iterates and also linear rate in expected function values. Gadat \emph{et al.} \cite{gadat-stoc-heavy-ball} describe some almost sure convergence results for non-convex coercive functions and derive non-asymptotic results for strongly convex and convex problems with decaying stepsize. Li and Orabona \cite{orabona} studied adaptive SGD with momentum and showed the gradient norm is small with high-probability in the non-convex setting. Sebbouh \emph{et al.} \cite{gower-shb} obtains almost sure convergence rates for stochastic GD and stochastic HB. Yang \emph{et al.} \bcgreen{\cite{yang2016unified}} provides a unified convergence theory for stochastic AGD and HB with constant stepsize under bounded gradient assumption. In our setting, due to strong convexity, bounded gradient assumption does not hold and by exploiting strong convexity, we can obtain stronger guarantees than the convex case handled in \cite{yang2016unified} and in the other references listed above when the stepsize is smaller than a threshold we will specify.
}

\subsubsection*{Notations and Preliminaries} Throughout the paper, we use the notations $I_d$ and $0_d$ for $d\times d$ identity and zero matrices, respectively. \bc{We denote non-negative real numbers by $\mathbb{R}_{+}$ and positive real numbers by $\mathbb{R}_{++}$. }The Kronecker product of two matrices $A$ and $B$ is \bcgreen{denoted by} $A \otimes B$. The set $\Sml$ is the set of all $\mu$-strongly convex and $L$-smooth functions $f:\mathbb{R}^d \rightarrow \mathbb{R}$. We assume $\mu\neq L$, otherwise the class $\Sml$ is trivial. Notice the for all $f\in\Sml$, the solution of the optimization problem \eqref{opt_problem} is unique which we will denote by $x_*$. 
We say $f(n)=\mathcal{O}(g(n))$ if $f$ is bounded above by $g$ asymptotically, i.e. there exist constants $k_0>0$ and $n_0\geq 0$ such that $f(n)\leq k_0 g(n)$ for all $n>n_0$. \mg{For functions $f:\mathbb{R}_{+} \to \mathbb{R}_{+}$ and $g:\mathbb{R}_{+} \to \mathbb{R}_{+}$, we say $f(\theta)=o(g(\theta))$ as $\theta\to 0$ if for any constant $c_0>0$, there exists $\theta_0>0$ such that $f(\theta) \leq c_0 g(\theta)$ for all $\theta$ with $\theta < \theta_0$}. \mg{We say $f = \Theta(g(\theta))$ as $\theta\to 0$ if there exists positive constants $c_0, c_1, \theta_0 >0$ such that $c_1 g(\theta) \leq f(\theta) \leq c_0 g(\theta)$ for all $\theta$ with $\theta < \theta_0$}. We denote the multivariate Gaussian distribution with mean $\mu \in \mathbb{R}^{d}$ and covariance matrix $\Sigma\in\mathbb{R}^{d\times d}$ as $\mathcal{N}(\mu,\Sigma)$. \mg{For any real scalar $x>0$,  $\log(x)$ denotes the natural logarithm of $x$}. \mg{For an eigenvalue $\lambda_i(Q)$ of a $d\times d$ matrix $Q$, we will sometimes drop the dependency to $Q$ if it is clear from the context and use the notation $\lambda_i$ interchangeably.}\looseness=-1
\subsubsection*{Organization}
In Section \ref{subsec: dynam_sys_rep}, we provide the dynamical system representation of \bcgreen{GMM} which simplifies the presentation of our proofs. In Section \ref{subsec: EVAR}, we introduce the notions of ``entropic risk" and ``entropic value at risk" and discuss their relevant properties.
Then, in Section \ref{sec: quad-obj}, we study the infinite-horizon entropic risk and EV@R of \bcgreen{GMM} on strongly convex quadratic objectives where we explicitly calculate the infinite-horizon risk and provide an upper bound on the EV@R when the noise is an isotropic Gaussian.
In Section \ref{sec: strongly-conv-obj}, 
we extend our analysis to strongly convex objectives beyond quadratics where we present convergence results for \bcgreen{GMM} subject to Gaussian gradient noise as well as an upper bound on the EV@R. \bcgreen{Based on these results, we propose the ``entropic risk-averse \bcgreen{GMM}" which design the \bcgreen{GMM} parameters in a systematic way to trade-off the convergence rate with \bcgreen{EV@R}.}
\mg{In Section \ref{sec: num-exp}, we provide numerical experiments that provide the efficiency of the proposed ``entropic risk-averse \bcgreen{GMM}" methods.}  
\bcred{Finally, in Section \ref{sec: Subgaussian}, we discuss how our Gaussian noise assumption can be relaxed to more general light-tailed (sub-Gaussian) noise structure.
}
\section{Preliminaries and Background}\label{sec: Prelim-and-background}
\mg{In the following, we consider \bcgreen{GMM} iterations subject to gradient noise. We will sometimes be referring to the resulting iterations by ``noisy \bcgreen{GMM}" in the rest of the paper. We first discuss how the iterations with constant stepsize and momentum parameters can be formulated as a discrete-time dynamical system which will be helpful to simplify the presentation of our proof techniques. Then, we introduce the concepts of \emph{entropic risk} and \emph{entropic value at risk}.}
\subsection{Dynamical system representation of noisy \bcgreen{GMM}}\label{subsec: dynam_sys_rep}
Dynamical system representation of momentum algorithms have been considered in the literature which allows us to use tools from control theory for analysis purposes \cite{lessard2016analysis,hu2017dissipativity,vanscoy2021speedrobustness}. In our case, noisy \bcgreen{GMM} iterations can be represented as the
discrete-time dynamical system 
\begin{subequations}
\label{Sys: TMM}
\begin{eqnarray} 
\bcred{z}_{k+1}&=&A \bcred{z}_{k}+B\tn f(y_k), \quad \tn f(y_k) := \bcred{\nabla f(y_k)+\bc{w_{k+1}}},\\
y_k &=& C \bcred{z}_k,
\end{eqnarray}
\end{subequations}
starting from (deterministic) initialization $\bcred{z}_0 \in \mathbb{R}^d$ where $\tn f(y_k)$ is the noisy gradient, $\bc{\veps_{k+1}}=\tn f(y_k)-\nabla f(y_k)$ is the gradient noise at step $k$, $\bcred{z}_k = [ x_{k}^\top , x_{k-1}^{\top}]^\top$ is the state vector (that concatenates the last two iterates at time $k$) and the system matrices \mg{$(A, B, C)$ are given by}
\begin{align}\label{def-tilde-ABC} 
A = \tilde{A}\otimes I_d, \quad 
B = \tilde{B}\otimes I_d, \quad
C = \tilde{C}\otimes I_d,
\end{align}
where
\begin{align*}
\tilde{A}=\begin{bmatrix} 
(1+\beta) & -\beta \\ 
1 & 0
\end{bmatrix} , \;\; \tilde{B}=\begin{bmatrix} 
-\alpha \\ 
0
\end{bmatrix},\;\; \tilde{C}=\begin{bmatrix}
(1+\gamma) & -\gamma 
\end{bmatrix}.
\end{align*}
Regarding the noise sequence $\{w_k\}_{k\geq 1}$, we first consider a simplified setting \bcgreen{where} the gradient noise is i.i.d. \bc{Gaussian and $w_{k+1}$ is independent from the filtration generated by the iterates $\{x_j\}_{j=0}^{k}$ of noisy \bcgreen{GMM}.} Such noise type has been previously considered in the literature in a number of papers to study convergence rate and robustness to gradient noise for momentum methods \cite{aybat2019robust,fallah2019robust,can2019accelerated} and in privacy-preserving optimization algorithms \cite{abadi2016deep,ganesh2022langevin}; it will allow us to obtain stronger guarantees for the ``entropic risk" of the suboptimality we will introduce in the next section. However, our results extend to more general noise that can be light-tailed (sub-Gaussian \bcgreen{when conditioned on the iterates}) in a straightforward manner which will be discussed in 
Section \ref{sec: Subgaussian}.
\begin{assump}\label{Assump: Noise}
For each $k\in \mathbb{N}$, the gradient noise $\bc{\veps_{k+1}}=\tn f(\bcred{y}_k) - \nabla f(\bcred{y}_k) $ is distributed according to \bcred{$\bc{\veps_{k+1}}\sim \mathcal{N}(0,\sigma^2 I_d)$} for some $\sigma>0$. Furthermore, \bcred{$\bc{w_{k+1}}$ is independent from the filtration $\mathcal{F}_k$ generated by $\{x_j\}_{j=0}^{k}$}. 
\end{assump}

We note that in the noiseless setting (when $\veps_k=0$ for all $k$), \bcgreen{GMM} reduces to GD, AGD and HB methods for specific choice of the parameters $(\alpha,\beta,\gamma)$. For example, setting $\beta=0=\gamma$ corresponds to the GD algorithm which consists of the iterations
$$
x_{k+1}=x_k-\alpha \nabla f(x_k),
$$
whereas if we set $\gamma=0$; we recover the HB method with iterations.
$$
x_{k+1}=(1+\beta)x_k-\beta x_{k-1} -\alpha \nabla f(x_k).
$$
Similarly, in the special case $\beta=\gamma$ and $\bc{w_{k+1}}=0$, we obtain AGD method with iterations
\begin{align*}
x_{k+1}&=(1+\beta)x_k -\beta x_{k-1} -\alpha \nabla f(y_k ),\\
y_k &= (1+\beta)x_k-\beta x_{k-1}.
\end{align*}
\subsection{Entropic risk and entropic value at risk (EV@R)}\label{subsec: EVAR}

Before introducing the notion of EV@R, we first provide the definition of the \textit{finite-horizon entropic risk} $\mr_{k,\sigma^2}(\theta)$ of the random variable (suboptimality at step $k$) $f(x_k) - f(x_*)$: 
\begin{equation}
\mr_{k,\sigma^2}(\theta)= \frac{2\sigma^2}{\theta} \log \mathbb{E}[e^{\frac{\theta}{2\sigma^2} \left(f(x_k)-f(x_*)\right)}],
\label{def-finite-horizon-entropic-risk}
\end{equation}
\mg{where the expectation is taken with respect to all the randomness encountered (gradient noise) in the iterations,} $\theta > 0$ is called the \emph{risk-\bcgreen{averseness} parameter}, $\sigma^2$ is \mg{the variance of the noise in each coordinate (see Assumption \ref{Assump: Noise})}. 
Entropic risk defined in \eqref{def-finite-horizon-entropic-risk} of the random variable $f(x_k) - f(x_*)$ is a measure of the deviations from the mean of suboptimality $f(x_k) - f(x_*)$. It is a fundamental risk measure in statistics and arises naturally if we want to provide a bound on the tail \bc{probability} $\mathbb{P}\{ f(x_k)-f(x_*)\geq a\}$ at a threshold level $a$ \cite{Javid,ahmadi2019portfolio}. To see this, note that the Chernoff inequality applied to the non-negative random variable $f(x_k)-f(x_*)$ yields the following inequality for any $\theta>0$ and \bcred{${a}_k>0$}:  
\begin{align}
\mathbb{P}\{ f(x_k)-f(x_*)\geq \bcred{{a}_k}\}\leq e^{-\frac{\theta}{2\sigma^2} \bcred{{a}_k}}\mathbb{E}\left[e^{\frac{\theta}{2\sigma^2} (f(x_k)-f(x_*))} \right].\label{ineq: Chernoff-2}
\end{align}
If we define the variable $\zeta:=e^{-\frac{\theta}{2\sigma^2} \bcred{{a}_k} }\mathbb{E}[e^{\frac{\theta}{2\sigma^2} \left( f(x_k)-f(x_*)\right)}]$ which is equal to the right-hand side and write \bcred{${a}_k$} as a function of $\zeta$ and $\theta$, this inequality is equivalent to
\begin{equation}
\mathbb{P}\left\{f(x_k)-f(x_*)\geq \bcred{{a}_k}(\theta, \zeta) \right\}\leq \zeta, \quad \mbox{where} \quad \bcred{{a}_k}(\theta, \zeta) := \mr_{k,\sigma^2}(\theta)+\frac{2\sigma^2}{\theta}\log(\frac{1}{\zeta}).
\label{eq-tail-proba}
\end{equation}
Entropic risk takes deviations from the mean of $f(x_k)-f(x_*)$ into account. In fact, when $\mr_{k,\sigma^2}(\theta)$ is finite for $\theta$ around zero, if we apply a first-order Taylor expansion in $\theta$, we obtain
\begin{equation}\label{def: risk_meas_infty}
\mr_{k, \sigma^2}(\theta)=\frac{2\sigma^2}{\theta}\log \mathbb{E}[e^{\frac{\theta}{2\sigma^2} (f(x_k)-f(x_*))}]= \mathbb{E}[f(x_{k})-f(x_*)]+\frac{\theta}{4\sigma^2}\mathbb{E}[|f(x_k)-f(x_*)|^2]+o(\theta),
\end{equation}
which implies that $\bcgreen{\mr_{k,\sigma^2}(\theta)}$ accounts for both the mean and the variance of the suboptimality $\bcgreen{f(x_k)-f(x_*)}$ (see also \cite[Example 2.1]{fleming2006controlled}). We observe that if the \mg{risk} parameter \mg{$\theta$} gets larger; the deviations from the mean of suboptimality are penalized more. 

By taking the number of iterations (horizon) $k$ to infinity, one can also consider the following limit
\begin{equation}\label{def: inf-horizon-risk-measure}
 \mr_{\sigma^2}(\theta):= \bcgreen{\underset{k\rightarrow \infty}{\lim\sup}}\; \mr_{k,\sigma^2}(\theta)=\frac{2\sigma^2}{\theta} \bcgreen{\underset{k\rightarrow \infty}{\lim\sup}}\;\log \mathbb{E}\left[ e^{\frac{\theta}{2\sigma^2} \left( f(x_k)-f(x_*) \right)}\right],
\end{equation}
which is called the (infinite-horizon) \textit{entropic risk}. If the distribution of the iterates $x_k$ converges to a stationary distribution $\pi$, then under some further assumptions\footnote{such as the uniform integrability of the family of random variables $\big\{e^{\frac{\theta}{2\sigma^2} \left( f(x_k)-f(x_*) \right)}\big\}_{k\geq 0}$.}, we can replace the limit superior and the expectation to obtain
\begin{equation}\label{def: inf-horizon-risk-measure-2}
 \mr_{\sigma^2}(\theta)= \frac{2\sigma^2}{\theta} \log \mathbb{E}\left[ e^{\frac{\theta}{2\sigma^2} \left( f(x_\infty)-f(x_*) \right)}\right],
\end{equation}
where $x_\infty$ is a random variable whose distribution is the stationary distribution $\pi$.
Similar to the bounds \eqref{eq-tail-proba} and \eqref{ineq: Chernoff-2} for the finite-horizon risk, 
for any $\theta>0$, ${a}>0$, we can write 
\begin{align*}
\mathbb{P}\{ f(x_\infty)-f(x_*)\geq {a}\}\leq e^{-\frac{\theta}{2\sigma^2} {a}}\mathbb{E}\left[e^{\frac{\theta}{2\sigma^2} (f(x_\infty)-f(x_*))} \right],
\end{align*}
\bcred{
and, if we set $\zeta=e^{-\frac{\theta}{2\sigma^2} a }\mathbb{E}[e^{\frac{\theta}{2\sigma^2} \left( f(x_\infty)-f(x_*)\right)}]$ and write ${a}$ as a function of $\zeta$ and $\theta$, this inequality is equivalent to
\begin{equation}\label{ineq: tail-prob-bound}
\mathbb{P}\left\{f(x_\infty)-f(x_*)\geq a(\theta, \zeta) \right\}\leq \zeta, \quad \mbox{where} \quad a(\theta, \zeta) := \mr_{\sigma^2}(\theta)+\frac{2\sigma^2}{\theta}\log(\frac{1}{\zeta}).
\end{equation}}
We can see from the inequality above that the accuracy of the algorithm (measured in terms of asymptotic suboptimality) under confidence level $\zeta$ is related to the \bcgreen{entropic} risk $\mr_{\sigma^2}(\theta)$. As a matter of fact, the \textit{entropic value at risk} of $f(x_\infty)-f(x_*)$ with confidence level $\zeta$, which will be denoted by $EV@R_{1-\zeta}[f(x_\infty)-f(x_*)]$, is defined as the smallest value of $a(\theta,\zeta)$ over the positive choices of $\theta$; i.e.  
\begin{equation}\label{def: EVaR}
\text{EV@R}_{1-\zeta}\left[f(x_\infty)-f(x_*)\right]:= \underset{\theta>0}{\inf} a(\theta,\zeta) =  \underset{\theta>0}{\inf}\left\{\mr_{\sigma^2}(\theta)+\frac{2\sigma^2}{\theta}\log(\frac{1}{\zeta})\right\}.
\end{equation}
The relation between EV@R and the entropy can be seen through the lens of duality. More specifically, the dual representation of the EV@R of a random variable $X$ with an associated probability measure $P$ is given by
\begin{equation} {\displaystyle {\text{EV@R}}_{1-\zeta }\bcred{[X]}=\sup _{Q\in \Im_P }(E_{Q}(X))},
\label{eq-dual-representation}
\end{equation}
where $\Im_P$ is a set of probability measures with $\Im_P=\{Q\ll P:D_{KL}(Q||P)\leq \ln (\frac{1}{\zeta}) \}$ and ${\displaystyle D_{KL}(Q||P):=\int {\frac {dQ}{dP}}}$ $\left(\ln {\frac {dQ}{dP}}\right)dP$ is the relative entropy of $Q$ with respect to $P$ (a.k.a. the Kullback-Leibler (KL) distance between $Q$ and $P$) and $Q\ll P$ means that $Q$ is absolutely continuous with respect to $P$ \cite{Javid}. This representation and connection to relative entropy is the reasoning behind the name ``entropic value at risk". \mg{Thus, from \eqref{eq-dual-representation}, EV@R can be interpreted as a robust version of the expectation, and it quantifies the worst-case expectation of a random variable when its underlying distribution is perturbed up to an error term of $\log(1/\zeta)$ in the KL distance.}\looseness=-1

\mg{In addition to its connection to entropy, EV@R also defines a coherent risk measure which is commonly studied in the mathematical finance and statistics \cite{ahmadi2019portfolio,cajas2021entropic,shi2021coherent}. Here, coherency of a risk measure is a desirable property which means that it satisfies the natural monotonicity, sub-additivity, homogeneity, and translational invariance properties \cite{Ruszczynski2006,riedel}. Some of the other alternative well-known risk measures can be less desirable as they violate some of these assumptions. For example, \textit{Value at risk} (V@R) is an alternative risk measure but is not coherent as it does not respect the sub-additivity property \cite{Javid}.} 
\textit{Conditional value at risk} (CV@R) is an alternative coherent risk measure which is commonly used in statistics. However, one can argue that a disadvantage of CV@R is that it is not strongly monotone\footnote{This means that even if the suboptimality improves over one iteration satisfying $f(x_{k+1})-f(x_*) \leq f(x_{k})-f(x_*)$ pointwise and there is a positive probability that the suboptimality is strictly improved (i.e. $\mathbb{P}\left(f(x_{k+1)}- f(x_*) < f(x_k) - f(x_*)\right) >0$), it is possible that CV@R of the suboptimality is not strictly improved at iteration $k+1$ compared to iteration $k$ (see  \cite[Sec. 2]{ahmadi2019portfolio}\bcgreen{).}}; whereas EV@R, is a strongly monotone risk measure \cite{ahmadi2019portfolio}. 
We finally note that
EV@R provides the tightest possible upper bound (obtained from the Chernoff inequality) for both V@R and CV@R \cite{Javid}. Therefore, our results on entropic value at risk also provides immediate bounds on CV@R.\looseness=-1 

Despite entropic risk measures being frequently studied in statistics and mathematical finance; their use in the study and design of optimization algorithms is relatively \bcgreen{an understudied subject}. To our knowledge, our work is the first to use entropic risk to design the parameters of momentum-based optimization algorithms for achieving robustness to the gradient noise and improving their statistical properties. Our approach also yields novel tail bounds for the suboptimality associated to \bcgreen{GMM} iterations. 


\section{Main results for strongly convex quadratic objectives}\label{sec: quad-obj}

Let $f\in\Sml$ be a quadratic function of the form 
\begin{equation}\label{def: quad-func}
f(x)=\frac{1}{2} x^\top Q x + p^\top x+ r,
\end{equation} 
where the eigenvalues $\lambda_i(Q)$ of the symmetric matrix $Q$ satisfy 
$0<\mu = \lambda_1(Q) \leq ... \leq \lambda_d(Q)=L.$ 
Noticing that the global minimum $x_*$ satisfies $0=\nabla f(x_*)=Qx_*+p $, we can write $f(x)=\frac{1}{2}(x-x_*)^\top Q (x-x_*)$ and $\nabla f(x)=Q(x-x_*)$. Therefore, for such quadratic function $f$, the dynamical system \eqref{Sys: TMM} becomes  
\begin{align}\label{sys: TMM_quad}
\bcred{z}_{k+1}-\bcred{z}_*=A_Q (\bcred{z}_{k}-\bcred{z}_*) + B\bc{\veps_{k+1}},
\end{align}
where $\bcred{z}_{*}:=[x_*^\top,x_*^\top]^\top$, \bcred{ $z_0=[x_0^\top, x_{-1}^\top]^\top $}\bcgreen{is the initialization} and
\begin{equation}\label{def: A_Q}
A_Q :=\begin{bmatrix} (1+\beta)I_d - \alpha (1+\gamma )Q  & -(\beta I_d-\alpha\gamma Q) \\
    I_d &  0_d 
    \end{bmatrix}.
\end{equation}
\mg{We observe from  \eqref{sys: TMM_quad} that the distance to the optimum $ \bcred{z}_k - \bcred{z}_* $ follows a (stochastic) affine recursion governed by the matrix $A_Q$ where the randomness comes from the gradient noise. Since the noise is assumed to be centered, the expected distance follows the linear recursion $\mathbb{E}(\bcred{z}_{k+1}-\bcred{z}_*)=A_Q \mathbb{E} (\bcred{z}_{k}-\bcred{z}_*)$. Then, from Gelfand's formula \cite{gelfand}, it follows that 
$$\|\mathbb{E}[\bcred{z}_{k}-\bcred{z}_*]\| \leq (\rho^{k} + \varepsilon_k)\|\mathbb{E}[\bcred{z}_{0}-\bcred{z}_*]\|,$$ 
for some non-negative sequence $\{\varepsilon_k\}_{k\geq 0}$ with the property that $\varepsilon_k \to 0$. However, such a result is asymptotic as the sequence $\varepsilon_k$ is not explicitly provided. Such asymptotic results have been obtained in \cite{gitman2019understanding} for the class of hyperbolic momentum methods (HMM) which cover \bcgreen{GMM} (by proper scaling of HMM parameters). In the next result, we obtain a new non-asymptotic bound on the expected distance with explicit constants. The proof is deferred to the appendix, it is based on explicitly computing the Jordan decomposition of the $A_Q$ as a function of the \bcgreen{GMM} parameters and showing that the Jordan blocks are of size 2 in the worst-case.}
\begin{lemma}
\label{lem: non_asym_conv_quad_obj}
Consider the \bc{noisy} \bcgreen{GMM} iterates $\bcred{z}_k$ satisfying the recursion \eqref{sys: TMM_quad} for minimizing a quadratic function $f$ of the form \eqref{def: quad-func} where the gradient noise obeys \bcgreen{Assumption} \ref{Assump: Noise}. Then, we have for any $k\geq 1$,
\begin{small}
\begin{align} 
\Vert\mathbb{E}[\bcred{z}_{k}]-\bcred{z}_{*}\Vert 
&\bcgreen{\leq 
    \max_{i=1,2, \dots,d}
        \begin{cases}  
      \frac{|c_i^2+2d_i+2|}{\sqrt{|c_i^2+4d_i|}} \rho_i^{k}\Vert z_0-z_*\Vert, 
                    & \mbox{if} \quad c_i^2 + 4d_i \neq 0,  \\
(\frac{c_i^2}{4}+2)\sqrt{2\rho_i^{2k}+k^2\rho_i^{2k-2}}\Vert z_0-z_*\Vert,
                  & \mbox{if}\quad c_i^2 + 4d_i = 0,
            \end{cases}
            }\nonumber\\
&\leq \bcgreen{C_k} \rho(A_Q)^{k-1} \Vert \bcred{z}_0-\bcred{z}_* \Vert,
\label{tmm-quad-perf-bound}
\end{align}
\end{small}
where $ \rho(A_Q)= \max_{i\in \{ 1,..,d\}}\{\rho_{i}\}$ is the spectral radius of $A_Q$ with
\bc{\begin{equation}\label{def: quad_rate}
\rho_i := \begin{cases}
       \frac{1}{2}|c_i|+\frac{1}{2}\sqrt{c_i^2+4d_i}, & \text{ if } c_i^2 +4d_i \geq 0,\\
       \sqrt{|d_i|}, & \text{ otherwise},
\end{cases}
\end{equation}
}
and \bc{$$c_i :=(1+\beta) - \alpha (1+\gamma )\lambda_i(Q), \quad d_i = -(\beta -\alpha\gamma \lambda_i(Q)),$$ 
$\lambda_i(Q)$ are the eigenvalues of the Hessian matrix $Q$ of $f$ \bcgreen{for $i=1,2,...,d$},
\bcgreen{
$$
C_k^{(i)}:= \begin{cases}                  \frac{|c_i^2+2d_i+2| \rho_i}{\sqrt{|c_i^2+4d_i|}},
                    & \mbox{if} \quad c_i^2 + 4d_i \neq 0,  \\
                    (\frac{c_i^2}{4}+2)\sqrt{2\rho_i^2 + k^2},
                   & \mbox{if}\quad c_i^2 + 4d_i = 0,
            \end{cases}
$$
}
and $ C_k:= \max_i C_k^{(i)}$.
}
\end{lemma}
\mg{
\begin{remark}\label{remark-tmm-deter-rate} We note that Lemma \ref{lem: non_asym_conv_quad_obj} leads to new non-asymptotic performance bounds for deterministic \bcgreen{GMM} (where the gradient noise $w_k = 0$ for every $k$) as well. For deterministic methods, the iterates $\bcred{z}_{k}$ follow the recursion $\bcred{z}_{k+1} - \bcred{z}_* = A_Q(\bcred{z}_k - \bcred{z}_*)$ and the proof technique of Lemma \ref{lem: non_asym_conv_quad_obj} is still applicable. In this case, we can simply ignore the expectation in \eqref{tmm-quad-perf-bound} and the deterministic \bcgreen{GMM} iterates $\bcred{z}_{k}$ will satisfy the bound $\bc{\Vert}\bcred{z}_k - \bcred{z}_*\bc{\Vert} \leq C_k \rho(A_Q)^{k-1}\|\bcred{z}_0 - \bcred{z}_*\|$.
\end{remark}
}

For $\rho(A_Q)<1$, 
the expected iterates converge, i.e. $\mathbb{E}[{\bcred{z}}_k] \to \bcred{z}_*$, and $\rho(A_Q)$ determines the rate of convergence where $\rho(A_Q)$ depends on the choice of parameters $(\alpha,\beta,\gamma)$ (see \eqref{tmm-quad-perf-bound}). 
Next, we introduce the \emph{stable set} $\mathcal{S}_q$ as the set of parameters for which the iterates converge globally in expectation, i.e. 
\begin{equation}\label{def: stable_set}
\bcgreen{\mathcal{S}}_q:=\left\{(\alpha,\beta,\gamma) ~\mid~  \rho(A_Q)  <1 \right\},
\end{equation}
\mg{where the subscript $q$ in the definition of $\bcgreen{\mathcal{S}}_q$ is to emphasize the fact that the objective $f$ in question is a quadratic \bcred{function}.}
\mg{In the next result, we show that the finite-horizon entropic risk measure, $r_{k,\sigma^2}$, converges linearly to the infinite-horizon risk measure $r_{\sigma^2}$ \bcgreen{for} strongly convex quadratic objectives, where the rate of convergence is also governed by $\rho(A_Q)$. These results show that convergence rate and the entropic risk are both relevant for designing the parameters of \bcred{noisy \bcgreen{GMM}}. Basically, the infinite-horizon entropic risk characterizes the tail probabilities of suboptimality in an asymptotic fashion, whereas the convergence rate characterizes how fast this asymptotic region is attained.}

\begin{proposition}\label{prop: quad_qisk_meas_convergence}
Consider the \bc{noisy} \bcgreen{GMM} iterates $\bcred{z}_k$  satisfying the recursion \eqref{sys: TMM_quad} for minimizing a quadratic function $f$ of the form \eqref{def: quad-func} where the gradient noise obeys \bcgreen{Assumption} \ref{Assump: Noise}. \bcred{For a fixed risk parameter $\theta>0$, the risk measure $\mr_{\sigma^2}(\theta)$ is finite \bc{if and only if} the parameters $(\alpha,\beta,\gamma)$ belong to the $\theta$-feasible set $\mathcal{F}_\theta$:
\begin{equation}\label{cond: var-on-Sigma-quad}
\mathcal{F}_{\theta}:=\bigg\{ (\alpha,\beta,\gamma)\in \mathbb{R}_{++}\times \mathbb{R}_+\times \mathbb{R}_+ ~\vline~ \; \bcgreen{|c_i|<|1-d_i|} \mbox{ and } \theta < 2 u_i \mbox{ for } i=1,2,\dots,d \bigg\},
\end{equation}
where
\beq
u_i=\frac{(1+d_i)[(1-d_i)^2-c_i^2]}{\lambda_i(Q)(1-d_i)\alpha^2},
\label{def-ui}
\eeq
with $c_i$ and $d_i$ as in Lemma \ref{lem: non_asym_conv_quad_obj}.} \bc{Moreover, in this case we have $\rho(A_Q)<1$ and 
the \bcred{finite-horizon entropic risk}\bcgreen{, $\mr_{k,\sigma^2}(\theta)$,} converges to the \bcred{infinite-horizon} entropic risk\bcgreen{, $\mr_{\sigma^2}(\theta)$,} linearly at a rate $\rho(A_Q)$ satisfying}
\mg{
\begin{eqnarray} 
| \mr_{k,\sigma^2}(\theta)-\mr_{\sigma^2}(\theta)| \leq \mathcal{O} \left(C_k^{2}\rho(A_Q)^{2(k-1)} +  C_k^4\rho(A_Q)^{4(k-1)} \right), 
\end{eqnarray}
where $C_k$ is as in Lemma \ref{lem: non_asym_conv_quad_obj}, and $\mathcal{O}(\cdot)$ hides constants that depends on the initialization $\bcred{z}_0$.
}
\end{proposition}
\begin{proof}
\bcgreen{The proof is deferred to Appendix \ref{app: quad_qisk_meas_convergence}}. 
\end{proof}

\bcgreen{By definition,} the sets $\mathcal{F}_\theta$ are increasing as $\theta$ gets smaller, \bcgreen{ in the sense that $\mathcal{F}_{\theta_1}\subset \mathcal{F}_{\theta_2}$ for $\theta_1 > \theta_2$}. 
Roughly speaking, this says that the less risk-averse we are (for smaller $\theta$); we have more flexibility in the choice of the parameters. To illustrate this point, in Figure \ref{fig: feasible-region}, we present the $\mathcal{S}_q$ and the sets $\mathcal{F}_\theta$ for different values of $\theta$ for the two-dimensional quadratic function $f(x_{(1)},x_{(2)})=x_{(1)}^2+0.1x_{(2)}^2$ where $x_{(1)},x_{(2)}\in\mathbb{R}$ and we take $\sigma^2=1$. It can be seen that as the risk parameter $\theta$ increases, the set $\mathcal{F}_\theta$ gets smaller; i.e. we need to choose stepsize and momentum more conservatively to be able to limit deviations from the mean and control the risk.\footnote{In Figure \ref{fig: feasible-region}, we first plot the set $\mathcal{S}_q$ in blue, and then the sets $\mathcal{F}_{0.25}, \mathcal{F}_{1}$ and $\mathcal{F}_{2}$ in purple, brown and green color respectively. If an area is colored multiple times, the plot will display the last color used. Consequently, the set $\mathcal{F}_1$ is given by the union of the green and brown segment, $\mathcal{F}_2$ is the union of green, brown and purple segment, and so on.} In fact, in Lemma \ref{lemma-Ftheta-vs-Sq} of the \bcgreen{Appendix}, we \bcgreen{also} show that $\mathcal{F}_\theta$ is a subset of $\mathcal{S}_q$ for every $\theta>0$ and we have 
\begin{equation}\label{eq: lim_F_sq}
\bcgreen{\lim_{\theta\to 0} \mathcal{F}_\theta=\mathcal{S}_q=\left\{ (\alpha,\beta,\gamma)\in \mathbb{R}_{++}\times \mathbb{R}_{+}\times \mathbb{R}_{+}~|~|c_i| < |1-d_i|,\;\; 0<u_i \; \text{for $i$=1,...,d}\right\}.}
\end{equation}
\bcgreen{This formula characterizes the set $\bcgreen{\mathcal{S}}_q$ as a function of the GMM parameters and is new to our knowledge.} 
\begin{figure}
\centering
    \includegraphics[width=0.4\linewidth,height=0.4\linewidth]{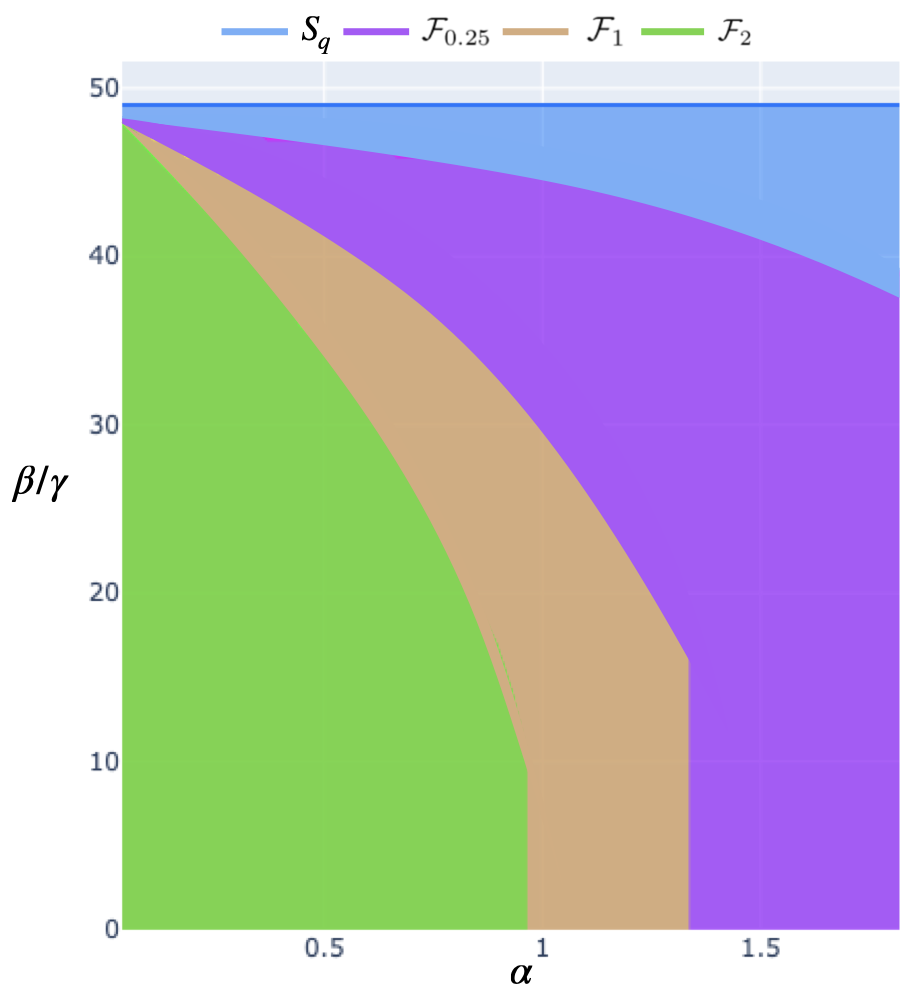}
    \caption{The feasible region $\mathcal{F}_{\theta}$ versus the stable set $\mathcal{S}_q$ for $f(x_{(1)},x_{(2)})=x_{(1)}^2+0.1x_{(2)}^2$ where $x_{(1)}, x_{(2)}\in\mathbb{R}$ and $\sigma^2=1$.}
    \label{fig: feasible-region}
    \vspace{-0.2cm}
\end{figure}
Next, we study the infinite-horizon entropic risk $\mr_{\sigma^2}(\theta)$ and the entropic value at risk of \bc{noisy} \bcgreen{GMM} on a strongly convex quadratic objective \bcred{under }\bcgreen{Assumption \ref{Assump: Noise}.}
\bcred{Let $U\Lambda U^\top$ be the eigenvalue decomposition of $Q$ where $\Lambda$ is a diagonal matrix containing the eigenvalues. With the change of variable $\tilde{x}_k:= U^\top x_k$ and $\tilde{\veps}_{k+1} := U^{\top} \bc{\veps_{k+1}}$, we can see that $\tilde{x}_k$ follows the recursion:}
\begin{eqnarray*}
\tilde{x}_{k+1}=(1+\beta) \tilde{x}_k-\beta \tilde{x}_{k-1}-\alpha \Lambda \left((1+\gamma)\tilde{x}_k -\gamma \tilde{x}_{k-1}\right)+\bc{\tilde{\veps}_{k+1}}.
\end{eqnarray*}
\bcred{Introducing $\tilde{\bcred{z}}_k:=[\tilde{x}_k^\top,\tilde{x}_{k-1}^\top]^{\top}$ and $\tilde{\bcred{z}}_*:=[\tilde{x}_*^\top,\tilde{x}_*^\top]^{\top}$, then the recursion becomes}
\begin{align} \label{sys: STMM_quad_Lambda}
\tilde{\bcred{z}}_{k+1}-\tilde{\bcred{z}}_*=A_\Lambda (\tilde{\bcred{z}}_k-\tilde{\bcred{z}}_*)+B \bc{\tilde{\veps}_{k+1}},
\end{align} 
\mg{where with slight abuse of notation $A_\Lambda$ refers to the matrix $A_Q$ defined in \eqref{def: A_Q} when we plug in $Q = \Lambda$}. \bcred{With a further change of variable $\bar{\bcred{z}}_k := \tilde{\bcred{z}}_k-\tilde{\bcred{z}}_*$, and using the fact that $\bc{\veps_{k+1}}$ and $\bc{\tilde{\veps}_{k+1}}$ have the same distribution we obtain} 
\begin{align}\label{sys: TMM_dynmcal_quad_obs-1}
\bar{\bcred{z}}_{k+1}&= A_{\Lambda} \bar{\bcred{z}}_k + B \bc{\veps_{k+1}},
\end{align}
here equality means that both sides have the same distribution. \bcred{If $(\alpha,\beta,\gamma)\in \mathcal{S}_q$, then it follows from Proposition \ref{prop: quad_qisk_meas_convergence} that the finite-horizon risk measure will converge to the infinite-horizon risk and we have\begin{equation}\label{eq-entropic-risk-quad}
\mr_{\sigma^2}(\theta)=\limsup_{k\to\infty}\frac{2\sigma^2}{\theta}\log \mathbb{E}[e^{\frac{\theta}{2\sigma^2}(f(x_k)-f(x_*))}] =\limsup_{k\to\infty}\frac{2\sigma^2}{\theta}\log \mathbb{E}[e^{\frac{\theta}{2\sigma^2}\Vert S \bar{\bcred{z}}_k \Vert^2}], 
\end{equation}
where in the last equality we used the eigenvalue decomposition of $Q$, \bcgreen{i.e. we choose the matrix $S=\frac{1}{\sqrt{2}}[\Lambda^{1/2}, 0_d]$ so that}
\begin{align*}
f(x_k)-f(x_*)=\frac{1}{2}(x_k-x_*)Q(x_k-x_*)
=\frac{1}{2}\Vert \Lambda^{1/2}U^\top (x_k-x_*)\Vert^2\bcgreen{=}\mg{\Vert S \bar{\bcred{z}}_k \Vert^2}. 
\end{align*}
}

\bcred{Next,} we explicitly compute the entropic risk for noisy \bcgreen{GMM} under the \bcred{isotropic} Gaussian noise (Assumption \ref{Assump: Noise}) and provide an upper bound on the EV@R of the asymptotic suboptimality. Since the noise has mean zero,  
we see from \eqref{sys: TMM_dynmcal_quad_obs-1} that  $\mathbb{E}[\bar{\bcred{z}}_{k}|\mathcal{F}_{k-1}]=A_\Lambda \bar{\bcred{z}}_{k-1}$ where $\mathcal{F}_{k-1}$ is the natural filtration generated by the iterates $\{\bcred{z}_j\}_{j=0}^{k-1}$ and \bc{the matrix} $V_k:=\mathbb{E}[\bar{\bcred{z}}_k \bar{\bcred{z}}_k^{\top}]$ satisfies the recursion
\begin{align*} 
V_{k+1}&=\mathbb{E}[\left( A_{\Lambda}\bar{\bcred{z}}_k + B\bc{\veps_{k+1}} \right)\left(A_{\Lambda}\bar{\bcred{z}}_k+B\bc{\veps_{k+1}} \right)^{\top}  ]\\
&= \mathbb{E}[ A_{\Lambda}\bar{\bcred{z}}_k\bar{\bcred{z}}_k^{\top}A^{\top}_{\Lambda} + A_{\Lambda}\bar{\bcred{z}}_k(\bc{\veps_{k+1}})^{\top}B^{\top}+ B\bc{\veps_{k+1}} \bar{\bcred{z}}_k^{\top}A^{\top}_{\Lambda}+B(\bc{\veps_{k+1}})(\bc{\veps_{k+1}})^{\top}B^{\top}]\\
&= A_{\Lambda} V_kA^{\top}_{\Lambda}+ \sigma^2BB^{\top}.
\end{align*}
\mg{Given that $(\alpha,\beta,\gamma)\in \mathcal{S}_q$, \bcred{we have} $\rho(A_\Lambda) = \rho(A_Q)<1$ and $\lim_{k\to\infty}\mathbb{E}[\bcred{z}_k]=\bcred{z}_*$ \bcred{from Lemma \ref{lem: non_asym_conv_quad_obj}}. \bcgreen{Since $\rho(A_\Lambda)<1$}, it follows from the recursion above that the limit $V_{\infty}:= \lim_{k\rightarrow \infty} V_{k}$ exists (see \cite[Theorem 5 D.6]{chen-book}) and satisfies the Lyapunov equation
\begin{equation}\label{eq: Quad_Lyapunov}
V_{\infty}= A_{\Lambda}V_{\infty}A_{\Lambda}^{\top} + \sigma^2 BB^{\top}.
\end{equation}
Under Assumption \ref{Assump: Noise}, we also observe from equality \eqref{sys: TMM_dynmcal_quad_obs-1} that $\bar{\bcred{z}}_{k}$ is a centered Gaussian random variable for every $k$ as $\bcred{z}_k$ is Gaussian. This implies that the limit $\bar{\bcred{z}}_\infty = \lim_k \bar{\bcred{z}}_{k}$ is also a Gaussian with a zero mean and with a covariance matrix $V_\infty$. Therefore, the distribution of the iterates $x_k$ converges to a Gaussian distribution as well; let $x_\infty$ denote a random variable with this limiting distribution. We also observe from \eqref{eq-entropic-risk-quad} that $\mr_{\sigma^2}(\theta) = \frac{2\sigma^2}{\theta} \log \mathbb{E}[e^{\frac{\theta}{2}X_\infty^\top X_\infty}]$ where $X_\infty = \frac{S\bar{\bcred{z}}_\infty}{\sigma}$ has a Gaussian distribution. By computing this integral explicitly, in the next result, we obtain an explicit characterization of the infinite-horizon entropic risk. The details of the proof is given in Appendix \ref{app: quad-risk-meas}.} 

\begin{proposition}\label{prop: quad-risk-meas-gauss-noise} Consider the \bc{noisy} \bcgreen{GMM} iterates $\bcred{z}_k$  satisfying the recursion \eqref{sys: TMM_quad} for minimizing a quadratic function $f\in\Sml$ of the form \eqref{def: quad-func} where the Hessian matrix $Q$ has the eigenvalues $0<\mu=\bcgreen{\lambda_1(Q)\leq \dots\leq}$ $\lambda_{d}(Q)=L$. Assume that the gradient noise obeys \bcgreen{Assumption} \ref{Assump: Noise}.
\bcgreen{The infinite-horizon entropic risk measure, $\mr_{\sigma^2}(\theta)$, is finite if and only if $(\alpha,\beta,\gamma)\in\mathcal{F}_{\theta}$ \bcgreen{, where $\mathcal{F}_\theta$ is as in \eqref{cond: var-on-Sigma-quad},} and is given as}
\begin{equation} \label{eq: STMM_Quad_mr}
\mg{\mr_{\sigma^2}(\theta) = 
-\frac{\sigma^2}{\theta}\sum_{i=1}^{d} \log \left(1-\frac{\theta}{2u_i}\right).}
\end{equation}
\end{proposition}

\begin{figure}[h]
\includegraphics[width=\linewidth]{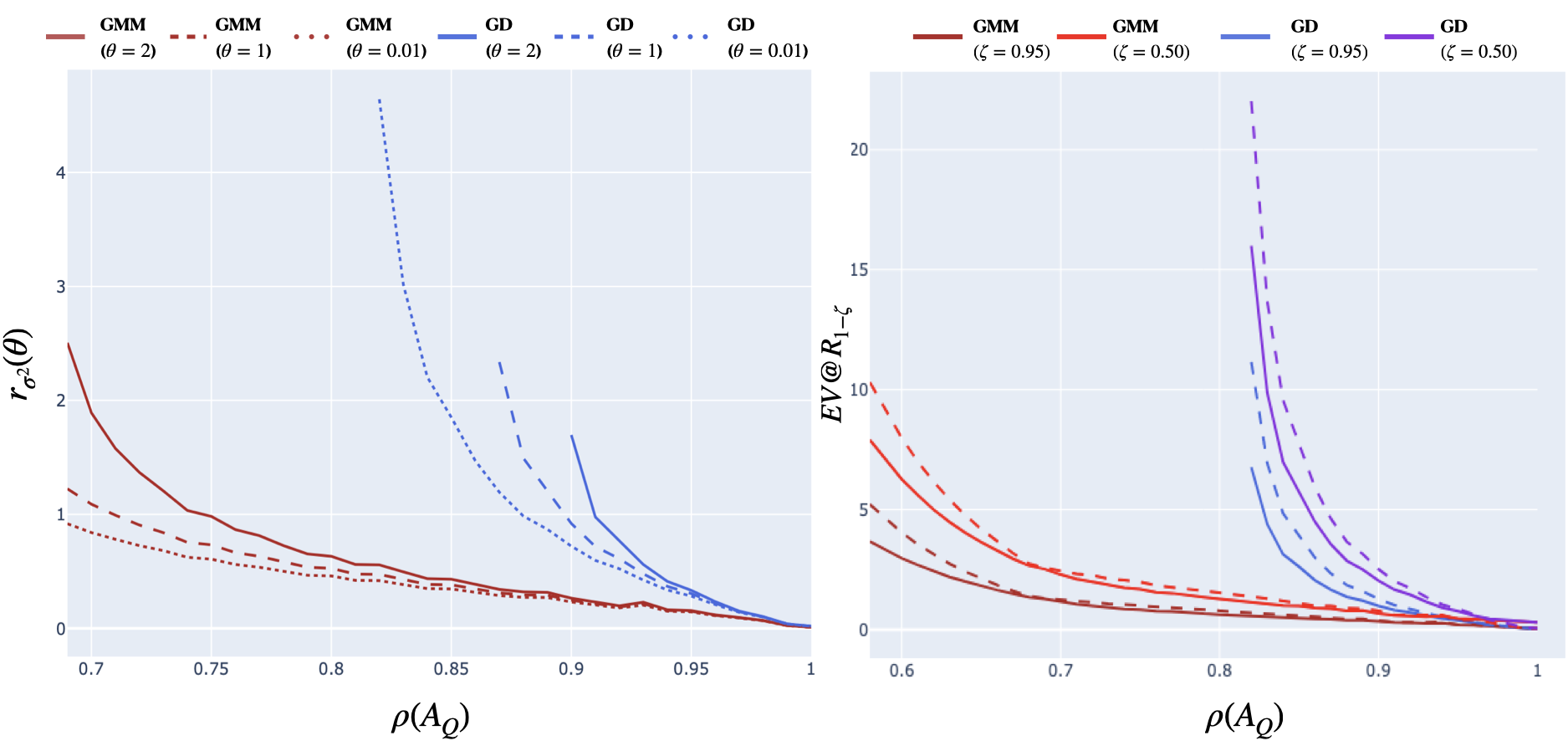}
\caption{{\small \bcgreen{Illustration of the} results of Proposition \ref{prop: quad-risk-meas-gauss-noise} and Theorem \ref{thm: quad-evar-bound} for noisy \bcgreen{GMM} and noisy GD on $f(x_{(1)},x_{(2)})=x_{(1)}^2+0.1 x_{(2)}^2$ with $\sigma^2 = 1$. \textbf{Left:} The convergence rate vs. optimal (smallest) infinite-horizon risk attainable at this convergence rate \bcgreen{for} $\theta \in \{0.01,1,2\}$, \textbf{Right:} The comparison of $EV@R_{1-\zeta}\bcred{[f(x_\infty)-f(x_*)]}$ \bcred{(straight lines)} and its approximation $\bar{E}^q_{1-\zeta}(\alpha,\beta,\gamma)$ \bcred{(dashed lines)} at confidence levels $\zeta\in \{0.95, 0.5\}$.}}
\label{fig: rate_vs_risk_meas_and_evar_quad}
\vspace{-0.2in}
\end{figure}

On the left panel of Figure \ref{fig: rate_vs_risk_meas_and_evar_quad}, we illustrate the trade-offs between the infinite-horizon entropic risk measure and the convergence rate $\rho(A_Q)$ where we compare noisy \bcgreen{GMM} with noisy gradient descent when the objective   
is the quadratic function $f(x_{(1)},x_{(2)})=x_{(1)}^2+0.1 x_{(2)}^2$ \bc{where $x_{(1)}, x_{(2)}\in\mathbb{R}$}. 
More specifically, for \bcgreen{GMM}, we grid the parameters $(\alpha,\beta,\gamma)$ to find the best (smallest) entropic risk that can be attained at a particular convergence rate $\rho(A_Q)$; i.e. among all the parameter choices $(\alpha,\beta,\gamma)$ that leads to a particular rate $\rho(A_Q)$, we calculate the smallest entropic risk that can be achieved. For noisy gradient descent, we follow the same procedure except that we take $\beta=\gamma=0$ and only vary the stepsize $\alpha$. We repeat this procedure for $\theta \in \{0.01,1,2\}$.
We observe that faster convergence (smaller values of $\rho(A_Q)$) is associated with worse (larger) entropic risk. Noisy \bcgreen{GMM} achieves better (smaller) entropic risk compared to noisy GD at any given convergence rate as expected; this is because \bcgreen{GMM} has more parameters than gradient descent to optimize over. We also observe that when the risk parameter $\theta$ is increased, the best risk that can be achieved at a given convergence rate level will increase. In fact, 
\mg{it can be seen from \eqref{eq: STMM_Quad_mr} that $\mr_{\sigma^2}(\theta)$ is a convex function of $\theta$ for $\theta \in (0, 2\min_i u_i)$ and $\mr_{\sigma^2}(\theta) \to \infty$ as $\theta \to 2\min_i u_i$. Therefore, based on the formula \eqref{def: EVaR}, $EV@R_{1-\zeta}[f(x_\infty)-f(x_*)]$ can be computed easily up to high accuracy by minimizing a one-dimensional convex objective \bcgreen{(in $\theta$)} over the interval $(0, 2\min_i u_i)$. We use the entropic risk formula from Proposition \ref{prop: quad-risk-meas-gauss-noise} to provide 
an explicit upper bound on the EV@R of suboptimality. This is achieved by approximating the optimal solution of the optimization problem 
\eqref{def: EVaR} that needs to be solved for computing the EV@R. The proof can be found in Appendix \ref{app: quad-evar-bound}.} 
\begin{theorem}\label{thm: quad-evar-bound} 
In the setting of Proposition \ref{prop: quad-risk-meas-gauss-noise}, consider noisy \bcgreen{GMM} \mg{with parameters $(\alpha,\beta,\gamma) \in \mathcal{S}_q$}.
The entropic value at risk at a given confidence level $\bc{\zeta\in(0,1)}$ of noisy \bcgreen{GMM} iterations at stationarity satisfies
\begin{align}\label{eq: Approx_EVAR_Quad}
\text{EV@R}_{1-\zeta}\bcred{[f(x_\infty)-f(x_*)]}\leq \bar{E}^{q}_{1-\zeta}(\alpha,\beta,\gamma):=
\frac{\sigma^2}{2\theta_0\bar{u}}\left[-d\log\left(1-\theta_0\right)+2\log(1/\zeta) \right],
\end{align}
where $\theta_0=\frac{\log(1/\zeta)}{d}\left[\sqrt{1+2d/\log(1/\zeta)}-1 \right]\bcred{<1}$, $\bar{u}=\bcred{\min}_{i\in \{1,..,d\}}\{u_i\}$ and \mg{$x_\infty$ is a random variable that is distributed according to the stationary distribution of the iterates $x_k$.}
\end{theorem}
\begin{example} (EV@R bounds for GD)
Theorem \ref{thm: quad-evar-bound} implies $EV@R$ bounds for noisy GD, if we set $\beta=0=\gamma$. In this case, 
we have $u_i=\frac{1-(1-\alpha \lambda_i(Q))^2}{\lambda_i(Q) \alpha^2}$ and $\rho(A_Q)=\rho_{GD}(\alpha):= \max\{|1-\alpha \mu|,|1-\alpha L|\}$. The feasible set, on the other hand, becomes
$$
\mathcal{F}_\theta =\left\{\alpha\bcgreen{>0} ~\vline~  (1-\alpha\lambda_i(Q))^2 < 1-\frac{\theta \alpha^2 \lambda_i(Q)}{2} \mbox{ for every } i=1,2,\dots,d.\right\}
$$
Notice that as $\theta \rightarrow 0$ this yields 
$$
\bcgreen{\lim_{\theta\rightarrow 0}\mathcal{F}_\theta}= \left\{\alpha ~\vline~ (1-\alpha \lambda_i(Q))^2< 1 \right\}= \left\{ \alpha ~|~ \max\{|1-\alpha \mu|, |1-\alpha L|\}<1 \right\}= \mathcal{S}_q,
$$
as expected from Lemma \ref{lem: feas-set-stab-set}. Moreover, for $\alpha \in (0,2/L)$, 
we can compute $\bar{u}$ as 
\mg{
$$
\bar{u}=\min_{i\in\{1,...,d\}}\{u_i\}\geq \frac{1}{L\alpha^2}\min_{i\in\{1,..,d\}}\{1-(1-\alpha \lambda_i(Q))^2\}=\frac{1-\rho_{GD}^2(\alpha)}{\alpha^2 L}.
$$
}
Hence, for $\alpha \in (0,2/L)$, the inequality \eqref{eq: Approx_EVAR_Quad} becomes
\begin{eqnarray}
EV@R_{1-\zeta}[f(x_\infty)-f(x_*)]&\leq& \bar{E}^q_{1-\zeta}(\alpha,0,0)= \frac{\sigma^2 \alpha^2 L}{2\theta_0(1-\rho_{GD}(\alpha)^2)}\left[-d\log(1-\theta_0)+2\log(1/\zeta) \right],
\label{ineq-evar-gd}
\end{eqnarray}
which relates the convergence rate to our EV@R bound.
\end{example}

As a direct consequence of the inequalities \eqref{ineq: tail-prob-bound} and \eqref{eq: Approx_EVAR_Quad}, and Theorem  \ref{thm: quad-evar-bound}, we have also the following corollary.
\bcred{
\begin{corollary}
In the setting of Theorem \ref{thm: quad-evar-bound}, we have also
\begin{equation}
    \mathbb{P}\big\{f(x_\infty)-f(x_*)\geq \bar{E}^{q}_{1-\zeta}(\alpha,\beta,\gamma)\big\}\leq\zeta,
\end{equation}
for any confidence level $\zeta\in \bc{(0,1)}$.
\end{corollary}
}

\mg{In order to illustrate that our EV@R bound $\bar{E}^q_{1-\zeta}(\alpha,\beta,\gamma)$ provided in Theorem \ref{thm: quad-evar-bound} is useful, we compare our EV@R bound to the actual $EV@R_{1-\zeta}($ $f(x_\infty) - f(x_*))$ value (which is defined as the optimal value of the minimization problem \eqref{def: EVaR}) on the right panel of Figure \ref{fig: rate_vs_risk_meas_and_evar_quad}. Here, we consider
the objective $f(x_{(1)},x_{(2)})=x_{(1)}^2+0.1x_{(2)}^2$ at $\zeta=0.95$ and $\zeta=0.5$ confidence levels. We grid the parameter space $(\alpha,\beta,\gamma)\in \mathcal{S}_q$ and compute the $EV@R_{1-\zeta}\bcred{[f(x_\infty)-f(x_*)]}$ and the convergence rate $\rho(A_Q)$ corresponding to each choice of parameters\footnote{\bcgreen{We find the $EV@R$ at any given parameters $(\alpha,\beta,\gamma)$ by using grid search over $\theta \in (0,2\min_{i}u_i)$.}} on the right panel of Figure \ref{fig: rate_vs_risk_meas_and_evar_quad}, we
plot the best EV@R value achievable at a given rate $\rho(A_Q)$ with straight lines, we also plot the upper bound $\bar{E}_{1-\zeta}^q$ for the EV@R value in dashed lines. We observe that the upper bound $\bar{E}^q_{1-\zeta}(\alpha,\beta,\gamma)$ is a good approximation in general to the actual EV@R value and the approximation quality is better when the convergence rate is closer to 1. For noisy GD, we follow the same procedure where we grid over the stepsize (and keep $\beta=\gamma = 0$). These results display a trade-off between the convergence rate $\rho(A_Q)$ and $EV@R_{1-\zeta}\bcred{[f(x_\infty)-f(x_*)]}$ as well. This raises the natural question of what is the best $EV@R$ achievable at a given rate level and \bcgreen{this is achieved} by which choice of parameters. The answer will naturally depend on the choice of the confidence interval $\zeta$ and how close we want to be to the fastest convergence rate. \bcgreen{Introducing $\rho(\alpha,\beta,\gamma):=\rho(A_Q)$ to emphasize the dependency of the spectral radius to parameters, }the distance to the
optimum of \bcgreen{GMM} iterates decay at rate $\rho(\alpha,\beta,\gamma)$ \bcgreen{in the absence of noise} (Remark \ref{remark-tmm-deter-rate}); however the linear convergence rate in suboptimality will be $\rho^2(\alpha,\beta,\gamma)$ as suboptimality is proportional to the squared distance to the optimum for strongly convex smooth functions. To our knowledge, the choice of parameters that optimizes the rate (minimizes $\rho(A_Q)$) for quadratics is not explicitly known for \bcgreen{GMM}; however it is explicitly known for AGD. More specifically, the fastest rate for AGD is $\rho^2_{q,*} =( 1 - \frac{2}{\sqrt{3\kappa+1}})^2$ where $\rho_{q,*}$ is the smallest value of the spectral radius $\rho(A_Q)$ that can be attained by the AGD algorithm \cite{lessard2016analysis}. Given $\zeta\in (0,1)$, the following optimization problem looks for the best EV@R bound that can be attained if we would allow \bcgreen{GMM} to be slower than the fastest AGD rate by a certain percent:}\looseness=-1
\begin{subequations}\label{def: risk_averse_stmm_quad}
\begin{align}
    (\alpha_{q},\beta_{q},\gamma_{q}):=\underset{(\alpha,\beta,\gamma)\in\mathcal{S}_q}{\text{argmin}}& \quad \bar{E}^q_{1-\zeta}(\alpha,\beta,\gamma)\\ 
    \text{s.t. }& \frac{\bcred{\rho^2}(\alpha,\beta,\gamma)}{\mg{\rho_{q,*}^2}} \mg{\leq} (1+\varepsilon) ,
\end{align}
\end{subequations}
\mg{
Here, the parameter $\varepsilon$ can be viewed as a trade-off parameter that captures the \bcgreen{percentage} of rate we want to give away from the benchmark rate to achieve better entropic risk and tail bounds. Since we have an explicit characterization of the set $\bcgreen{\mathcal{S}}_q$, $\rho(\alpha,\beta,\gamma)$ and $\bar{E}^q_{1-\zeta}(\alpha,\beta,\gamma)$ as a function of the parameters, this \bcgreen{three-dimensional} optimization problem can be solved by a simple grid search on the parameters. In Section \ref{sec: num-exp}, we will provide numerical experiments where we will illustrate how the parameter $\varepsilon$ leads to systematic trade-offs between the convergence rate and the entropic risk \bcgreen{as well as the} tail behavior of the iterates.}
\section{Main results for strongly convex \bcgreen{smooth} objectives}\label{sec: strongly-conv-obj}

This section extends our entropic risk and EV@R results from quadratic objectives to (more general) strongly convex objectives $f\in \mathcal{C}_\mu^{L}(\mathbb{R}^d)$. \mg{The convergence rate analysis of \emph{deterministic} \bcgreen{GMM} (without gradient noise) can be carried out based on 
the Lyapunov function}
\begin{equation}\label{def: Lyapunov}
{\mV}_P(\bcred{z}_k):=(\bcred{z}_k-\bcred{z}_*)^\top P (\bcred{z}_{k}-\bcred{z}_*)+f(x_k)-f(x_*),
\end{equation}
where the matrix $P$ is symmetric positive semi-definite (PSD) with $P= \tilde{P}\otimes I_d$ for a $3\times 3$ matrix $\tilde{P}$. More specifically, \mg{in \cite[Lemma 5]{hu2017dissipativity}, the authors show that for any given $\rho^2>0$, if the matrix $\tilde{P}$ satisfies a particular $3\times 3$ matrix inequality (that depends on $\rho$) then \bcgreen{GMM} admits linear convergence with rate $\rho^2$. However, one limitation of existing theory for strongly convex minimization is that both the rate $\rho^2$ and the matrix $\tilde{P}$ are not explicitly available for general choice of parameters. By using an alternative Lyapunov function, convergence rates can be obtained for (deterministic) robust momentum methods which correspond to a specific parametrization of GMM (where all the parameters are expressed as a particular function of the convergence rate $\rho$) (see \cite[Theorem 21 and Table 1]{vanscoy2021speedrobustness}), however such results do not apply to GMM with more general choice of parameters. Furthermore, to our knowledge, existing results do not provide non-asymptotic \bcgreen{suboptimality} bounds for GMM at a given iterate $k$ \bcgreen{in the presence of gradient noise}.\looseness=-1

In the following result, we reparametrize the \bcgreen{GMM} parameters in two free variables $(\vartheta,\psi)$ and obtain new (non-asymptotic) finite-time convergence rate results for \bcgreen{GMM} \bcgreen{in this parametrization}. The proof is based on constructing the Lyapunov matrix $P$ in \eqref{def: Lyapunov} explicitly where we take $P=\tilde{P}\otimes I_d$ with
\begin{align}\label{def: P_Matrix}
\tilde{P}=\begin{bmatrix} \sqrt{\frac{\vartheta}{2\alpha}}\\ -\sqrt{\frac{\vartheta}{2\alpha}}+\sqrt{\frac{\mu}{2}} \end{bmatrix}\begin{bmatrix}\sqrt{\frac{\vartheta}{2\alpha}}\\-\sqrt{\frac{\vartheta}{2\alpha}}+\sqrt{\frac{\mu}{2}} \end{bmatrix}^\top,
\end{align}
which generalizes the Lyapunov functions previously considered in \cite{hu2017dissipativity,aybat2019robust} for AGD. To our knowledge, these are the first non-asymptotic convergence results for \bcgreen{GMM} in the strongly convex case subject to noise. Previous results applied to the quadratics or gave only local convergence rates \bcgreen{ for non-quadratic strongly convex case (i.e. shows convergence only if initialization is close enough to the optimum)}
\cite{gitman2019understanding}. Furthermore, our choice of parameters has the advantage that it eliminates the need for solving the associated matrix inequality numerically to estimate the rate\bcgreen{, and provides  us a uniform analysis for AGD and HB methods}.
}

\begin{theorem}\label{thm: TMM-MI-solution}
We introduce the sets
\begin{align*} 
\mathcal{S}_{0}&= \{(\vartheta,\psi) ~\mid~ \vartheta=1=\psi\},\;\;\;\; \mathcal{S}_{+}=\left\{(\vartheta,\psi) ~\mid~ \psi>1\, \&\, 1<\vartheta \leq 2-\frac{1}{\psi} \right\},\\
\mathcal{S}_{-}&= \left\{(\vartheta,\psi) ~\mid~ 0\leq\psi<1\,\&\, \max\left\{2-\frac{1}{\psi},\frac{1}{1+\kappa(1-\psi)}\right\}\leq \vartheta < 1 \right\},\\
\mathcal{S}_1&= \left\{ (\vartheta,\psi) \big|  \psi\neq 1,  \left[1-\sqrt{\frac{(1-\vartheta)\vartheta}{\kappa(1-\psi)}}\right]\left[1-\frac{(1-\vartheta)(\mu\psi^2-L(1-\psi)^2)}{L(1-\psi)\vartheta}\right] \leq \left[1-\frac{(1-\vartheta)\psi}{\kappa(1-\psi)}\right]^2\right\},
\end{align*}
with the convention that $\max\{2-\frac{1}{0},\frac{1}{1+\kappa}\}=\frac{1}{1+\kappa}$.
\bcred{Under Assumption \ref{Assump: Noise}, noisy \bcgreen{GMM} with the parameters, }
\begin{subequations}\label{def: stmm-params-str-cnvx}
\begin{align}
&\alpha_{\vartheta,\psi} \in \begin{cases}
\{\frac{1-\vartheta}{L (1-\psi)}\}, & \text{if } (\vartheta,\psi)\in \mathcal{S}_c := (\mathcal{S}_{-}\cup \mathcal{S}_{+})\cap \bcred{\mathcal{S}_1},\\
(0,\frac{1}{L}], & \text{if } (\vartheta,\psi) \in \mathcal{S}_0,
\end{cases}\\
&\beta_{\vartheta,\psi}=\frac{1-\sqrt{\vartheta\alpha_{\vartheta,\psi}\mu}}{1-\alpha_{\vartheta,\psi}\psi\mu}\left[1-\sqrt{\frac{\alpha_{\vartheta,\psi}\mu}{\vartheta}} \right],\\
& \gamma_{\vartheta,\psi}=\psi \beta_{\vartheta,\psi}.
\end{align}
\end{subequations}
\bcred{satisfies the inequality } 
\bcred{
\begin{equation}\label{ineq: exp-subopt-str-cnvx-gaus}
\mathbb{E}[f(x_{k})]-f(x_*) \leq \rho^{2k}_{\vartheta,\psi}\mathcal{V}_P(\bcred{z}_0)+\left(\frac{\sqrt{\alpha_{\vartheta,\psi}}(L\alpha_{\vartheta,\psi}+\vartheta)}{2\sqrt{\vartheta\mu}}\right) d\sigma^{2},
\end{equation}
where $\mathcal{V}_P$ is the Lyapunov function given by \ref{def: Lyapunov}, \bc{$P=\tilde{P}\otimes I_d$ with $\tilde{P}$ as in \eqref{def: P_Matrix}} and the convergence rate is}
$$\rho_{\vartheta,\psi}^2=1-\sqrt{\vartheta\alpha_{\vartheta,\psi}\mu}~ \mg{\in (0,1)}.$$
\end{theorem} 
\begin{remark}\label{remark-noise-in-L2} (More general noise) \mg{The proof of Theorem \ref{thm: TMM-MI-solution} uses only the unbiasedness of the noise, i.e.  $\mathbb{E}[w_{k+1}|\mathcal{F}_k] = 0$, and the boundedness of the conditional second moments $\mathbb{E}[\|w_{k+1}\|^2 | \mathcal{F}_k]$ where $\mathcal{F}_k$ is the natural filtration generated by the iterates $\{x_j\}_{j=0}^k$. Therefore, as long as the noise sequence $\{w_{k+1}\}_{k\geq 0}$ satisfies $\mathbb{E}[w_{k+1}|\mathcal{F}_k] = 0$ and $\mathbb{E}[\|w_{k+1}\|^2 | \mathcal{F}_k] \leq \sigma^2 d$ for some $\sigma^2>0$ and for every $k$, our results would still hold verbatim.}
\end{remark}
\begin{remark}\label{remark-deter-TMM} (AGD case) \mg{The case $(\vartheta,\psi) \in \mathcal{S}_0$ (when $\psi=\vartheta=1$) in Theorem \ref{thm: TMM-MI-solution} yields the AGD algorithm with momentum parameter $\beta=\frac{1-\sqrt{\alpha \mu}}{1+\sqrt{\alpha \mu}}$ which has been studied in the literature. In this case, our rate result recovers the existing rate results $\rho^2 = 1 - \sqrt{\alpha\mu}$ for the AGD \cite{hu2017dissipativity,aybat2019robust}. 

}
\end{remark} 
\begin{remark}\label{remark-hb} (HB case)
\bcgreen{
The set $\mathcal{S}_{-}$ given in Theorem \ref{thm: TMM-MI-solution} allows $\psi=0$ and it can be easily shown that $\{(\vartheta,0) | \vartheta\in [\frac{\kappa}{1+\kappa}, 1)\}$ is contained in the stable set $\mathcal{S}_c$. \bcgreen{T}herefore, Theorem \ref{thm: TMM-MI-solution} recovers the Polyak's heavy ball method (HB) with parameters $\alpha = \frac{1-\vartheta}{L}$ and $\beta=\left[1-\sqrt{\frac{\vartheta(1-\vartheta)}{\kappa}}\right]\left[ 1-\sqrt{\frac{(1-\vartheta)}{\kappa\vartheta}}\right]$, for $\vartheta \in [\frac{\kappa}{1+\kappa},1)$ and implies that deterministic HB can achieve the linear convergence rate $\rho^2= 1-\sqrt{\frac{\vartheta(1-\vartheta)}{\kappa}}$ on strongly convex \bcgreen{smooth} objectives. By setting $\vartheta=1-\alpha L$, this implies that for $\alpha \in (0, \frac{1}{L(\kappa+1)} ]$ and $\beta=[1-\sqrt{(1-\alpha L)\alpha\mu}][1-\sqrt{\frac{\alpha \mu}{1-\alpha L}}]$, \bcgreen{HB admits the} convergence rate $\rho^2=1-\sqrt{(1-\alpha L)\alpha\mu}$. Previously in \cite{ghadimi-heavy-ball}, it was shown that deterministic HB with parameters $\alpha\in (0,\frac{1}{L}]$, $\beta=\sqrt{(1-\alpha \mu)(1-\alpha L)}$ can achieve the rate $\rho^2_{HB}=1-\alpha\mu$ on strongly convex non-quadratic objectives. For $\alpha \in (0,\frac{1}{L(\kappa+1)}]$, \bcgreen{to our knowledge} convergence rate we prove for HB is faster than existing rate, $\rho^2_{HB}$, from the literature. Our rate for HB scales with $\sqrt{\alpha}$ similar to the rate of AGD. In \cite{liu2020improved}, it is shown that HB method subject to noise assumption discussed in Remark \ref{remark-noise-in-L2} satisfies $\mathbb{E}[f(x_k) - f(x_*)] \leq \mathcal{O}(r^k + \frac{L}{\mu} \alpha \sigma^2d)$ for $r=\max(1-\alpha\mu, \beta)$ provided that $k\geq k_0 := \lfloor \frac{\log 0.5}{\log (\beta)}\rfloor$. If the stepsize is small enough, or if the target accuracy $\varepsilon$ (in expected suboptimality) is small enough, then our result leads to a better iteration complexity. \looseness=-1
}
\end{remark}

\bcgreen{One natural question to ask is how large the other set $\mathcal{S}_c=(\mathcal{S}_{-}\cup \mathcal{S}_{+})\cap \bcred{\mathcal{S}_1}$ for which we can provide performance bounds is. In Lemma \ref{lem: proof-of-non-empty-set} of the Appendix, we show that the set $\mathcal{S}_c$ is non-empty and contains non-trivial parameter choices \bcgreen{in addition to the AGD and HB cases we discussed above}. In Figure \ref{fig: str-cnvx-evar-rate-on-stab-reg}, we also plot the range of parameters $\alpha_{\vartheta,\psi}$ and $\beta_{\vartheta,\psi}/\gamma_{\vartheta,\psi}$ on an example when $(\vartheta,\psi) \in \mathcal{S}_c$ \bcgreen{to illustrate our results}. 
}

\mg{Theorem \ref{thm: TMM-MI-solution} shows that noisy \bcgreen{GMM} converges to a neighborhood of the optimum. The suboptimality bound \eqref{ineq: exp-subopt-str-cnvx-gaus} is a sum of two terms: A ``bias term" that decays exponentially with rate $\rho^2_{\vartheta,\psi}$ and is about how fast the suboptimality at the initialization decays and a ``variance term" due to noise that characterizes the asymptotic suboptimality (and the radius of the neighborhood that the iterates belong to in expectation asymptotically). There are classic bias variance trade-offs in the bound we obtain in the sense that when the rate $\rho^2_{\vartheta,\psi}$ gets smaller, the bias term will get smaller but the variance term will get larger.}
\mg{The inequality \eqref{ineq: exp-subopt-str-cnvx-gaus} provides performance bounds on the expected suboptimality but Theorem \ref{thm: TMM-MI-solution} lacks providing further information about the deviations from the expected suboptimality. In the next proposition, we obtain a characterization of the finite-horizon risk which captures information about the deviations from the expected suboptimality as well as the tail probabilities of the iterates. The proof is based on bounding the suboptimality $f(x_{k+1}) - f(x_*)$ at step $k+1$ in terms of a decay in the Lyapunov function $V_P(\bcred{z}_k)$ at the previous step $k$ and a random error term that depends on the noise $\bc{\veps_{k+1}}$ and the difference $x_{k}-x_{k-1}$ (see Lemma \ref{lem: func-contr-prop} from the \bc{A}ppendix) and then integrating the noise along the \bcgreen{GMM} iterations using the properties of the moment generating function of Gaussian distributions. We can then show that the finite-horizon entropic risk $\mr_{k,\sigma^2}(\theta)$ decays linearly with rate $\bar{\rho}_{\vartheta,\psi}^2$ to a limiting value (in the sense that $\mr_{\sigma^2}(\theta) = \limsup_{k\to\infty}\mr_{k,\sigma^2}(\theta)$ exists for which we can provide an upper bound). For any $\theta>0$, due to the random error term that needs to be exponentially integrated, the convergence rate $\bar{\rho}_{\vartheta,\psi}^2$ we obtain for the entropic risk is slower than the convergence rate ${\rho}_{\vartheta,\psi}^2$ we have previously showed for the expected suboptimality in Theorem \ref{thm: TMM-MI-solution}, i.e. we have ${\rho}_{\vartheta,\psi}^2< \bar{\rho}_{\vartheta,\psi}^2<1$. However, in the limit as $\theta\to 0$, $\bar{\rho}_{\vartheta,\psi} \to {\rho}_{\vartheta,\psi}$; this is because as $\theta\to 0$, $\mr_{k,\sigma^2}(\theta)\to \mathbb{E}[f(x_k) - f(x_*)]$ (see \eqref{def: risk_meas_infty}) in which case the entropic risk captures the expected suboptimality previously studied in Theorem \ref{thm: TMM-MI-solution}.
}

\begin{proposition}\label{prop: risk-meas-bound-gaussian} Consider the noisy \bcgreen{GMM} iterations for $f \in \Sml$, where the parameters $(\alpha,\beta,\gamma)$ (as a function of $(\vartheta,\psi)$) are chosen as in 
Theorem \ref{thm: TMM-MI-solution}. Assume the gradient noise satisfies Assumption \ref{Assump: Noise}.
Then, the finite-horizon entropic risk $\mr_{k,\sigma^2}(\theta)$ is finite for 
\begin{align}\label{cond: cond-on-theta-gaus}
\theta < \theta^{g}_{u}:= \frac{2\sqrt{\vartheta\mu}}{\alpha_{\vartheta,\psi}(8\mathtt{v}_{\vartheta,\psi}\sqrt{\alpha_{\vartheta,\psi}}+\sqrt{\vartheta\mu}(\vartheta+\alpha_{\vartheta,\psi}L))},
\end{align}
with
\begin{equation*}\small\mtv_{\vartheta,\psi}=\frac{2L^2}{\mu}\left(2(\beta_{\vartheta,\psi}-\gamma_{\vartheta,\psi})^2+(1-\alpha_{\vartheta,\psi} L)^2(1+2\gamma_{\vartheta,\psi}+2\gamma_{\vartheta,\psi}^2)\right)+ \frac{\vartheta}{2\alpha_{\vartheta,\psi}}(1-\sqrt{\vartheta \alpha_{\vartheta,\psi}\mu}),
\end{equation*}
and admits the bound
\bcred{
\begin{equation}\label{ineq: fin-risk-meas-bound-gaus}
\mr_{k,\sigma^2}(\theta) < \frac{\sigma^2 d \alpha_{\vartheta,\psi} \big(\vartheta + \alpha_{\vartheta,\psi}L\big)}{(1-\bar{\rho}^2_{\vartheta,\psi})(2-\theta \alpha_{\vartheta,\psi} \big( \vartheta+\alpha_{\vartheta,\psi}L\big))} + \bar{\rho}^{2k}_{\vartheta,\psi}2\mV_P(\bcred{z}_0),
\end{equation}
where $\bar{\rho}_{\vartheta,\psi}^2<1$ is given by the formula}
\bc{
\begin{multline}\label{def: bar-rho}
\bar{\rho}^{2}_{\vartheta,\psi}=\frac{1}{2} \left(1-\sqrt{\vartheta\alpha_{\vartheta,\psi}\mu}+\frac{\theta 4\alpha^2_{\vartheta,\psi}\mtv_{\vartheta,\psi}}{2-\theta \alpha_{\vartheta,\psi}(\vartheta+\alpha_{\vartheta,\psi}L)}\right)\\+\frac{1}{2}\sqrt{\Big( 1-\sqrt{\vartheta\alpha_{\vartheta,\psi}\mu}+ \frac{\theta 4\alpha^2_{\vartheta,\psi}\mtv_{\vartheta,\psi}}{2-\theta \alpha_{\vartheta,\psi}(\vartheta+\alpha_{\vartheta,\psi}L)}\Big)^2+\frac{16 \theta \alpha_{\vartheta,\psi}^2\mtv_{\vartheta,\psi}}{2-\theta \alpha_{\vartheta,\psi}(\vartheta+\alpha_{\vartheta,\psi}L)}}.
\end{multline}
\mg{Consequently, taking limit superior of both sides of \eqref{ineq: fin-risk-meas-bound-gaus}, we obtain
$$\mr_{\sigma^2}(\theta) \leq \frac{\sigma^2 d \alpha_{\vartheta,\psi} \big(\vartheta + \alpha_{\vartheta,\psi}L\big)}{(1-\bar{\rho}^2_{\vartheta,\psi})(2-\theta \alpha_{\vartheta,\psi} \big( \vartheta+\alpha_{\vartheta,\psi}L\big))}.$$
}
}
\end{proposition}
\mg{In the next result, by plugging the entropic risk bound into the EV@R formula \eqref{def: EVaR} and approximating the solution of the optimization problem over $\theta$ in \eqref{def: EVaR}, we obtain an upper bound for the EV@R of the suboptimality $f(x_k) - f(x_*)$} for every $k$.
\begin{theorem}\label{thm: Evar-TMM-str-cnvx-bound}
Consider the noisy \bcgreen{GMM} to minimize the objective $f \in \Sml$ 
under the setting of Proposition \ref{prop: risk-meas-bound-gaussian}. Let $\varphi \in (0,1)$ be a fixed constant and let the confidence level $\zeta\in(0,1)$ be given. Set $\theta_{\varphi}^g:=\varphi\theta_u^g$ and 
\begin{multline}\label{def: bbrho}
\bbrho_{\vartheta,\psi}:=\frac{1}{2} \left(1-\sqrt{\vartheta\alpha_{\vartheta,\psi}\mu}+\frac{\theta_{\varphi}^g 4\alpha^2_{\vartheta,\psi}\mtv_{\vartheta,\psi}}{2-\theta_{\varphi}^g \alpha_{\vartheta,\psi}(\vartheta+\alpha_{\vartheta,\psi}L)}\right)\\+\frac{1}{2}\sqrt{\Big( 1-\sqrt{\vartheta\alpha_{\vartheta,\psi}\mu}+ \frac{\theta_{\varphi}^g 4\alpha^2_{\vartheta,\psi}\mtv_{\vartheta,\psi}}{2-\theta_{\varphi}^g \alpha_{\vartheta,\psi}(\vartheta+\alpha_{\vartheta,\psi}L)}\Big)^2+\frac{16 \theta_{\varphi}^g \alpha_{\vartheta,\psi}^2\mtv_{\vartheta,\psi}}{2-\theta_{\varphi}^g \alpha_{\vartheta,\psi}(\vartheta+\alpha_{\vartheta,\psi}L)}}.
\end{multline}
If confidence level satisfies the condition
\begin{equation}\label{cond: conf-lev-gauss}
\log(\frac{1}{\zeta}) \leq \frac{d}{2(1-\bbrho_{\vartheta,\psi})}\left( \frac{\theta_\varphi^g \alpha_{\vartheta,\psi}(\vartheta+\alpha_{\vartheta,\psi}L)}{2-\theta_{\varphi}^g\alpha_{\vartheta,\psi}(\vartheta+\alpha_{\vartheta,\psi}L)} \right)^2,
\end{equation}
then, the EV@R of the $f(x_k)-f(x_*)$ at iteration $k\geq1$ admits the bound
\begin{equation}\label{ineq: fin-EVAR-bound-gauss-1}
EV@R_{1-\zeta}[f(x_k)-f(x_*)]\bcred{\leq}\frac{\sigma^2 \alpha_{\vartheta,\psi}(\vartheta+\alpha_{\vartheta,\psi}L)}{2}\left( \sqrt{\frac{d}{1-\bbrho_{\vartheta,\psi}}}+ \sqrt{2 \log(1/\zeta)}\right)^2 + (\bbrho_{\vartheta,\psi})^k 2\mV_P(\bcred{z}_0).
\end{equation}
In the case \eqref{cond: conf-lev-gauss} does not hold, we have 
\begin{equation}\label{ineq: fin-EVAR-bound-gauss-2}
    EV@R_{1-\zeta}[f(x_k)-f(x_*)]
    \bcred{\leq}\frac{\sigma^2 d \alpha_{\vartheta,\psi}(\vartheta+\alpha_{\vartheta,\psi}L)}{(1-\bbrho_{\vartheta,\psi})(2-\theta_{\varphi}^g \alpha_{\vartheta,\psi}(\vartheta+\alpha_{\vartheta,\psi}L))}+\frac{2\sigma^2}{\theta_\varphi^g}\log(1/\zeta)+ (\bbrho_{\vartheta,\psi})^k2\mV_P(\bcred{z}_0).
\end{equation}
Moreover, the EV@R of the suboptimality satisfies asymptotically 
\begin{multline}\label{ineq: EVAR-bound-gauss}
\mg{\limsup_{k\to\infty}}\;EV@R_{1-\zeta}[f(x_\mg{k})-f(x_*)]\\
\leq\bar{E}_{1-\zeta}:=\begin{cases}
\frac{\sigma^2 \alpha_{\vartheta,\psi}(\vartheta+\alpha_{\vartheta,\psi}L)}{2}\left( \sqrt{\frac{d}{1-\bbrho_{\vartheta,\psi}}}+ \sqrt{2 \log(1/\zeta)}\right)^2, & \text{ if \eqref{cond: conf-lev-gauss} holds.}\\
\frac{\sigma^2 d \alpha_{\vartheta,\psi}(\vartheta+\alpha_{\vartheta,\psi}L)}{(1-\bbrho_{\vartheta,\psi})(2-\theta_{\varphi}^g \alpha_{\vartheta,\psi}(\vartheta+\alpha_{\vartheta,\psi}L))}+\frac{2\sigma^2}{\theta_\varphi^g}\log(1/\zeta), & \text{ otherwise.}
\end{cases}
\end{multline}
\end{theorem}

\mg{
Let us denote $\bar{E}_{1-\zeta}$ by $\bar{E}_{1-\zeta}(\vartheta,\psi)$ for the bound provided in Theorem \ref{thm: Evar-TMM-str-cnvx-bound} to emphasize its dependency on the parameters $\vartheta$ and $\psi$. A consequence of \eqref{eq-tail-proba} and Theorem \ref{thm: Evar-TMM-str-cnvx-bound} is the following corollary which provides us information about the tail probabilities of suboptimality.
\begin{corollary}\label{cor: gaussian-tail-bound}
Consider noisy \bcgreen{GMM} under the setting of Proposition \ref{prop: risk-meas-bound-gaussian}. Let risk parameter $\theta$ satisfy $\theta<\theta_{u}^{g}$, then the tail probability of suboptimality admits the following bound, 
\begin{equation}\label{ineq: tail-bound-1}
\mathbb{P} \left\{ f(x_k)-f(x_*)\geq t_{\zeta}  \right\}< e^{\frac{\theta}{2\sigma^2}\big(\bar{\rho}_{\vartheta,\psi}^{2k} 2\mV_P(\bcred{z}_0)-t_{\zeta}\big)+\frac{\theta d\alpha_{\vartheta,\psi}(\vartheta+\alpha_{\vartheta,\psi} L)}{2(1-\bar{\rho}^2_{\vartheta,\psi})(2-\theta \alpha_{\vartheta,\psi}(\vartheta+\alpha_{\vartheta,\psi}L))}},
\end{equation}
where \bcgreen{$z_0=[x_0^\top,x_{-1}^{\top}]^\top$ is the initialization} and $\bar{\rho}^2_{\vartheta,\psi}$ is as given in Proposition \ref{prop: risk-meas-bound-gaussian}. Furthermore, for a given confidence level $\zeta \in (0,1)$ and a constant $\varphi\in(0,1)$, optimizing over the choice of $\theta$,
\begin{equation}\label{ineq: tail-bound-2}
    \mathbb{P}\{ f(x_k)-f(x_*) \geq \bar{E}_{1-\zeta}(\vartheta,\psi)+(\bbrho_{\vartheta,\psi})^k 2\mV_P(\bcred{z}_0)\}\leq \zeta,
\end{equation}
where $\bbrho_{\vartheta,\psi}$ as given in \eqref{def: bbrho}. 
\end{corollary}
}
\mg{
\begin{remark} (More general noise)
In Proposition \ref{prop: risk-meas-bound-gaussian} and Theorem \ref{thm: Evar-TMM-str-cnvx-bound}, we assumed the gradient noise $w_k$ to be additive, i.i.d. Gaussian with a covariance matrix $\sigma^2 I_d$. Our results would hold for Gaussian noise with an arbitrary covariance matrix as well. More specifically, if the gradient noise $w_k\sim \mathcal{N}(0,\Sigma_{w_k})$ is independent from the past iterates $\{x_j\}_{j=0}^{k-1}$ and satisfies $\Sigma_{w_k}\preceq \sigma^2 I_d$ \bcgreen{for every $k$}, then our proof techniques extend and the same results would hold. In Section \ref{sec: Subgaussian}, we will also obtain results for more general noise structure that is not necessarily additive allowing sub-Gaussian tails.
\end{remark}
}
\bcred{
The parameter choices studied in Theorem \ref{prop: risk-meas-bound-gaussian} and Proposition \ref{thm: Evar-TMM-str-cnvx-bound} captures the AGD method (when $\vartheta=1=\psi$), HB method (when $\psi=0$) and \bcgreen{GMM} but requires the momentum term $\beta>0$. Therefore, these results do not directly apply to noisy gradient descent where $\beta=\gamma=0$. Using similar techniques we used to derive Proposition \ref{prop: risk-meas-bound-gaussian},
in the following we obtain bounds on the  entropic risk of suboptimality for noisy GD.
\begin{proposition}\label{prop: gd-evar-bound-gauss}
Consider the noisy GD iterations $\{x_k\}_{k\geq 0}$ to minimize an objective $f\in \Sml$ with stepsize $\alpha \in (0,\frac{2}{\mu+L}]$ where the gradient noise obeys \bcgreen{Assumption} \ref{Assump: Noise}. If the risk parameter $\theta>0$ satisfies the condition 
\beq
\theta < \frac{2(1-\rho_{GD}(\alpha)^2)}{L\alpha^2},
\label{ineq: upper bound on theta}
\eeq
where $\rho_{GD}(\alpha)=\max\{|1-\alpha \mu|,|1-\alpha L|\}$, then the finite-horizon entropic risk satisfies,
\begin{equation}\label{ineq: gd-risk-meas-bound-gauss}
\mr_{k,\sigma^2}(\theta) < \frac{\sigma^2 d\alpha^2 L}{2(1-\bcgreen{\rho_{GD}}(\alpha)^2)-\theta\alpha^2L}+\bcgreen{\big(\bar{\rho}_{GD}(\alpha)\big)^{2k}}\frac{L\Vert x_0-x_*\Vert^2}{2},
\end{equation}
where $\bar{\rho}_{GD}^2:=\frac{\rho_{GD}(\alpha)}{1-\theta \alpha^2 (L/2)} \in [0,1)$. Similarly, the infinite-horizon entropic risk admits the following bound: 
\begin{equation}\label{ineq: gd-inf-risk-meas-bound-gauss}
\mr_{\sigma^2}(\theta) \leq \frac{\sigma^2 d\alpha^2 L}{2(1-\rho_{GD}(\alpha)^2)-\theta\alpha^2L}.
\end{equation}
\end{proposition}
}

\mg{For $\alpha\in (0,\frac{2}{L+\mu})$, the upper bound $\bar{\theta}:=\frac{2(1-\rho_{GD}(\alpha)^2)}{L\alpha^2}$ on $\theta$ in \eqref{ineq: upper bound on theta} is a decreasing function of $\alpha$ whereas the rate $\rho_{GD}(\alpha)$ is a decreasing function of $\alpha$. Therefore, for a given $\alpha_0\in (0,\frac{2}{L+\mu})$, if we are more risk-averse and consequently if we want to increase $\theta$ beyond \bcgreen{$\bar{\theta}_0=\frac{2(1-\rho_{GD}(\alpha_0)^2)}{L\alpha^2_0}$}, then in Proposition \ref{prop: gd-evar-bound-gauss}, our analysis applies only if we give up from the rate $\rho_{GD}(\alpha_0)$ by choosing a smaller stepsize than $\alpha_0$. This phenomenon is not a shortcoming of our analysis in the sense that for the case of quadratic objectives (which are also in the class $\Sml$) where we can do exact computations, we also see \bcgreen{a similar behaviour}.
}
\mg{In the next result, we obtain EV@R results for gradient descent based on the previous result and the formula \eqref{def: EVaR}.}
\bc{
\begin{theorem}\label{thm: gd-evar-bound-gaussian}
Consider the noisy GD under the setting of Proposition \ref{prop: gd-evar-bound-gauss}. Let confidence level $\zeta\in (0,1)$ be given. The entropic value at risk of noisy GD admits the bound, 
\begin{small}
\begin{equation}\label{ineq: gd-evar-bound-gaus}
EV@R_{1-\zeta}[f(x_k)-f(x_*)]
\leq \frac{\sigma^2 \alpha^2 L}{2(1-\rho_{GD}(\alpha)^2)}\left( \sqrt{2\log(1/\zeta)}+\sqrt{d} \right)^2 + \Big( \frac{\rho_{GD}(\alpha)^2}{1-\varphi_{GD} (1-\rho_{GD}(\alpha)^2)}\Big)^k\frac{L\Vert x_0-x_*\Vert^2}{2},
\end{equation}
\end{small}
where $\varphi_{GD}= \frac{\sqrt{2\log(1/\zeta)}}{\sqrt{2\log(1/\zeta)}+\sqrt{d}}$, and $\alpha\leq\frac{2}{\mu+L}$.
\end{theorem}
}

\begin{figure}[ht!]
    \centering
    \includegraphics[width=0.9\linewidth]{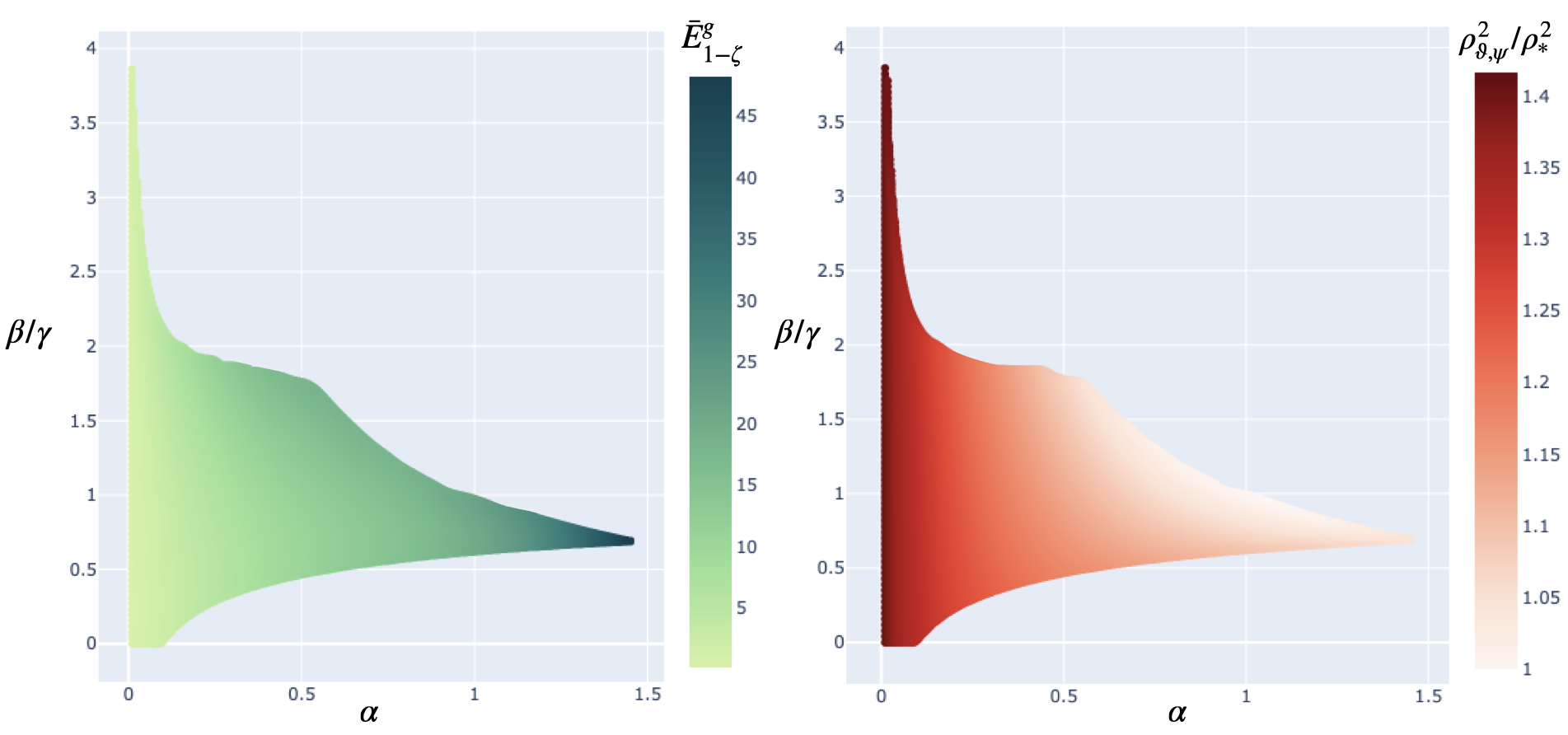}
    \caption{The EV@R bound $\bar{E}_{1-\zeta}$ (\textbf{Left}) and the rate $\rho_{\vartheta,\psi}^2/\rho_*^2$ (\textbf{Right}) on stable set $\mathcal{S}_c$ at $\varphi=0.99$ and $\zeta=0.99$ confidence level where $x\in\mathbb{R}^{10}$, $L=1$, $\mu=0.1$ for the objective function and the Gaussian noise is $\mathcal{N}(0,I_{10})$.}
    \label{fig: str-cnvx-evar-rate-on-stab-reg}
    \vspace{-0.2in}
\end{figure}

\mg{In Figure \ref{fig: str-cnvx-evar-rate-on-stab-reg}, on the left panel, we plot the EV@R bound $\bar{E}_{1-\zeta}$ as a function of the parameters allowed by Theorem \ref{thm: TMM-MI-solution}. More specifically, for visualization purposes, we consider the two dimensional set 
$$ 
\mathcal{A}:= \{ (\alpha_{\vartheta,\psi}, \beta_{\vartheta,\psi}/\gamma_{\vartheta,\psi}) : ({\vartheta,\psi}) \in \mathcal{S}_0 \cup \mathcal{S}_c
\}
$$
and compute and plot the EV@R for every point in $\mathcal{A}$ based on Theorem \ref{thm: Evar-TMM-str-cnvx-bound} for confidence level $\zeta = 0.99$. On the other hand, on the right panel of Figure \ref{fig: str-cnvx-evar-rate-on-stab-reg}, we plot the relative convergence rate
$\rho_{\vartheta,\psi}^2/\rho_*$  for every point in the set $\mathcal{A}$ where $\rho_*^2 = 1 - \sqrt{\alpha_{\vartheta,\psi} \mu}$ is the fastest convergence rate we can certify (corresponds to the choice $\vartheta=\psi=1$) in Theorem \ref{thm: TMM-MI-solution}. We take the noise variance $\sigma^2 = 1$, $d=10$, \bcred{and we set $\varphi=0.99$}. Similar to the quadratic case, our results show that faster convergence, i.e. smaller values of the rate, is associated with larger values of the upper bound on the risk $\bar{E}_{1-\zeta}$. We also observe that decreasing the stepsize in this region leads to a deterioration in the convergence rate while improving the risk bound $\bar{E}_{1-\zeta}$. These results highlight the trade-offs between the convergence rate and the entropic risk and tail probabilities at stationarity (as $k\to \infty)$.} 
\mg{Motivated by these results, we consider the following optimization problem to design the parameters of \bcgreen{GMM} algorithms to achieve systematic trade-offs between the relative convergence rate $\rho_{\vartheta,\psi}^2/\rho_*$ (which will be controlled by the parameter $\varepsilon$ below) and the entropic risk:}
\begin{subequations}
\label{def: risk-averse-stmm-param-design}
\begin{align}
    (\vartheta_\star,\psi_\star)=\underset{(\vartheta,\psi)\in \mg{ \mathcal{S}_c} \cup \mathcal{S}_0}{\text{argmin}} &\; \bar{E}_{1-\zeta}(\vartheta,\psi)\\
    \text{s.t.}\quad & \frac{\rho_{\vartheta,\psi}^2}{\rho_*^2} \leq (1+\varepsilon).
\end{align}
\end{subequations}
\mg{The parameter $\varepsilon$ is a trade-off parameter \bcgreen{as before}, it captures what percent of rate we want to give up from the fastest convergence rate $\rho_*^2$ to ensure better entropic risk bounds.}
\mg{Since the stepsize and momentum parameters are parametrized by $(\vartheta,\psi)$, once we find the optimal $(\vartheta_\star,\psi_\star)$; we can calculate the optimal 
parameters 
directly based on the relationship \eqref{def: stmm-params-str-cnvx}. We call the \bcgreen{GMM} algorithm with this choice of parameters \emph{entropic-risk averse \bcgreen{GMM}}. In the next section, we study the performance of noisy \bcgreen{GMM} algorithms numerically and illustrate our theoretical results. }
\section{Numerical Experiments} \label{sec: num-exp}


In the numerical experiments, we first consider the following quadratic optimization problem where the objective in \eqref{opt_problem} is given by
\beq
f(x)=\frac{1}{2}x^\top Q x+ b^\top x+ \frac{\lambda_{\scriptscriptstyle\mbox{reg}}}{2} \Vert x\Vert^2,
\label{opt-pbm-quad}
\eeq
\mg{in dimension $d=10$} with $b=\frac{1}{\Vert \tilde{b}\Vert}\tilde{b}$ for $\tilde{b}=[1,...,1]\in\mathbb{R}^{10}$, \mg{$Q$ is a diagonal matrix with diagonals $Q_{ii}=i^2$} for $i\in\{1,..,10\}$, and $\lambda_{\scriptscriptstyle\mbox{reg}}=5$ is a regularization parameter. We find the parameters $(\alpha_q,\beta_q,\gamma_q)$ of the (entropic) risk-averse \bcgreen{GMM} (RA-\bcgreen{GMM}), by solving the optimization problem 
\eqref{def: risk_averse_stmm_quad} using grid search over $\mathcal{S}_q$ where we set the trade-off parameter as $\varepsilon=0.25$ at a confidence level $\zeta=0.95$. We take gradient noise according to Assumption \ref{Assump: Noise} with $\sigma^2=1$. For the (entropic) risk-averse AGD (RA-AGD), we consider the problem \eqref{def: risk_averse_stmm_quad} but also enforce the additional constraint \mg{$\beta_{\vartheta,\psi} = \gamma_{\vartheta,\psi}$} 
using the same $\varepsilon$ and $\zeta$ values. For comparison purposes, we also consider GD with standard parameters $(\alpha=\frac{1}{L}, \beta=\gamma=0)$ used in deterministic optimization, as well as AGD with standard parameters $(\alpha=1/L,\beta=\gamma=\frac{\sqrt{L}-\sqrt{\mu}}{\sqrt{L}+\sqrt{\mu}})$ used in deterministic optimization \cite{hu2017dissipativity}. 

In Figure \ref{fig: quad_subopt_and_hist}, we compare RA-\bcgreen{GMM}, RA-AGD, AGD and GD algorithms. 
\begin{figure}
    \centering
    \includegraphics[width=\linewidth]{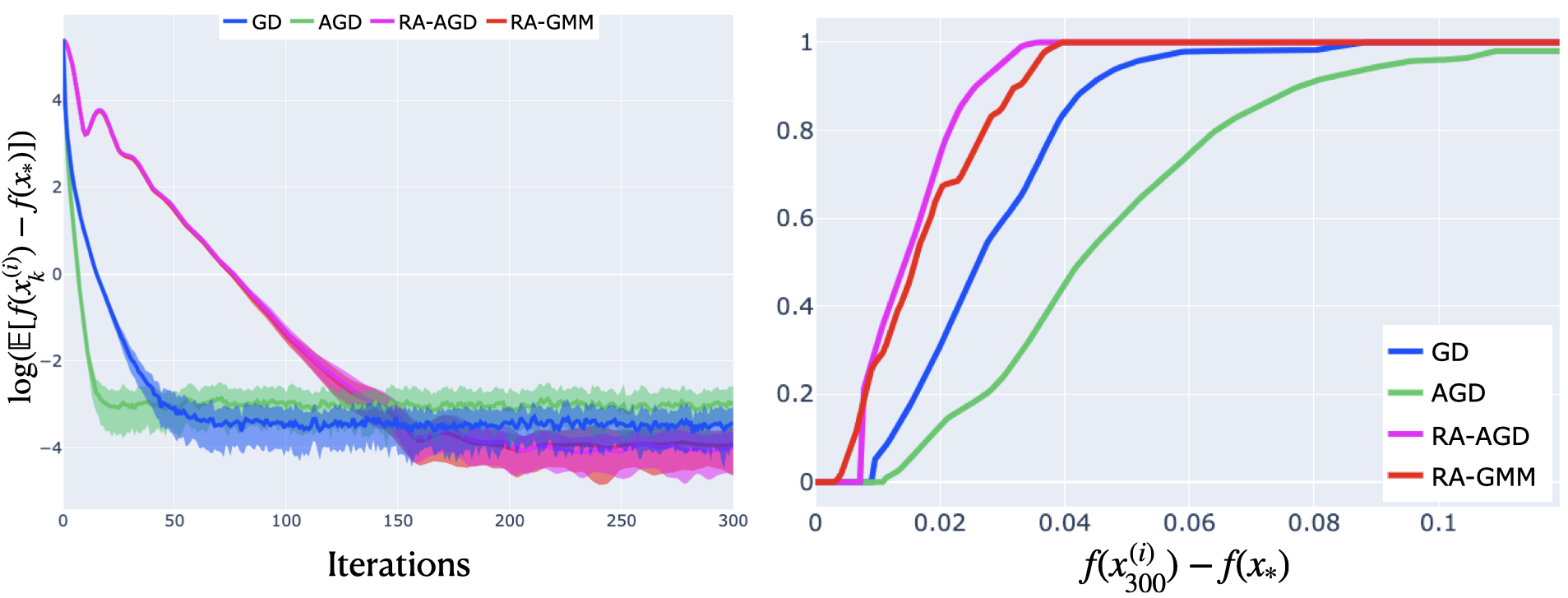}
    \caption{\textbf{(Left panel)} The expected suboptimality versus iterations for GD, AGD, RA-AGD and RA-\bcgreen{GMM}. \textbf{(Right panel)} The cumulative distribution of the suboptimality of the last iterates for GD, AGD, RA-AGD and RA-\bcgreen{GMM} after $k=300$ iterations.}
    \label{fig: quad_subopt_and_hist}
    \vspace{-0.2in}
\end{figure}
We initialized all the algorithms at $x_0=[1,...,1]\in\mathbb{R}^{10}$. We generated 50 sample paths $\{f(x_{0}^{(i)}),..., f(x_{300}^{(i)})\}_{i=1}^{50}$ by running the algorithms for 300 iterations, then computed the empirical mean and standard deviation of the \mg{suboptimality $f(x_k)-f(x_*)$ based on the samples $f(x_k^{(i)})-f(x_*)$ at every iteration $k\leq 300$}. The left panel of Figure \ref{fig: quad_subopt_and_hist} displays the average suboptimality, $(\bar{f}_1,...,\bar{f}_{300})$ where $\bar{f}_k:=\frac{1}{50}\sum_{i=1}^{50}f(x_k^{(i)})-f(x_*)$. Light colored area around the plots illustrate the standard deviation of the mean suboptimality.\footnote{More specifically, we highlight the region between $(\bar{f}_{0}\pm \sigma^{f}_0,...,\bar{f}_{300}\pm \sigma^f_{300})$ where $\sigma^f_k:=\big(\frac{1}{50}\sum_{i=1}^{50} |f(x_k^{(i)})-f(x_*)|^2\big)^{1/2}$.}  RA-AGD and RA-\bcgreen{GMM} performed similarly, where we see an overlap on their plots. These results illustrate that RA-AGD and RA-\bcgreen{GMM} can trade-off convergence rate with the entropic risk; both algorithms converge slower than GD and AGD; but achieve better asymptotic mean suboptimality. We also plotted the cumulative distribution of $\{f(x_{300}^{(i)})-f(x_*)\}_{i=1}^{50}$ on the right panel of Figure \ref{fig: quad_subopt_and_hist}. 
We can observe from this figure that both of the risk-averse algorithms (RA-\bcgreen{GMM} and RA-AGD) have stochastic dominance over \bcred{GD} and \bcred{AGD}; that is, for any given $t>0$, the tail probability $\mathbb{P}\{f(x^{(i)}_{300})-f(x_*) \geq t\}$ of the samples $\{f(x_{300}^{(i)})\}_{i=1}^{50}$ generated by risk averse algorithms is \mg{smaller} than the the same probability of the ones generated by standard algorithms. This is expected as RA-AGD and RA-\bcgreen{GMM} minimizes the entropic risk subject to rate constraints and this leads to better control of the tail behavior of suboptimality. We also consider the \emph{empirical finite-horizon entorpic risk}, \beq\tilde{r}_{k,\sigma^2}(\theta):= \frac{2\sigma^2}{\theta}\log \frac{1}{50}\sum_{i=1}^{50}e^{\frac{\theta}{2\sigma^2}(f(x_k^{(i)})-f(x_*))},
\label{eq-empirical-risk}
\eeq 
which approximates the finite-horizon entropic risk $r_{k,\sigma^2}(\theta)$, of each of the algorithms. In Figure \ref{fig: quad_risk_meas}, we illustrate the the convergence of empirical finite-horizon risk $\tilde{r}_{k,\sigma^2}(\theta)$ for $\theta=5$. According to \bc{Proposition \ref{prop: quad_qisk_meas_convergence}}, finite-horizon risk converges to the infinite horizon risk with rate $\rho(A_Q)$. As expected, GD and AGD converges to stationarity faster (as their $\rho(A_Q)$ is smaller) than RA-AGD and RA-\bcgreen{GMM}; however RA-AGD and RA-\bcgreen{GMM} have better asymptotic entropic risk as intended. 
\begin{figure}[ht!]
    \centering
    \includegraphics[width=0.5\linewidth,height=0.4\linewidth]{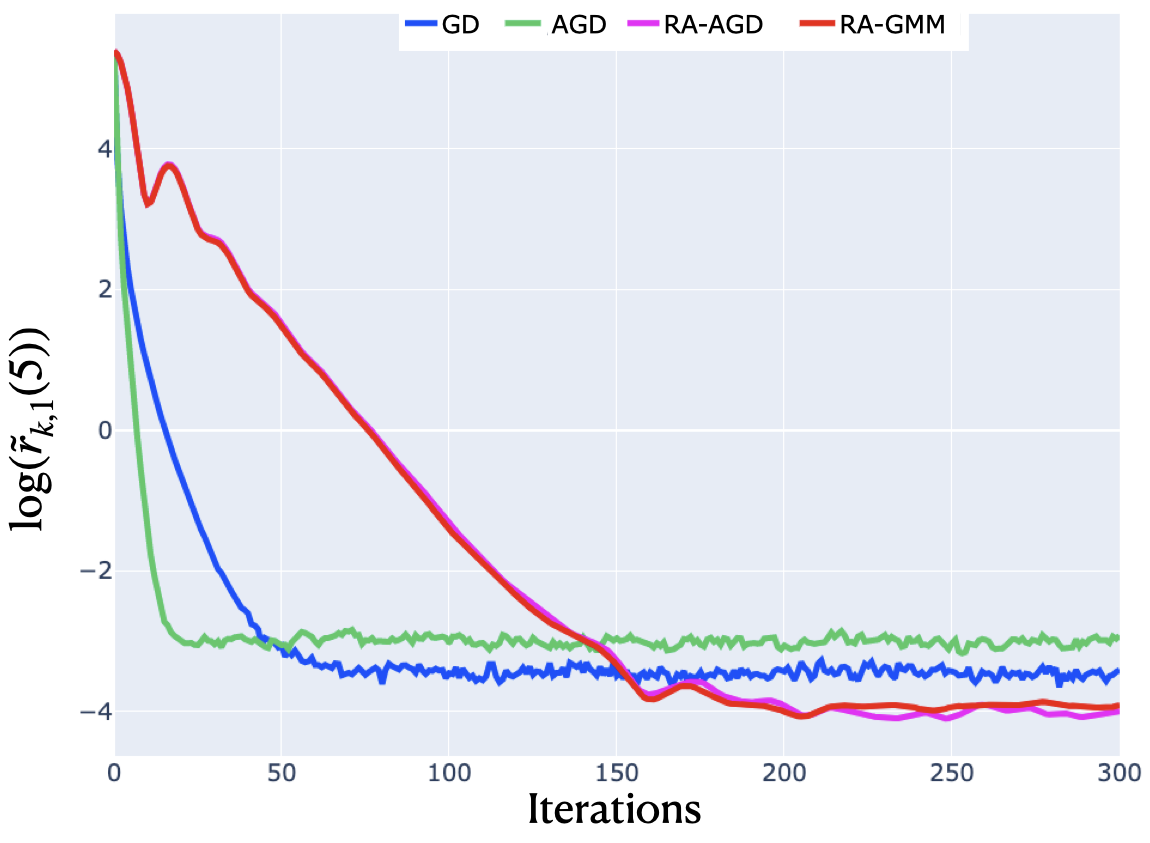}
    \caption{The empirical finite-horizon risk measure, $\tilde{r}_{k,1}(5)$ versus iterations on the quadratic optimization problem \eqref{opt-pbm-quad}.}
    \label{fig: quad_risk_meas}
    \vspace{-0.2in}
\end{figure}

Next, we consider the regularized binary logistic regression problem for which the objective function $f$ in \eqref{opt_problem} is 
$$
f(x)=\frac{1}{N}\sum_{i=1}^{N} \log(1+\exp\{-y_i (X_i^\top x) \})+\frac{\lambda_{\scriptscriptstyle\mbox{reg}}}{2}\Vert x\Vert^2,
$$
where $X_i\in\mathbb{R}^{100}$ is the feature vector and $y_i\in \{-1,1\}$ is the label of $i$-th sample. We set the regularization parameter $\lambda_{\scriptscriptstyle\mbox{reg}}=1$, the number of features $d=100$, and the sample size $N=1000$. We generated the synthetic data by sampling a random vector $\mathbf{w}\in\mathbb{R}^{100}$, \bc{where its each component, $\mathbf{w}_i$, is drawn from } $\mathcal{N}(0,25)$. \bc{We also generated the vectors $X_i\in \mathbb{R}^{100}$ by sampling its each component from $\mathcal{N}(0,25)$ for every $i=1,2,\dots,N$. 
We created the labels by $y_i=1 $ if $X_i^\top \mathbf{w}\geq 0$ and $-1$ otherwise. For the logistic regression, we calculated the parameters of the entropic risk-averse \bcgreen{GMM} \bc{(RA-\bcgreen{GMM})} by solving the optimization problem \eqref{def: risk-averse-stmm-param-design} for \bcred{$\varphi=0.99$ (used to estimate the EV@R bound),} $\zeta=0.95$ and $\varepsilon=0.05$ on the sets 
using grid search for the parameters. \bc{For risk-averse AGD (RA-AGD), we choose $\vartheta=1=\psi$. Under this choice of parameters, the class of parameters $(\alpha_{\vartheta,\psi},\beta_{\vartheta,\psi},\gamma_{\vartheta,\psi})$ we can choose simplifies and can be represented as
$(\alpha,\frac{1-\sqrt{\alpha\mu}}{1-\sqrt{\alpha\mu}},\frac{1-\sqrt{\alpha \mu}}{1+\sqrt{\alpha\mu}})$ where $\alpha \in (0,1/L]$. Therefore, we can see that the EV@R bound can be written as a function of $\alpha$ only, i.e., $\bar{E}_{1-\zeta}=\bar{E}_{1-\zeta}(\alpha)$, and similarly the convergence rate can be written as $\rho_{AGD}^2(\alpha)=1-\sqrt{\alpha\mu}$. Therefore, the analogue of the problem \eqref{def: risk-averse-stmm-param-design} for RA-AGD becomes $\alpha \in (0,1/L]$ to solve $\min_{\alpha \in (0,1/L]}\bar{E}_{1-\zeta}(\alpha)$ subject to $\rho_{AGD}^2(\alpha)\leq (1+\varepsilon)\rho_*^2$ where we take the same $\zeta=0.95$ and $\varepsilon=0.05$ to compare with RA-\bcgreen{GMM}.} We use grid search over the parameter $\alpha$ to solve this problem. On the left panel of Figure \ref{fig: str_cnvx_subopt_and_hist}, we compared the mean suboptimality of RA-\bcgreen{GMM}, RA-AGD, AGD and GD over 50 sample paths for $k=600$ iterations.\footnote{For plotting the suboptimality, the optimal value $f(x_*)$ of the objective in logistic regression needs to be estimated, for this we run deterministic gradient descent on the logistic regression problem.} Similar to the quadratic example, we added standard deviations of the runs and started the iterations from $x_0=[1,...,1]^{\top}\in\mathbb{R}^{100}$ .} 
Our results in Figure \ref{fig: str_cnvx_subopt_and_hist} (on the left) shows that AGD is the fastest among the algorithms as expected; however, it yields the least accurate result asymptotically. GD is slower than AGD but with better asymptotic performance.
Both of the risk-averse algorithms converge slower than GD and AGD but they yield to better mean suboptimality subject to standard deviations of the suboptimality. This is expected from the formula \eqref{def: risk_meas_infty} as the entropic risk captures both the mean suboptimality and the standard deviation of the suboptimality.

RA-\bcgreen{GMM} does have a slight improvement over RA-AGD. On the right panel of Figure \ref{fig: str_cnvx_subopt_and_hist}, we illustrate the cumulative distribution of the suboptimality on the last iterate. 
Similar to what we have observed on quadratic problem, 
we can see that distributions corresponding to \bcred{RA-\bcgreen{GMM}} and \bcred{RA-AGD} have stochastic dominance over the distributions of GD and AGD. This illustrates that with RA-AGD and RA-\bcgreen{GMM} the probability that the suboptimality exceeds a threshold is smaller. 

Lastly in Figure \ref{fig: str_cnvx_risk_meas}, we plot the empirical finite-horizon entropic risk (see \eqref{fig: str_cnvx_subopt_and_hist}) for logistic regression, taking $\theta=5$. For AGD and \bcgreen{GMM}, from  \eqref{ineq: fin-risk-meas-bound-gaus} we expect the finite-horizon entropic risk to converge linearly with a rate proportional to the convergence rate $\bar{\rho}_{\vartheta,\psi}^2$ to an interval; for GD, based on\eqref{ineq: fin-risk-meas-bound-gaus}, we also expect the same behavior with convergence rate $\rho_{GD}^2(\alpha)$. Figure \ref{fig: str_cnvx_risk_meas} illustrates these results where we observe that RA-AGD and RA-\bcgreen{GMM} converges to a smaller asymptotic entropic risk value but at a slower pace as expected illustrating our theoretical results.  
\begin{figure}
    \centering
    \includegraphics[width=0.8\linewidth]{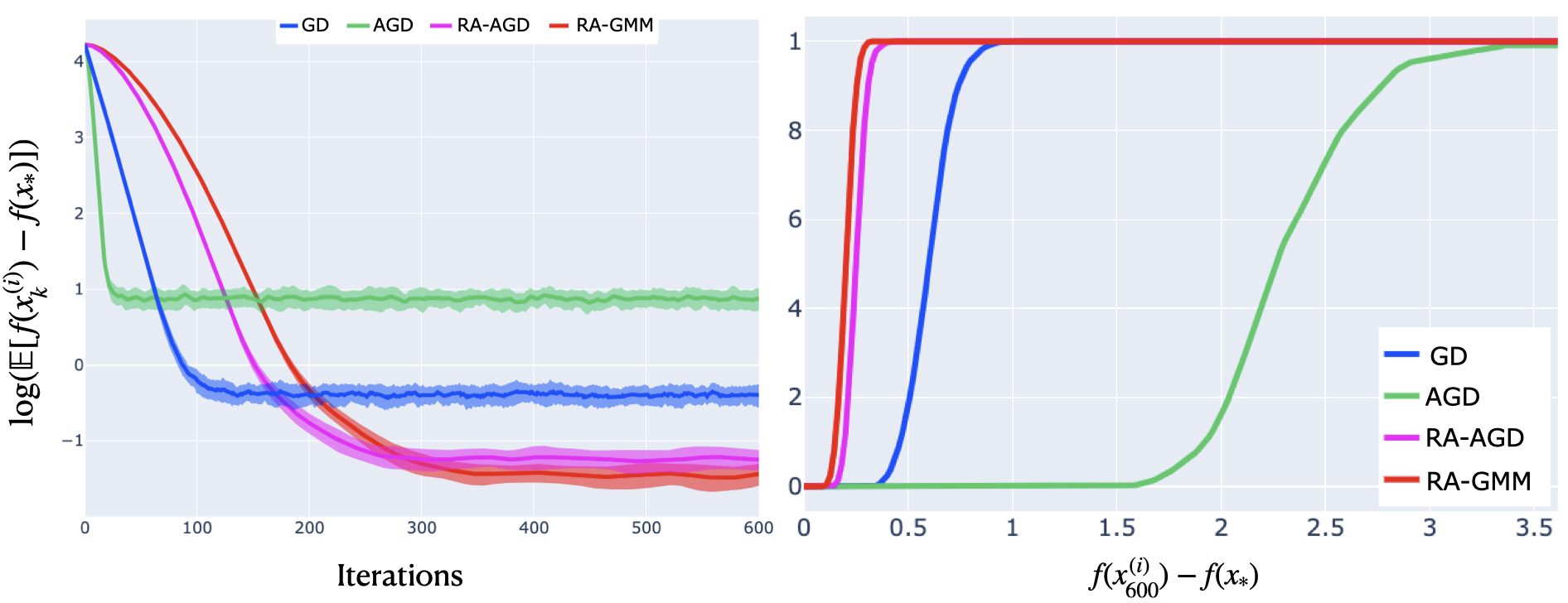}
    \caption{\textbf{(Left)} The expected suboptimality curves for each algorithm (lightly colored regions around each curve highlight the points within one standard deviation from the expected suboptimality), \textbf{(Right)} The cumulative distribution of the last iterates generated by \bcgreen{RA-GMM, RA-AGD, AGD, GD subject to gradient noise} on the logistic regression problem.}
    \label{fig: str_cnvx_subopt_and_hist}
\end{figure}
\begin{figure}
    \centering
\includegraphics[width=0.4\linewidth,height=0.3\linewidth]{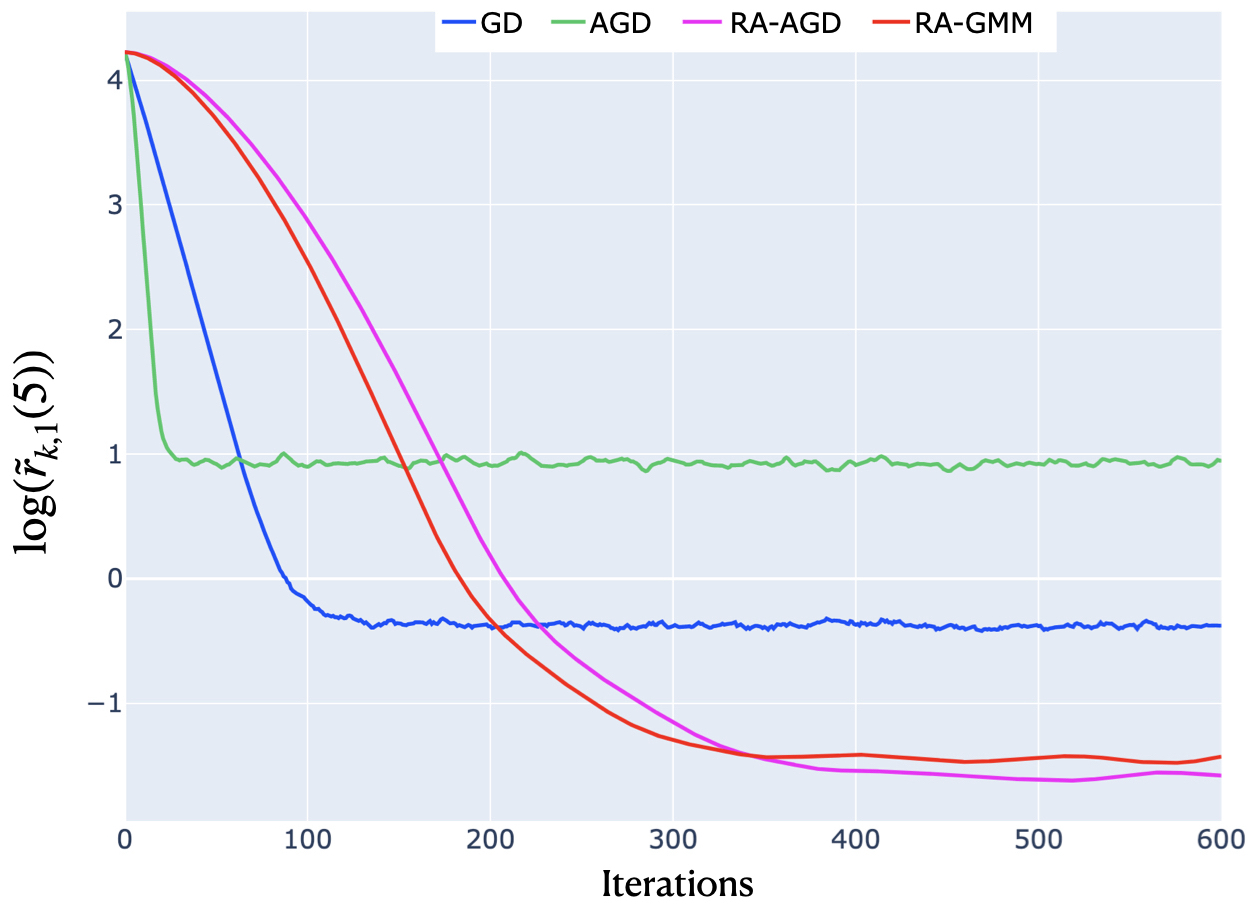}
    \caption{The empirical finite-horizon risk measure $\tilde{r}_{k,1}(5)$ over iterations for the logistic regression problem.}
    \label{fig: str_cnvx_risk_meas}
\end{figure}
\section{\bcgreen{Extensions} \bcred{to sub-Gaussian noise and \bcgreen{to more} general choice of parameters}}\label{sec: Subgaussian}
\bcgreen{
In this section, we discuss how our results can be extended to hold under more general noise that have light tails. The same proof techniques as before will work, albeit some minor differences in the constants. In the following, we assume the noise, when conditioned on the iterates, have sub-Gaussian tails. Such noise structure have been commonly considered in the literature for studying stochastic gradient methods and their variants (see e.g. \cite{ghadimi-lan,alistarh2018byzantine,agarwal}). In empirical risk minimization problems such as the logistic regression problem we considered in Section \ref{sec: num-exp}; if the gradients are estimated from randomly selected $b$ data points, the noise will behave sub-Gaussian for $b$ large enough due to the central limit theorem and this assumption will hold \cite{agarwal}. If the noise is uniformly bounded over iterations (with probability one); this assumption will also hold.}
\begin{assump}\label{Assump: Gen-Noise}
For each $k\in \mathbb{N}$, the gradient noise $\bc{\veps_{k+1}}=\tn f(y_k) - \nabla f(y_k) $ is independent from the natural filtration $\mathcal{F}_k$ generated by the iterates $\{x_j\}_{j=0}^k$, it is centered and sub-Gaussian with a variance proxy $\sigma^2>0$, i.e.,
\begin{eqnarray} 
\mathbb{E}[\bcgreen{w_{k+1}} ~|~ \mathcal{F}_k ] &=& 0,\label{cond: SubGaussian - mean}\\
\mathbb{P}\{\Vert \bcgreen{w_{k+1}}\Vert \geq t ~|~ \mathcal{F}_k\}&\leq& 2e^{-\frac{t^2}{2\sigma^2}}. \label{cond: SubGaussian}
\end{eqnarray} 
\end{assump}
\mg{Theorem \ref{thm: TMM-MI-solution} assumed Gaussian noise and used the properties of the moment generating function of Gaussian distributions to integrate the noise. The next result, extends this result to the sub-Gaussian setting where the proof idea applies the same; except the bounds on the moment generating function of the noise needs to be adapted. The proof is in Appendix \ref{app: SubOpt-Subgaussian}. We obtain the same rate and the asymptotic suboptimality bound (the variance term) has the same dependency to the parameters and is larger by a constant factor of 2 compared to the Gaussian noise setting considered in Theorem  \ref{thm: TMM-MI-solution}}.
\bcred{
\begin{proposition}\label{prop: SubOpt-Subgaussian}
Consider \bc{noisy} \bcgreen{GMM} on $f\in \Sml$ with initialization $\bc{z}_0 = \begin{bmatrix} x_0^{\top} & x_{-1}^{\top} \end{bmatrix}^{\top}\in \mathbb{R}^{2d}$ and with parameters  \bcgreen{$(\alpha,\beta,\gamma)=(
\alpha_{\vartheta,\psi},\beta_{\vartheta,\psi},\gamma_{\vartheta,\psi})$} as given in Theorem \ref{thm: TMM-MI-solution}. Under Assumption \ref{Assump: Gen-Noise}, the expected suboptimality satisfies the bound 
\begin{equation}\label{ineq: subopt-subgauss}
\mE[f(x_k)-f(x_*)] \leq \bcgreen{ \rho^{2k}\mV(\bc{z}_0)+\bcgreen{ \frac{\sigma^2\alpha^2}{1-\rho^2} (\frac{L}{2}+\lambda_P)}},
\end{equation}
\bcgreen{where $\lambda_P=\frac{\vartheta}{2\alpha_{\vartheta,\psi}}$ and $\rho^2=\rho^2_{\vartheta,\psi}$ as in \ref{thm: TMM-MI-solution}.}
\end{proposition}
Next, we provide bounds on the finite/infinite-horizon entropic risk. The results look similar to Proposition \ref{prop: risk-meas-bound-gaussian} where $\theta$ needs to be chosen smaller; and the entropic risk bounds we obtain are \bcgreen{a bit} larger\footnote{This follows from comparing the formulae \eqref{ineq: fin-hor-risk-meas-str-cnvs-subgaus} and \eqref{ineq: fin-risk-meas-bound-gaus} with the observation that, for $\theta < \theta_u^{sg}$, we have $\frac{1}{2-\theta \alpha_{\vartheta,\psi} (\vartheta + \alpha L)} \leq \frac{1}{2-\frac{1}{4}} = \frac{4}{7}$ and
$\bar{\rho}^{2}_{\vartheta,\psi}\leq \hat{\rho}^{2}_{\vartheta,\psi}$.}. This is expected as we consider more general sub-Gaussian noise (that contains Gaussian noise as a special case). The proof is in the Appendix \ref{app: entr-risk-meas-str-cnvx-subgaus}.
}
\begin{proposition}\label{prop: entr-risk-meas-str-cnvx-subgaus} Consider \bc{noisy} \bcgreen{GMM} on $f\in \Sml$ with initialization $\bc{z}_0=\begin{bmatrix} x_0^{\top} & x_{-1}^{\top} \end{bmatrix}^{\top}\in \mathbb{R}^{2d}$ and with parameters \bcgreen{ $(\alpha,\beta,\gamma)=(
\alpha_{\vartheta,\psi},\beta_{\vartheta,\psi},\gamma_{\vartheta,\psi})$} as given in Theorem \ref{thm: TMM-MI-solution}. Under
Assumption \ref{Assump: Gen-Noise}, if the risk-averseness parameter $\theta$ satisfies the condition 
\begin{equation}\label{cond: cond-theta-subgauss}
    \theta < \theta_u^{sg}:=\min\left\{ \bcgreen{\frac{1-\rho^2}{64 \alpha^2\hat{\mtv}_{\alpha,\beta,\gamma}}}, \bcgreen{\frac{1}{4\alpha^2(L+2\lambda_P)} }\right\},
\end{equation}
where \bcgreen{$\hat{\mtv}_{\alpha,\beta,\gamma}=\frac{2L^2}{\mu}(2(\beta-\gamma)^2+(1-\alpha L)^2(1+2\gamma+2\gamma^2))+\lambda_P \rho^2$ with $\rho=\rho_{\vartheta,\psi}$ and $\lambda_P=\frac{\vartheta}{2\alpha_{\vartheta,\psi}}$}. The finite-horizon entropic risk measure of the algorithm admits the bound
\begin{equation}\label{ineq: fin-hor-risk-meas-str-cnvs-subgaus}
\bcgreen{
    \mr_{k,\sigma^2}(\theta) < \frac{4\sigma^2 \alpha^2(L+2\lambda_P)}{1-\hat{\rho}^2_{\alpha,\beta,\gamma}}+ \hat{\rho}^{2k}_{\alpha,\beta,\gamma} 2\mV_P(\bc{z}_0),}
\end{equation}
where \bcgreen{$P:=\tilde{P}\otimes I_d$ for $\tilde{P}$ defined in \eqref{def: P_Matrix}} and the rate $\hat{\rho}^2_{\alpha,\beta,\gamma} <1$ is given as 
\begin{equation}\label{def: bar-rho-subgauss}
    \bcgreen{\hat{\rho}^{2}_{\alpha,\beta,\gamma}= \frac{1}{2}\left(\rho^2+32\alpha^2\hat{\mtv}_{\alpha,\beta,\gamma}\theta \right)+ \frac{1}{2}\sqrt{(\rho^2+32\alpha^2\hat{\mtv}_{\alpha,\beta,\gamma}\theta)^2 + 128 \alpha^2\hat{\mtv}_{\alpha,\beta,\gamma}\theta}.}
\end{equation}
\end{proposition}
Similarly, based on the results of Proposition \ref{prop: entr-risk-meas-str-cnvx-subgaus}, we can obtain bounds on EV@R that extends Theorem \ref{thm: Evar-TMM-str-cnvx-bound} to the sub-Gaussian noise, the proof is \bcgreen{deferred to} the Appendix.
\begin{theorem}\label{thm: evar-bound-str-cnvx-subgaus}
Consider the \bc{noisy} \bcgreen{GMM} on $f\in\Sml$ where the parameters are set as
\bcgreen{$\alpha=\alpha_{\vartheta,\psi}$, $\beta=\beta_{\vartheta,\psi}$, and $\gamma=\gamma_{\vartheta,\psi}$} as in Theorem \ref{thm: TMM-MI-solution} and the noise on the gradient admits Assumption \ref{Assump: Gen-Noise}. \bcred{Fix an arbitrary $\varphi \in (0,1)$ and confidence level $\zeta\in(0,1)$. Consider $\theta_\varphi^{sg}:=\varphi \theta_{u}^{sg}$ and 
$$
\bcgreen{\htv_{\alpha,\beta,\gamma} := \sqrt{(\rho^2+32\alpha^2\hat{\mtv}_{\alpha,\beta,\gamma}\theta_{\varphi}^{sg})^2+128\alpha^2\hat{\mtv}_{\alpha,\beta,\gamma}\theta_{\varphi}^{sg}},}
$$
\bcgreen{where $\rho=\rho_{\vartheta,\psi}$ and $\lambda_P=\frac{\vartheta}{2\alpha_{\vartheta,\psi}}$.}
If the confidence level satisfies the condition, 
\begin{equation}\label{cond: conf-lev-subgauss}
    \bcgreen{\log(1/\zeta) \leq 128\alpha^4\hat{\mtv}_{\alpha,\beta,\gamma}(L+2\lambda_P)\Big(\frac{\theta_{\varphi}^{sg}}{2-\rho^2-\htv_{\alpha,\beta,\gamma}-\theta_{\varphi}^{sg}32\alpha^2\hat{\mtv}_{\alpha\beta,\gamma}}\Big)^2,}
\end{equation}
\bcgreen{for $\hat{\mtv}_{\alpha,\beta,\gamma}$ defined in Proposition \ref{prop: entr-risk-meas-str-cnvx-subgaus}},
then the EV@R of $f(x_k)-f(x_*)$ admits the bound 
\begin{equation}\label{ineq: fin-step-evar-bound-subgauss-1}
    EV@R_{1-\zeta}[f(x_k)-f(x_*)]
    \bcgreen{
    \bc{\leq} \frac{\sigma^2\alpha^2}{2-\rho^2-\htv_{\alpha,\beta,\gamma}}\Big(\sqrt{8(L+2\lambda_P)}+\sqrt{64\hat{\mtv}_{\alpha,\beta,\gamma}\log(1/\zeta)} \Big)^2 + (\hhrho_{\alpha,\beta,\gamma})^{k}2\mV(\bc{z}_0)}.
\end{equation}
In the case \eqref{cond: conf-lev-subgauss} does not hold, we have 
\begin{equation}\label{ineq: fin-step-evar-bound-subgauss-2}
EV@R_{1-\zeta}[f(x_k)-f(x_*)]\\
\bc{\leq} \frac{8\sigma^2 \bcgreen{\alpha^2(L+2\lambda_P)}}{\bcgreen{\rho^2-\htv_{\alpha,\beta,\gamma}-32\alpha^2\hat{\mtv}_{\alpha,\beta,\gamma}\theta_{\varphi}^{sg}}}+\frac{2\sigma^2}{\theta_{\varphi}^{sg}}\log(1/\zeta)+(\hhrho_{\alpha,\beta,\gamma})^k 2\mV(\bc{z}_0),
\end{equation}
where the rate \bcgreen{$\hhrho_{\alpha,\beta,\gamma}$} is given as 
\begin{equation}\label{def: hat-hat-rho}
\bcgreen{
\hhrho_{\alpha,\beta,\gamma}:= \frac{1}{2}\big(\rho^2+32\alpha^2\hat{\mtv}_{\alpha,\beta,\gamma}\theta_{\varphi}^{sg}\big)+\frac{1}{2}\sqrt{(\rho^2+32\alpha^2\hat{\mtv}_{\alpha,\beta,\gamma}\theta_{\varphi}^{sg})^2+128 \alpha^2\hat{\mtv}_{\alpha,\beta,\gamma}\theta_{\varphi}^{sg}}.}
\end{equation}
We have also
\begin{small}
\begin{multline}\label{ineq: inf-evar-bound-subgauss}
    \underset{k\rightarrow \infty}{\lim\sup}\; EV@R_{1-\zeta}[f(x_k)-f(x_*)]
    \\ \leq \hat{E}_{1-\zeta}(\alpha,\beta,\gamma):=\begin{cases} 
    \bcgreen{\frac{\sigma^2\alpha^2}{2-\rho^2-\htv_{\alpha,\beta,\gamma}}\Big(\sqrt{8(L+2\lambda_P)}+\sqrt{64\hat{\mtv}_{\alpha,\beta,\gamma}\log(1/\zeta)} \Big)^2, }
    & \text{ if \eqref{cond: conf-lev-subgauss} holds},\\
    \frac{8\sigma^2 \bcgreen{\alpha^2(L+2\lambda_P)}}{\bcgreen{\rho^2-\htv_{\alpha,\beta,\gamma}-32\alpha^2\hat{\mtv}_{\alpha,\beta,\gamma}\theta_{\varphi}^{sg}}}+\frac{2\sigma^2}{\theta_{\varphi}^{sg}}\log(1/\zeta),
    & \text{ otherwise}.
    \end{cases}
\end{multline}
\end{small}

\begin{corollary}
Consider the \bc{noisy} \bcgreen{GMM} in the premise of Theorem \ref{thm: evar-bound-str-cnvx-subgaus}. The tail of the distribution $f(x_k)-f(x_*)$ at $k$-th iteration admits the inequality 
\begin{equation}\label{ineq: tail-bound-subgauss-1}
    \mathbb{P}\{ f(x_k)-f(x_*) \geq \bcgreen{\hat{E}_{1-\zeta}(\alpha,\beta,\gamma)}+(\bcgreen{\hhrho_{\alpha,\beta,\gamma}})^k 2\mV(\bc{z}_0)\}\leq \zeta,
\end{equation}
where $\hhrho$ as given in \eqref{def: bbrho}. Moreover, we have the following for the limiting distribution of \bc{noisy} \bcgreen{GMM}
\begin{equation}\label{ineq: tail-bound-subgauss-2}
    \mathbb{P}\{ f(x_\infty)-f(x_*)\geq \bcgreen{\hat{E}_{1-\zeta}(\alpha,\beta,\gamma)}\} \leq \zeta.
\end{equation} 
\end{corollary}
\begin{proof}We omit the proof since it is similar to the proof of Corollary \ref{cor: gaussian-tail-bound}.
\end{proof}
\bcgreen{
\begin{remark} \label{remark: general-params}
Proposition \ref{prop: SubOpt-Subgaussian}, Proposition \ref{prop: entr-risk-meas-str-cnvx-subgaus}, and Theorem \ref{thm: evar-bound-str-cnvx-subgaus} provide explicit bounds for particular choice of parameters based on constructing explicit solutions to the \eqref{MI} provided in Appendix \ref{sec-app-b}. However, if we do not seek for explicit bounds, then our results
can be extended to more general choice of parameters than the ones provided in Theorem \ref{thm: TMM-MI-solution}. In particular, suppose there exist a choice of parameters $(\alpha,\beta, \gamma)$, a positive semi-definite matrix $\tilde{P}\in\mathbb{R}^2$ and a convergence rate $\rho^2 \in (0,1)$ which satisfies the matrix inequality \eqref{MI}. If such parameters, $\tilde{P}$ and $\rho^2$ can be found numerically, then
the results of Proposition \ref{prop: SubOpt-Subgaussian}, Proposition \ref{prop: entr-risk-meas-str-cnvx-subgaus}, and Theorem \ref{thm: evar-bound-str-cnvx-subgaus} will hold verbatim if 
we take $\lambda_P$ to be the largest eigenvalue of the $d$-by-$d$ principle minor matrix of $P := \tilde{P} \otimes I_d$ and use the parameters $(\alpha,\beta,\gamma)$ and the rate $\rho^2$ that satisfies the \eqref{MI}.
\end{remark}
}
}
\end{theorem}
\section{Conclusion}\label{sec: conclusion}
\mg{In this work, we considered \bcgreen{generalized momentum methods (GMM)} subject to random gradient noise. We study the convergence rate and entropic risk of suboptimality as a function of the parameters. For both quadratic objectives and strongly convex smooth objectives, we obtained a number of new non-asymptotic convergence results when the noise is light-tailed (Gaussian or sub-Gaussian) as a function of the parameters. Furthermore, we provided explicit bounds on the entropic risk and entropic value at risk of suboptimality at a given iterate as well as the probability that the suboptimality exceeds a given threshold. For quadratic objectives, we also obtained (explicit) sharp characterizations of the entropic risk and the set of parameters for which the entropic risk is finite. Our results highlight trade-offs between the convergence rate and the entropic risk. We also introduced risk-averse \bcgreen{GMM} (RA-\bcgreen{GMM}) which can select the parameters to trade-off the convergence rate and the entropic value at risk in a systematic manner. Finally, we provided numerical experiments which illustrate our result and show that RA-\bcgreen{GMM} leads to improved tail behavior.
}

\vspace{1cm}
\textbf{Acknowledgements} This work was partially funded by the grants ONR N00014-21-1-2244, NSF CCF-1814888, and NSF DMS-2053485.

\bibliographystyle{unsrt}
\bibliography{references}  

\begin{thebibliography}{10}

\bibitem{bottou2018optimization}
L{\'e}on Bottou, Frank~E Curtis, and Jorge Nocedal.
\newblock Optimization methods for large-scale machine learning.
\newblock {\em SIAM Review}, 60(2):223--311, 2018.

\bibitem{robbins1951stochastic}
Herbert Robbins and Sutton Monro.
\newblock A stochastic approximation method.
\newblock {\em The annals of mathematical statistics}, pages 400--407, 1951.

\bibitem{kuru2020differentially}
Nurdan Kuru, {\c{S}}~{\.I}lker Birbil, Mert G\"urb\"uzbalaban, and Sinan Y{\i}ld{\i}r{\i}m.
\newblock Differentially private accelerated optimization algorithms.
\newblock {\em Siam Journal on Optimization, To Appear,}, 2022.

\bibitem{bassily2014private}
Raef Bassily, Adam Smith, and Abhradeep Thakurta.
\newblock Private empirical risk minimization: Efficient algorithms and tight error bounds.
\newblock In {\em 2014 IEEE 55th Annual Symposium on Foundations of Computer Science}, pages 464--473. IEEE, 2014.

\bibitem{aybat2019robust}
Necdet~Serhat Aybat, Alireza Fallah, Mert Gürbüzbalaban, and Asuman Ozdaglar.
\newblock Robust accelerated gradient methods for smooth strongly convex functions.
\newblock {\em SIAM Journal on Optimization}, 30(1):717--751, 2020.

\bibitem{dieuleveut2020bridging}
Aymeric Dieuleveut, Alain Durmus, and Francis Bach.
\newblock Bridging the gap between constant step size stochastic gradient descent and {M}arkov chains.
\newblock {\em The Annals of Statistics}, 48(3):1348--1382, 2020.

\bibitem{flammarion2015averaging}
Nicolas Flammarion and Francis Bach.
\newblock From averaging to acceleration, there is only a step-size.
\newblock In Peter Grünwald, Elad Hazan, and Satyen Kale, editors, {\em Proceedings of The 28th Conference on Learning Theory}, volume~40 of {\em Proceedings of Machine Learning Research}, pages 658--695, Paris, France, 03--06 Jul 2015. PMLR.

\bibitem{pmlr-v99-jain19a}
Prateek Jain, Dheeraj~M. Nagaraj, and Praneeth Netrapalli.
\newblock Making the last iterate of sgd information theoretically optimal.
\newblock {\em SIAM Journal on Optimization}, 31(2):1108--1130, 2021.

\bibitem{harvey2019tight}
Nicholas J.~A. Harvey, Christopher Liaw, Yaniv Plan, and Sikander Randhawa.
\newblock Tight analyses for non-smooth stochastic gradient descent.
\newblock In Alina Beygelzimer and Daniel Hsu, editors, {\em Proceedings of the Thirty-Second Conference on Learning Theory}, volume~99 of {\em Proceedings of Machine Learning Research}, pages 1579--1613. PMLR, 25--28 Jun 2019.

\bibitem{nesterov2003introductory}
Yurii Nesterov.
\newblock {\em Introductory lectures on convex optimization: A basic course}, volume~87.
\newblock Springer Science \& Business Media, 2003.

\bibitem{polyak1987introduction}
Boris~T Polyak.
\newblock Introduction to optimization. optimization software.
\newblock {\em Inc., Publications Division, New York}, 1:32, 1987.

\bibitem{Hardt-blog}
M.~Hardt.
\newblock Robustness versus acceleration, August 2014.

\bibitem{devolder2014first}
O.~Devolder, F.~Glineur, and Y.~Nesterov.
\newblock First-order methods of smooth convex optimization with inexact oracle.
\newblock {\em Mathematical Programming}, 146(1-2):37--75, 2014.

\bibitem{aspremontSmooth08}
A.~d'Aspremont.
\newblock Smooth optimization with approximate gradient.
\newblock {\em SIAM Journal on Optimization}, 19(3):1171--1183, 2008.

\bibitem{Schmidt11InexactProx}
M.~Schmidt, Nicolas {Le Roux}, and F.R. Bach.
\newblock Convergence rates of inexact proximal-gradient methods for convex optimization.
\newblock In J.~Shawe-Taylor, R.~S. Zemel, P.~L. Bartlett, F.~Pereira, and K.~Q. Weinberger, editors, {\em Advances in Neural Information Processing Systems 24}, pages 1458--1466. Curran Associates, Inc., 2011.

\bibitem{devolder2013thesis}
Olivier Devolder.
\newblock {\em Exactness, inexactness and stochasticity in first-order methods for large-scale convex optimization}.
\newblock PhD thesis, ICTEAM and CORE, Universit{\'e} Catholique de Louvain, 2013.

\bibitem{can2019accelerated}
Bugra Can, Mert G\"urb\"uzbalaban, and Lingjiong Zhu.
\newblock Accelerated linear convergence of stochastic momentum methods in {W}asserstein distances.
\newblock In Kamalika Chaudhuri and Ruslan Salakhutdinov, editors, {\em Proceedings of the 36th International Conference on Machine Learning}, volume~97 of {\em Proceedings of Machine Learning Research}, pages 891--901. PMLR, 09--15 Jun 2019.

\bibitem{aybat2019universally}
Necdet~Serhat Aybat, Alireza Fallah, Mert G\"urb\"uzbalaban, and Asuman Ozdaglar.
\newblock A universally optimal multistage accelerated stochastic gradient method.
\newblock {\em Advances in {N}eural {I}nformation {P}rocessing {S}ystems}, 32, 2019.

\bibitem{gadat-stoc-heavy-ball}
Sébastien Gadat, Fabien Panloup, and Sofiane Saadane.
\newblock {Stochastic heavy ball}.
\newblock {\em Electronic Journal of Statistics}, 12(1):461 -- 529, 2018.

\bibitem{hu2017dissipativity}
Bin Hu and Laurent Lessard.
\newblock Dissipativity theory for {N}esterov's accelerated method.
\newblock In Doina Precup and Yee~Whye Teh, editors, {\em Proceedings of the 34th International Conference on Machine Learning}, volume~70 of {\em Proceedings of Machine Learning Research}, pages 1549--1557. PMLR, 06--11 Aug 2017.

\bibitem{gannot2021frequency}
Oran Gannot.
\newblock A frequency-domain analysis of inexact gradient methods.
\newblock {\em Mathematical Programming}, pages 1--42, 2021.

\bibitem{scoy-tmm}
Bryan Van~Scoy, Randy~A. Freeman, and Kevin~M. Lynch.
\newblock The fastest known globally convergent first-order method for minimizing strongly convex functions.
\newblock {\em IEEE Control Systems Letters}, 2(1):49--54, 2018.

\bibitem{vanscoy2021speedrobustness}
Bryan Van~Scoy and Laurent Lessard.
\newblock The speed-robustness trade-off for first-order methods with additive gradient noise, 2021.

\bibitem{ruszczynski2013advances}
Andrzej Ruszczy{\'n}ski.
\newblock Advances in risk-averse optimization.
\newblock In {\em Theory Driven by Influential Applications}, pages 168--190. INFORMS, 2013.

\bibitem{Javid}
Amir Ahmadi-Javid.
\newblock Entropic value-at-risk: A new coherent risk measure.
\newblock {\em Journal of Optimization Theory and Applications}, 155(3):1105--1123, 2012.

\bibitem{gitman2019understanding}
Igor Gitman, Hunter Lang, Pengchuan Zhang, and Lin Xiao.
\newblock Understanding the role of momentum in stochastic gradient methods.
\newblock {\em Advances in {N}eural {I}nformation {P}rocessing {S}ystems}, 32, 2019.

\bibitem{liu2020improved}
Yanli Liu, Yuan Gao, and Wotao Yin.
\newblock An improved analysis of stochastic gradient descent with momentum.
\newblock {\em Advances in Neural Information Processing Systems}, 33:18261--18271, 2020.

\bibitem{ghadimi-heavy-ball}
Euhanna Ghadimi, Hamid~Reza Feyzmahdavian, and Mikael Johansson.
\newblock Global convergence of the heavy-ball method for convex optimization, 2014.

\bibitem{michalowsky2014multidimensional}
Simon Michalowsky and Christian Ebenbauer.
\newblock The multidimensional n-th order heavy ball method and its application to extremum seeking.
\newblock In {\em 53rd IEEE Conference on Decision and Control}, pages 2660--2666. IEEE, 2014.

\bibitem{lessard2016analysis}
Laurent Lessard, Benjamin Recht, and Andrew Packard.
\newblock Analysis and design of optimization algorithms via integral quadratic constraints.
\newblock {\em SIAM Journal on Optimization}, 26(1):57--95, 2016.

\bibitem{hu2017approxSG}
Bin Hu, Peter Seiler, and Laurent Lessard.
\newblock Analysis of biased stochastic gradient descent using sequential semidefinite programs, 2017.

\bibitem{fazlyab2017analysis}
M.~Fazlyab, A.~Ribeiro, M.~Morari, and V.~Preciado.
\newblock Analysis of optimization algorithms via integral quadratic constraints: Nonstrongly convex problems.
\newblock {\em SIAM Journal on Optimization}, 28(3):2654--2689, 2018.

\bibitem{fallah2019robust}
Alireza Fallah, Mert G\"urb\"uzbalaban, Asu Ozdaglar, Umut Simsekli, and Lingjiong Zhu.
\newblock Robust distributed accelerated stochastic gradient methods for multi-agent networks.
\newblock {\em arXiv preprint arXiv:1910.08701}, 2019.

\bibitem{mohammadi2020robustness}
Hesameddin Mohammadi, Meisam Razaviyayn, and Mihailo~R Jovanovi{\'c}.
\newblock Robustness of accelerated first-order algorithms for strongly convex optimization problems.
\newblock {\em IEEE Transactions on Automatic Control}, 66(6):2480--2495, 2020.

\bibitem{zhang2021robust}
Xuan Zhang, Necdet~Serhat Aybat, and Mert G{\"u}rb{\"u}zbalaban.
\newblock Robust accelerated primal-dual methods for computing saddle points.
\newblock {\em arXiv preprint arXiv:2111.12743}, 2021.

\bibitem{michalowsky2021robust}
Simon Michalowsky, Carsten Scherer, and Christian Ebenbauer.
\newblock Robust and structure exploiting optimisation algorithms: an integral quadratic constraint approach.
\newblock {\em International Journal of Control}, 94(11):2956--2979, 2021.

\bibitem{mohammadi2018variance}
Hesameddin Mohammadi, Meisam Razaviyayn, and Mihailo~R Jovanovi{\'c}.
\newblock Variance amplification of accelerated first-order algorithms for strongly convex quadratic optimization problems.
\newblock In {\em 2018 IEEE Conference on Decision and Control (CDC)}, pages 5753--5758. IEEE, 2018.

\bibitem{lan2012optimal}
Guanghui Lan.
\newblock An optimal method for stochastic composite optimization.
\newblock {\em Mathematical Programming}, 133(1):365--397, 2012.

\bibitem{ghadimi2012optimal}
Saeed Ghadimi and Guanghui Lan.
\newblock Optimal stochastic approximation algorithms for strongly convex stochastic composite optimization {I}: A generic algorithmic framework.
\newblock {\em SIAM Journal on Optimization}, 22(4):1469--1492, 2012.

\bibitem{ghadimi-lan}
Saeed Ghadimi and Guanghui Lan.
\newblock Optimal stochastic approximation algorithms for strongly convex stochastic composite optimization, {II}: Shrinking procedures and optimal algorithms.
\newblock {\em SIAM Journal on Optimization}, 23(4):2061--2089, 2013.

\bibitem{Loizou-shb}
Nicolas {Loizou} and Peter {Richt{\'a}rik}.
\newblock {Momentum and Stochastic Momentum for Stochastic Gradient, Newton, Proximal Point and Subspace Descent Methods}.
\newblock {\em arXiv e-prints}, page arXiv:1712.09677, December 2017.

\bibitem{orabona}
Xiaoyu {Li} and Francesco {Orabona}.
\newblock {A High Probability Analysis of Adaptive SGD with Momentum}.
\newblock {\em arXiv e-prints}, page arXiv:2007.14294, July 2020.

\bibitem{gower-shb}
Othmane Sebbouh, Robert~M Gower, and Aaron Defazio.
\newblock Almost sure convergence rates for stochastic gradient descent and stochastic heavy ball.
\newblock In Mikhail Belkin and Samory Kpotufe, editors, {\em Proceedings of Thirty Fourth Conference on Learning Theory}, volume 134 of {\em Proceedings of Machine Learning Research}, pages 3935--3971. PMLR, 15--19 Aug 2021.

\bibitem{yang2016unified}
Tianbao Yang, Qihang Lin, and Zhe Li.
\newblock Unified convergence analysis of stochastic momentum methods for convex and non-convex optimization.
\newblock {\em arXiv preprint arXiv:1604.03257}, 2016.

\bibitem{abadi2016deep}
Martin Abadi, Andy Chu, Ian Goodfellow, H.~Brendan McMahan, Ilya Mironov, Kunal Talwar, and Li~Zhang.
\newblock Deep learning with differential privacy.
\newblock In {\em Proceedings of the 2016 ACM SIGSAC Conference on Computer and Communications Security}, CCS '16, page 308–318, New York, NY, USA, 2016. Association for Computing Machinery.

\bibitem{ganesh2022langevin}
Arun Ganesh, Abhradeep Thakurta, and Jalaj Upadhyay.
\newblock Langevin diffusion: An almost universal algorithm for private {E}uclidean (convex) optimization.
\newblock {\em arXiv preprint arXiv:2204.01585}, 2022.

\bibitem{ahmadi2019portfolio}
Amir Ahmadi-Javid and Malihe Fallah-Tafti.
\newblock Portfolio optimization with entropic value-at-risk.
\newblock {\em European Journal of Operational Research}, 279(1):225--241, 2019.

\bibitem{fleming2006controlled}
Wendell~H Fleming and Halil~Mete Soner.
\newblock {\em Controlled Markov processes and viscosity solutions}, volume~25.
\newblock Springer Science \& Business Media, 2006.

\bibitem{cajas2021entropic}
Dany Cajas.
\newblock Entropic portfolio optimization: a disciplined convex programming framework.
\newblock {\em Available at SSRN 3792520}, 2021.

\bibitem{shi2021coherent}
Xiang Shi, Young~Shin Kim, et~al.
\newblock Coherent risk measures and normal mixture distributions with applications in portfolio optimization.
\newblock {\em International Journal of Theoretical and Applied Finance (IJTAF)}, 24(04):1--18, 2021.

\bibitem{Ruszczynski2006}
Andrzej Ruszczy{\'{n}}ski and Alexander Shapiro.
\newblock {\em Optimization of Risk Measures}, pages 119--157.
\newblock Springer London, London, 2006.

\bibitem{riedel}
Frank Riedel.
\newblock Dynamic coherent risk measures.
\newblock {\em Stochastic Processes and their Applications}, 112(2):185--200, 2004.

\bibitem{gelfand}
I.~Gelfand.
\newblock Normierte ringe.
\newblock {\em Rec. Math. [Mat. Sbornik] N.S.}, 9(51):3--24, 1941.

\bibitem{chen-book}
Chi-Tsong Chen.
\newblock {\em Linear System Theory and Design}.
\newblock Oxford University Press, Inc., USA, 3rd edition, 1998.

\bibitem{alistarh2018byzantine}
Dan Alistarh, Zeyuan Allen-Zhu, and Jerry Li.
\newblock Byzantine stochastic gradient descent.
\newblock {\em Advances in Neural Information Processing Systems}, 31, 2018.

\bibitem{agarwal}
Alekh Agarwal, Sahand Negahban, and Martin~J Wainwright.
\newblock Stochastic optimization and sparse statistical recovery: Optimal algorithms for high dimensions.
\newblock In F.~Pereira, C.~J.~C. Burges, L.~Bottou, and K.~Q. Weinberger, editors, {\em Advances in Neural Information Processing Systems}, volume~25. Curran Associates, Inc., 2012.

\bibitem{Recht}
Benjamin Recht.
\newblock Lyapunov analysis and the heavy ball method.
\newblock {\em Online lecture notes}, 2012.

\bibitem{lancaster1972norms}
Peter Lancaster and Hanafi~K Farahat.
\newblock Norms on direct sums and tensor products.
\newblock {\em mathematics of computation}, 26(118):401--414, 1972.

\bibitem{schacke2004kronecker}
Kathrin Schacke.
\newblock On the {K}ronecker product.
\newblock {\em Master's thesis, University of Waterloo}, 2004.

\bibitem{vershynin2018high}
Roman Vershynin.
\newblock {\em High-dimensional probability: An introduction with applications in data science}, volume~47.
\newblock Cambridge university press, 2018.

\bibitem{abramowitzhandbook}
Milton Abramowitz and Irene~A Stegun.
\newblock {\em HandBook of mathematical functions with formulas, graphs, and mathematical tables}, volume~55.
\newblock US Government printing office, 1964.

\end{thebibliography}

\appendix
\section{Proofs of Section \ref{sec: quad-obj}} 
\subsection{Proof of Lemma \ref{lem: non_asym_conv_quad_obj}}
The result follows from the Jordan decomposition of the matrix $A_Q$.  By arguments similar to \cite{Recht}, there exists an orthonormal matrix \bcgreen{$U$} such that 
    \begin{equation*}
     \bcgreen{U}A_Q\bcgreen{U}^{\top} = \text{blkdiag}( \{M_i\}_{i=1}^n), 
     \end{equation*}
where $M_i$ is the $2\times 2$ matrix 
\begin{equation}\label{def: AQ-block-diag-form}
M_i :=\begin{bmatrix} (1+\beta) - \alpha (1+\gamma )\lambda_i  & -(\beta -\alpha\gamma \lambda_i) \\
    1 &  0 
    \end{bmatrix}.
\end{equation}
\bc{Recall that} 
$$c_i=(1+\beta) - \alpha (1+\gamma )\lambda_i, \quad d_i = -(\beta -\alpha\gamma \lambda_i),$$ 
$M_i$ admits the Jordan decomposition $M_i = V_i J_i V_i^{-1}$ where
 $$ V_i = 
\begin{cases} 
    \begin{bmatrix} 
            \frac{c_i - \sqrt{c_i^2 + 4d_i}}{2} & \frac{c_i + \sqrt{c_i^2 + 4d_i}}{2}\\
            1 & 1
    \end{bmatrix}, & \mbox{if} \quad c_i^2 + 4d_i \neq 0, \\
     \begin{bmatrix} 
            \frac{c_i}{2} & 1\\
            1 & 0
    \end{bmatrix}, & \mbox{if} \quad c_i^2 + 4d_i = 0,
 \end{cases}
 $$
  $$ V_i^{-1} =  
\begin{cases} 
    \frac{1}{\sqrt{c_i^2 + 4d_i}} \begin{bmatrix} 
            -1 &  \frac{c_i + \sqrt{c_i^2 + 4d_i}}{2}\\
            1 & -\frac{c_i - \sqrt{c_i^2 + 4d_i}}{2} 
    \end{bmatrix}, & \mbox{if} \quad c_i^2 + 4d_i \bc{\neq} 0, \\
     \begin{bmatrix} 
            0 & 1\\
            1 & \frac{-c_i}{2}
    \end{bmatrix}, & \mbox{if} \quad c_i^2 + 4d_i = 0,
 \end{cases}
 $$
  $$ J_i =  
\begin{cases} 
     \begin{bmatrix} 
            \frac{c_i - \sqrt{c_i^2 + 4d_i}}{2} & 0 \\
             0 & \frac{c_i + \sqrt{c_i^2 + 4d_i}}{2}
    \end{bmatrix}, & \mbox{if} \quad c_i^2 + 4d_i \neq 0, \\
     \begin{bmatrix} 
           \frac{c_i}{2} & 1 \\
             0 & \frac{c_i }{2}
    \end{bmatrix}, & \mbox{if} \quad c_i^2 + 4d_i = 0.
 \end{cases}
 $$
\bc{We observe that }if $c_i^2 + 4d_i = 0$, then $M_i$ has a double eigenvalue with a corresponding Jordan block of size 2, otherwise $M_i$ is diagonalizable with two different eigenvalues (Jordan blocks are of size 1). It follows that
   $$ M_i^k = V_i J_i^k V_i^{-1},$$
where 

$$ J_i^k =  
\begin{cases} 
     \begin{bmatrix} 
            \left(\frac{c_i - \sqrt{c_i^2 + 4d_i}}{2}\right)^k & 0 \\
             0 & \left(\frac{c_i + \sqrt{c_i^2 + 4d_i}}{2}\right)^k
    \end{bmatrix}, & \mbox{if} \quad c_i^2 + 4d_i \neq 0, \\
     \begin{bmatrix} 
           (\frac{c_i}{2})^k & k (\frac{c_i}{2})^{k-1} \\
             0 & (\frac{c_i }{2})^k
    \end{bmatrix}, & \mbox{if} \quad c_i^2 + 4d_i = 0.
 \end{cases}
 $$
 Then, we have
\begin{eqnarray} \|A_Q^k\| &=& \|\bcgreen{U}^{\top} \text{blkdiag}( \{M_i^k\}_{i=1}^n) \bcgreen{U}\|, \\
             &\leq& \max_i \|M_i^k\|, \\
             &\leq & \max_i  \|V_i\| \|J_i^k\| \|V_i^{-1}\|, \label{ineq-power-matrix}
\end{eqnarray}
where we used $\|P \| = 1$. \mg{Notice that 
\begin{align*}
\rho(J_i) 
= \rho_i = \begin{cases} 
\frac{1}{2}|c_i|+ \frac{1}{2}\sqrt{c_i^2 + 4d_i}, & \text{ if } c_i^2+4d_i>0, \\
\sqrt{|d_i|}, & \text{ otherwise},
\end{cases}
\end{align*}
We have also
}
$$\rho(A_Q) = \max_i \rho_i$$
by definition. If $c_i^2 + 4d_i \neq 0$ then $J_i^k$ is diagonal and we obtain
\begin{equation}
    \|J_i^k\| \leq \mg{\rho_i^k} \quad \mbox{if} \quad c_i^2 + 4d_i \neq 0.
\end{equation}
Otherwise, if $c_i^2 + 4d_i=0$, we have $\bc{\frac{c_i}{2} =} \mg{\rho_i}$ and
\begin{equation}
    \|J_i^k\|_2 \leq \|J_i^k\|_F := \sqrt{2(\frac{c_i}{2})^{2k}+{k^2(\frac{c_i}{2})^{2k-2}}} \bc{=} \sqrt{\mg{2\rho_i^{2k}+k^2\rho_i^{2k-2}}},
\end{equation}
where $\|\cdot\|_F$ denotes the Frobenius norm. In any case, we have 
\begin{equation}
    \|J_i^k\|_2 \leq \sqrt{2 \big(\rho(A_Q)\big)^2+ k^2}\rho(A_Q)^{k-1} \mbox{ for every } i,
\end{equation}
where we used $\mg{\rho_i} \leq \rho(A_Q)$. \bcred{Moreover},
\begin{equation}
    \|V_i\|_2 \leq \|V_i\|_F := 
            \begin{cases} 
            \sqrt{|c_i^2+2d_i+2|},
                    & \mbox{if} \quad c_i^2 + 4d_i \neq 0,  \\
                    \sqrt{\frac{c_i^2}{4} +2}, &  \mbox{if}\quad c_i^2 + 4d_i = 0.
            \end{cases}
\end{equation}
Similarly, 
\begin{equation}
    \|V_i^{-1}\|_2 \leq \|V_i^{-1}\|_F := 
            \begin{cases}  
        \sqrt{\frac{|c_i^2+2d_i+2|}{|c_i^2+4d_i|}},
                    & \mbox{if} \quad c_i^2 + 4d_i \neq 0,  \\
                   \sqrt{\frac{c_i^2}{4} + 2}, & \mbox{if}\quad c_i^2 + 4d_i = 0.
            \end{cases}
\end{equation}
Therefore,
$$\|V_i\|_2  \|V_i^{-1}\|_2 \leq 
            \begin{cases}  
                   \frac{|c_i^2+2d_i+2|}{\sqrt{|c_i^2+4d_i|}},
                    & \mbox{if} \quad c_i^2 + 4d_i \neq 0,  \\
                    \frac{c_i^2}{4}+2,
                   & \mbox{if}\quad c_i^2 + 4d_i = 0.
            \end{cases}
$$            
Putting everything together, we conclude from \bc{\eqref{sys: TMM_quad} that}
$$\bc{\|\mathbb{E}[z_k-z_*]\|  \leq  \| A_Q^k\| \|z_0-z_*\|},$$ 
where
    \begin{eqnarray*} \|A_Q^k\| &\leq&  \mg{
    \max_{i=1,2, \dots,d}
        \begin{cases}  
      \frac{|c_i^2+2d_i+2|}{\sqrt{|c_i^2+4d_i|}} \rho_i^{k}, 
                    & \mbox{if} \quad c_i^2 + 4d_i \neq 0,  \\
                 (\frac{c_i^2}{4}+2)\sqrt{2\rho_i^{2k}+k^2\rho_i^{2k-2}},
                   & \mbox{if}\quad c_i^2 + 4d_i = 0,
            \end{cases}
            }\\
&\leq& \bcgreen{ \rho(A_Q)^{k-1} \max_{i=1,2, \dots,d}
        \begin{cases}  
      \frac{|c_i^2+2d_i+2|\rho_i}{\sqrt{|c_i^2+4d_i|}}  
    & \mbox{if} \quad c_i^2 + 4d_i \neq 0,  \\
                    (\frac{c_i^2}{4}+2)\sqrt{2 \rho_i^2+ k^2}
                   & \mbox{if}\quad c_i^2 + 4d_i = 0.
            \end{cases}}
    \end{eqnarray*}
\bc{This completes the proof.}
%
\subsection{Proof of Proposition \ref{prop: quad_qisk_meas_convergence}}\label{app: quad_qisk_meas_convergence}
\bcred{Before proving Proposition \ref{prop: quad_qisk_meas_convergence}, we first give the characterization of eigenvalues of the matrix $Q^{1/2}\Sigma_\infty Q^{1/2}$, where $\Sigma_\infty$ is the limiting covariance matrix \bc{defined} as $\Sigma_\infty:= \lim_{k\rightarrow \infty} \mathbb{E}[(x_k-\mathbb{E}[x_k])(x_k-\mathbb{E}[x_k])^\top]$.}
\bcred{
\begin{lemma}\label{lem: sol-lyap-eq-quad} Under the setting of Proposition \ref{prop: quad_qisk_meas_convergence}, the eigenvalues of the matrix $Q^{1/2}\Sigma_\infty Q^{1/2}$ where $\Sigma_\infty$ is the limiting covariance matrix $\Sigma_\infty:= \lim_{k\rightarrow \infty} \mathbb{E}[(x_k-\mathbb{E}[x_k])(x_k-\mathbb{E}[x_k])^\top]$ admit the form 
$$
\lambda_i\left( Q^{1/2}\Sigma_\infty Q^{1/2}\right)= 
\frac{\lambda_i(Q) (1-d_i)\alpha^2 \sigma^2}{(1+d_i)[(1-d_i)^2-c_i^2]},
$$
for $i=1,2,\dots,d$.
\end{lemma}
\bc{
\begin{proof}
The proof is given in Appendix \ref{app: sol-lyap-eq-quad}.
\end{proof}
}
}
Under the assumption on the noise each iteration $x_k$ \bc{is} \bc{Gaussian with $\mu_k:=\mathbb{E}[x_k]$} and covariance matrix $\Sigma_k:=\mathbb{E}[(x_k-\mu_k)(x_k-\mu_k)^\top]$. \bc{Notice that the following recursion holds for the covariance matrix $\Xi_k$ of $\bc{z}_k$ \bcred{under} Assumption \ref{Assump: Noise} (see \bcred{also} the proof of Lemma \ref{lem: sol-lyap-eq-quad}), 
\begin{align}
\Xi_{k+1}= A_{Q}\Xi_{k}A_Q^\top + \sigma^2 BB^\top. \label{eq: Var_It_Quad}
\end{align} 
\bcred{Moreover, for all $(\alpha,\beta,\gamma)\in\mathcal{F}_\theta$ the rate satisfies $\rho(A_Q)<1$ since $(\alpha,\beta,\gamma)\in\mathcal{S}_q$ by Lemma \ref{lem: feas-set-stab-set}. Therefore, we can see that Lyapunov equation given above is stable, \bcgreen{i.e. $\rho(A_Q)<1$}; hence, }
there exists $\Xi_\infty$ such that $\lim_{k\rightarrow \infty}\Xi_k=\Xi_\infty$ which also satisfies the Lyapunov equation $\Xi_\infty= A_Q\Xi_\infty A^\top_Q+ \sigma^2 BB^\top$. Therefore, we can define the matrix $\Sigma_{\infty}:= \lim_{k\rightarrow \infty}\Sigma_k$ which is the d-by-d principle minor of the matrix $\Xi_\infty$.} \bcred{Since} the iterates $\{x_k\}_{k\geq 0}$ admit normal distribution $\mathcal{N}(\mu_k,\Sigma_k)$, we can write,
\begin{equation*}
    \bcgreen{R_k}(\theta):=\mathbb{E}\left[\exp\{\frac{\theta}{2\sigma^2}(f(x_k)-f(x_*))\}\right]=n_k\int \exp\{\frac{\theta}{\bc{4}\sigma^2}(x-x_*)^\top Q(x-x_*)\}\exp\{ -\frac{1}{2} (x-\mu_k)^{\top}\Sigma_k^{-1}(x-\mu_k)\}dx,
\end{equation*}
where $n_k:=(2\pi)^{-\frac{d}{2}}det(\Sigma_k)^{-\frac{1}{2}}$. Notice that $$(x-x_*)^\top Q(x-x_*)=(x-\mu_k+\mu_k-x_*)^\top Q(x-\mu_k+\mu_k-x_*).$$
Hence we write 
\begin{align*}
\bcgreen{R_k}(\theta)&=n_k \int \Big(e^{ \frac{\theta}{\bc{4}\sigma^2}(x-\mu_k)^\top Q(x-\mu_k)+\frac{\bc{\theta}}{2\sigma^2}(x-\mu_k)^\top Q(\mu_k-x_*)}+e^{\frac{\theta}{\bc{4}\sigma^2}(\mu_k-x_*)^\top Q (\mu_k-x_*)-\frac{1}{2}(x-\mu_k)^\top \Sigma_k^{-1}(x-\mu_k)}\Big)dx,\\
&=n_ke^{\frac{\theta}{\bc{4}\sigma^2}(\mu_k-x_*)^\top Q(\mu_k-x_*)}\int e^{\frac{\bc{\theta}}{2\sigma^2}(x-\mu_k)^\top Q (\mu_k-x_*)-\frac{1}{2}(x-\mu_k)^\top [\Sigma_k^{-1}-\frac{\theta}{\bc{2}\sigma^2}Q](x-\mu_k)}dx,\\
&=n_k e^{\frac{\theta}{\bc{4}\sigma^2}(\mu_k-x_*)^\top Q(\mu_k-x_*)+\frac{\theta^2}{\bc{8}\sigma^4}(\mu_k-x_*)^\top Q^{1/2}M_k^{-1}Q^{1/2}(\mu_k-x_*)}\int e^{-\frac{1}{2}(x-v_k)^\top Q^{1/2}M_kQ^{1/2}(x-v_k)}dx,
\end{align*}
where $M_k=Q^{-1/2}\Sigma_k^{-1}Q^{-1/2}-\frac{\theta}{\bc{2}\sigma^2}I$ and $v_k=\mu_k+\frac{\theta}{\bc{2}\sigma^2}Q^{-1/2}M_k^{-1}Q^{1/2}(\mu_k-x_*)$. 
\bcgreen{We can see from the characterization of eigenvalues of $\lambda_i(Q^{1/2}\Sigma_\infty Q^{1/2})$ given at Lemma \ref{lem: sol-lyap-eq-quad} and the definition of $u_i$ that $Q^{1/2}\Sigma_\infty Q^{1/2} \prec \frac{2\sigma^2}{\theta}I_d$ for all $(\alpha,\beta,\gamma)\in\mathcal{F}_\theta$. Moreover, the covariance matrix $\Sigma_k$ satisfies $\Sigma_k \preceq \Sigma_\infty$ for each $k\geq 0$ which we prove later in Lemma \ref{lem: X-and-Y-pos-def}. Therefore, we also obtain the property that $Q^{1/2}\Sigma_kQ^{1/2}\prec \frac{2\sigma^2}{\theta}I_d$.} \bcgreen{Consequently, both} $M_k$ and $M_\infty$ are invertible.
The integral on the right-hand side of the last equality can be considered as a multidimensional Gaussian integral which can be computed as 
\bcgreen{
\begin{equation}\label{eq: gauss-integral}
    \int e^{-\frac{1}{2}(x-v_k)^\top Q^{1/2}M_kQ^{1/2}(x-v_k)}dx
    =\begin{cases} 
    (2\pi)^{d/2}\det(Q^{-1/2}M_k^{-1}Q^{-1/2})^{1/2}, &\text{if }  \lambda_{\min}\{Q^{-1/2}M_k^{-1}Q^{-1/2}\}>0,\\
    \infty, &\text{otherwise},
    \end{cases}
\end{equation}
\bcgreen{where $\lambda_{\min}$ denotes the smallest eigenvalue.} Since $(\alpha,\beta,\gamma)\in\mathcal{F}_\theta$, \bcgreen{it can be seen} that the integral \eqref{eq: gauss-integral} is finite \bcgreen{which} yields} 
\begin{equation*}
\mathbb{E}[e^{\frac{\theta}{2\sigma^2}(f(x_k)-f(x_*))}]=\det(I_d-\frac{\theta}{\bc{2}\sigma^2}\Sigma_k^{1/2}Q\Sigma_k^{1/2})^{-1/2} e^{\frac{\theta}{\bc{4}\sigma^2}(\mu_k-x_*)^\top \left[ Q+\frac{\theta}{\bc{2}\sigma^2}Q^{1/2}M_k^{-1}Q^{1/2}\right](\mu_k-x_*)}.
\end{equation*}
\bcgreen{Hence we can compute the entropic risk,}
\begin{equation*}
\mr_{k,\sigma^2}(\theta)=(\mu_k-x_*)^\top \left[ Q+\frac{\theta}{\bc{2}\sigma^2}Q^{1/2}M_k^{-1}Q^{1/2}\right](\mu_k-x_*)-\frac{\bc{2}\sigma^2}{\theta}\log \det(I_d-\frac{\theta}{\bc{2}\sigma^2}\Sigma_k^{1/2}Q\Sigma_k^{1/2}).
\end{equation*}
\bcgreen{In the following, we }\bc{adopt the notation $\mu_\infty :=\lim_{k\rightarrow \infty}\bcgreen{\mu_k}$. Since $\Sigma_\infty$ exists by \bcgreen{our arguments above}, we can see that the infinite horizon entropic risk measure is of the form 
$$
\mr_{\sigma^2}(\theta)=\frac{2\sigma^2}{\theta}\log\mE[\exp\{ \frac{\theta}{2\sigma^2} f(x_\infty)-f(x_*)\}],
$$
where $x_\infty \sim \mathcal{N}(\mu_\infty, \Sigma_\infty)$.
\bcgreen{Notice that we can compute $\mr_{\sigma^2}(\theta)$ by writing 
$$
R_\infty(\theta)=\mathbb{E}[\exp\{\frac{\theta}{2\sigma^2}f(x_\infty)-f(x_*)\}],
$$ 
for $x_\infty \sim \mathcal{N}(\mu_\infty,\Sigma_\infty)$. Then, by following the arguments we made for $\mr_{k,\sigma^2}(\theta)$, we obtain 
\begin{equation*}
R_\infty(\theta)= n_\infty e^{\frac{\theta}{\bc{4}\sigma^2}(\mu_\infty-x_*)^\top Q(\mu_\infty-x_*)+\frac{\theta^2}{\bc{8}\sigma^4}(\mu_\infty-x_*)^\top Q^{1/2}M_\infty^{-1}Q^{1/2}(\mu_\infty-x_*)}
\times\int e^{-\frac{1}{2}(x-v_\infty)^\top Q^{1/2}M_\infty Q^{1/2}(x-v_\infty)}dx,
\end{equation*}
where $n_\infty = \lim_{k\rightarrow \infty}n_k$ and $v_\infty=\lim_{k\rightarrow \infty}v_k$. If we substitute $v_k$ and $\Sigma_k$ with $v_\infty$ and $\Sigma_\infty$ at \eqref{eq: gauss-integral}, we can see from the eigenvalues of $Q^{1/2}\Sigma_\infty Q^{1/2}$ provided in Lemma \ref{lem: sol-lyap-eq-quad} that the infinite-horizon risk measure is finite if and only if $(\alpha,\beta,\gamma)\in\mathcal{F}_{\theta}$. After showing the existence of infinite-horizon risk measure, we are next going to show the linear convergence of $r_{k,\sigma^2}(\theta)$ to $r_{\sigma^2}(\theta)$. The difference between finite-horizon and infinite-horizon risk is}}
\begin{small}
\begin{multline*}
\mr_{k,\sigma^2}(\theta)-\mr_{\sigma^2}(\theta)=\frac{\bc{2}\sigma^2}{\theta}\log \det(I_d-\frac{\theta}{\bc{2}\sigma^2}\Sigma_\infty^{1/2}Q\Sigma_\infty^{1/2})-\frac{\bc{2}\sigma^2}{\theta}\log \det(I_d-\frac{\theta}{\bc{2}\sigma^2}\Sigma_k^{1/2}Q\Sigma_k^{1/2})\\
+(\mu_k-x_*)^\top \left[ Q+\frac{\theta}{\bc{2}\sigma^2}Q^{1/2}M_k^{-1}Q^{1/2} \right] (\mu_{k}-x_*)\\-(\mu_\infty-x_*)^\top \left[ Q+\frac{\theta}{\bc{2}\sigma^2}Q^{1/2}M_\infty^{-1}Q^{1/2} \right] (\mu_{\infty}-x_*).
\end{multline*}
\end{small}
\bc{Since $(\alpha,\beta,\gamma)\in \mathcal{S}_q$ by Lemma \ref{lem: feas-set-stab-set}, the equality $\mu_\infty=x_*$ directly follows from Lemma \ref{lem: non_asym_conv_quad_obj}. Adding and subtracting $\frac{\theta}{\bcred{2}\sigma^2}(\mu_k-x_*)^\top Q^{1/2}M_\infty^{-1}Q^{1/2} (\mu_{k}-x_*)$ \bcgreen{leads to},
\begin{small}
\begin{multline*}
\mr_{k,\sigma^2}(\theta)-\mr_{\sigma^2}(\theta)=\frac{\bcred{2}\sigma^2}{\theta}\log \det(I_d-\frac{\theta}{\bcred{2}\sigma^2}\Sigma_\infty^{1/2}Q\Sigma_\infty^{1/2})-\frac{\bcred{2}\sigma^2}{\theta}\log \det(I_d-\frac{\theta}{\bcred{2}\sigma^2}\Sigma_k^{1/2}Q\Sigma_k^{1/2})\\
+(\mu_k-x_*)^\top \left[ Q+\frac{\theta}{\bcred{2}\sigma^2}Q^{1/2}\left(M_k^{-1}-M_{\infty}^{-1}\right)Q^{1/2} \right] (\mu_{k}-x_*)\\
+\frac{\theta}{\bcred{2}\sigma^2}(\mu_k-\bcgreen{x}_*)^\top Q^{1/2}M_\infty^{-1}Q^{1/2} (\mu_k-x_*).
\end{multline*}
\end{small}
The matrices $Q, M_k$ and $M_\infty$ are all symmetric matrices. \bcgreen{From the} Cauchy-Schwarz inequality; 
\begin{multline}\label{ineq: risk_diff_bound}
    |\mr_{k,\sigma^2}(\theta)-\mr_{\sigma^2}(\theta)| \leq \frac{\bcred{2}\sigma^2}{\theta}|\log \det(I_d-\frac{\theta}{\bcred{2}\sigma^2}\Sigma_\infty^{1/2}Q\Sigma_\infty^{1/2})-\log \det(I_d-\frac{\theta}{\bcred{2}\sigma^2}\Sigma_k^{1/2}Q\Sigma_k^{1/2})|\\
    +\Vert Q+\frac{\theta}{\bcred{2}\sigma^2}Q^{1/2}M_\infty^{-1}Q^{1/2} \Vert_2 \Vert \mu_{k}-x_*\Vert^2\\+ \frac{\theta}{\bcred{2}\sigma^2}\Vert M_k^{-1}-M_\infty^{-1}\Vert_2\Vert Q^{1/2} (\mu_{k}-x_*)\Vert^2.
    \vspace{-0.2cm}
\end{multline}
Recall the following result from Lemma \ref{lem: non_asym_conv_quad_obj},
\begin{equation*}
\Vert \mathbb{E}[\bc{z}_k]-\bc{z}_* \Vert^2 \leq C_k^2 \rho(A_Q)^{2(k-1)} \Vert \bc{z}_0-\bc{z}_*\Vert^2,
\end{equation*}
where $C_k$ is as given in Lemma \ref{lem: non_asym_conv_quad_obj}. Since $\bc{z}_k^\top=[x_k^\top, x_{k-1}^{\top}]^\top$, we also obtain 
\begin{equation}\label{ineq: lin_conv_mean}
    \Vert \mu_k-x_*\Vert^2 \leq \Vert \mathbb{E}[\bc{z}_k]-\bc{z}_* \Vert^2\leq C_k^2\rho(A_Q)^{2(k-1)}\Vert \bc{z}_0-\bc{z}_*\Vert^2.
\end{equation}
On the other hand, the difference $\Xi_k-\Xi_\infty$ also satisfies the recursion,
\begin{equation}\label{eq: covar-recur}
\Xi_k-\Xi_\infty = A_Q \left( \Xi_{k-1}-\Xi_\infty\right) A_Q^\top.
\end{equation}
\bcgreen{We} \bc{introduce the notation  
$$
vec(X')=[X'_{11},X'_{21},..., X'_{d1},X'_{12}, ..., X'_{dd}],
$$
to denote the vectorization $vec(X')$ of a matrix $X' \in\mathbb{R}^{d\times d}$.
Under this notation, \eqref{eq: covar-recur} can be \bcgreen{rewritten},}
\begin{equation*}
vec(\Xi_k-\Xi_\infty)=vec\left(A_Q^k(\Xi_{0}-\Xi_\infty)(A_Q^\top)^k \right)= (A_Q^k \otimes A_Q^k) vec(\Xi_0-\Xi_\infty),
\end{equation*}
\bcgreen{where $\otimes$ denotes the Kronecker product.}
Hence,
\begin{equation}
\Vert \Xi_k-\Xi_\infty \Vert_F^2=\Vert vec(\Xi_k-\Xi_\infty)\Vert^2= \Vert A_Q^k\Vert_2^{4}\Vert vec(\Xi_{0}-\Xi_\infty)\Vert^2,
\end{equation}
where we used the fact that $\Vert A_Q^k\otimes A_Q^k\Vert_2^2=\Vert A_Q^k\Vert_2^4$ (see e.g. \cite{lancaster1972norms}). Notice that $\Vert \Sigma_k-\Sigma_\infty \Vert^2_F\bcgreen{\leq}\Vert \Xi_k-\Xi_\infty\Vert^2_F$ since $\bc{z}_k^\top=[x_k^\top,x_{k-1}^{\top}]$ and $\bc{z}_\infty=[x_\infty^\top,x_\infty^\top]^\top$. \bcgreen{Hence,}
\begin{equation}\label{ineq: lin_conv_variance}
    \Vert \Sigma_k-\Sigma_\infty\Vert^2_F\leq  \Vert A_Q^k\Vert^4_2 \Vert vec(\Xi_0-\Xi_\infty)\Vert^2_2.
\end{equation}
Now, we are going to provide the convergence rate for \bcgreen{the} log-determinant term on \eqref{ineq: risk_diff_bound}; \bc{but first,}  \bcgreen{we provide} \bc{the following property of the log-determinant:}
\begin{align}\label{ineq: log-det}
\log\det(X')-\log\det(Y')&=\log\det((Y')^{-1/2}X'(Y')^{-1/2}),\nonumber \\
&\leq \text{Tr}((Y')^{-1/2}X'(Y')^{-1/2}-I), \nonumber\\
&=\text{Tr}((Y')^{-1}(X'-Y'))=vec((Y')^{-1})^{\top}vec(X'-Y'),
\end{align}
for any \bcgreen{positive definite} matrices $X'$ and $Y'$. \bcred{In Lemma \ref{lem: X-and-Y-pos-def}, we show that $X$ and $Y$ are positive definite matrices; so that we can \bcgreen{apply} the equality \eqref{ineq: log-det} \bcgreen{with  $X=X'$ and $Y=Y'$}.

\begin{lemma}\label{lem: X-and-Y-pos-def} 
\bcgreen{Suppose $(\alpha,\beta,\gamma)\in\mathcal{F}_\theta$.} Then, the matrices $X=I-\frac{\theta}{2\sigma^2}\Sigma_k^{1/2}Q\Sigma_{k}^{1/2}$ and $Y=I-\frac{\theta}{2\sigma^2}\Sigma_\infty^{1/2}Q\Sigma_{\infty}^{1/2}$ are positive definite for all $k\in\mathbb{N}$. Moreover, \bc{$\Sigma_k \preceq \Sigma_\infty$ for all $k\in\mathbb{N}$.} 
\end{lemma}
\begin{proof}
The proof is \bcgreen{deferred to} Appendix \ref{app: X-and-Y-pos-def}
\end{proof}
}
\bcgreen{Lemma \ref{lem: X-and-Y-pos-def} implies that the matrices $X$ and $Y$ are positive-definite since $(\alpha,\beta,\gamma)\in \mathcal{F}_\theta$ by assumption of Proposition \ref{prop: quad_qisk_meas_convergence}; therefore, the inequality \eqref{ineq: log-det} holds for $X'=X$ and $Y'=Y$.} \bc{Notice that the matrices $\Sigma_k^{1/2}Q\Sigma_k^{1/2}$ and $\Sigma_{\infty}^{1/2}Q\Sigma_{\infty}^{1/2}$ share the same eigenvalues with $Q^{1/2}\Sigma_kQ^{1/2}$ and $Q^{1/2}\Sigma_\infty Q^{1/2}$, respectively. On the other hand, we also have the eigenvalue characterization of the log-determinant which implies $\log\det(X') = \sum_{i=1}^{d}\lambda_i(X')$ where $\lambda_i(X')$ is the $i$-th largest eigenvalue of the positive\bcgreen{-}definite matrix $X'$. Hence we can write 
\begin{equation*}
\log\det(X)-\log\det(Y) = \log \det \left(I- \frac{\theta}{2\sigma^2}Q^{1/2}\Sigma_kQ^{1/2}\right)-\log\det\left(I-\frac{\theta}{2\sigma^2} Q^{1/2}\Sigma_\infty Q^{1/2}\right),
\end{equation*}
and the following inequalities hold by applying the inequality \eqref{ineq: log-det} \bcgreen{for} $X'= I-\frac{\theta}{2\sigma^2}Q^{1/2}\Sigma_kQ^{1/2}$ and $Y'=I-\frac{\theta}{2\sigma^2}Q^{1/2}\Sigma_\infty Q^{1/2}$, and using Cauchy-Schwarz inequality
\begin{equation*}
\frac{\bcred{2}\sigma^2}{\theta}\left|\log\det(X)-\log\det(Y)\right|\leq  \Big\Vert vec((I-\frac{\theta}{\bcred{2}\sigma^2}Q^{1/2}\Sigma_{\infty}Q^{1/2})^{-1})\Big\Vert \Big\Vert vec\left(Q^{1/2}(\Sigma_k-\Sigma_\infty) Q^{1/2}\right)\Big\Vert.
\end{equation*}
The Frobenius norm satisfies the property $\Vert vec(X')\Vert^2= \Vert X' \Vert^2_F$ for any matrix $X'$, and we can also write $vec(Q^{1/2}(\Sigma_k-\Sigma_\infty)Q^{1/2})= (Q^{1/2}\otimes Q^{1/2})vec(\Sigma_k-\Sigma_\infty)$ using the Kronecker product. Therefore, \bcgreen{the inequality above} becomes 
\begin{align}
\frac{\bcred{2}\sigma^2}{\theta}\left|\log\det(X)-\log\det(Y)\right|& \leq \Big\Vert (I-\frac{\theta}{\bcred{2}\sigma^2}Q^{1/2}\Sigma_\infty Q^{1/2})^{-1}\Big\Vert_F \Big\Vert Q^{1/2}\otimes Q^{1/2}\Big\Vert_2 \Big\Vert vec(\Sigma_k-\Sigma_\infty)\Big\Vert\nonumber,\nonumber\\
& \leq  L\Big\Vert (I-\frac{\theta}{2\sigma^2}Q^{1/2}\Sigma_\infty Q^{1/2})\Big\Vert_F \Big\Vert \Sigma_k-\Sigma_\infty\Big\Vert_F,\label{ineq: another-helper}
\end{align}
where we used the fact that eigenvalues of $Q^{1/2}\otimes Q^{1/2}$ are $\lambda_i(Q)^{1/2}\lambda_j(Q)^{1/2}$ \bcgreen{for $i=1,2,...,d$ and $j=1,2,...,d$} where  $\lambda_i(Q)$ is the $i$-th largest eigenvalue of $Q$\footnote{See e.g. \cite{schacke2004kronecker} for details.}. \bcgreen{ Combining \eqref{ineq: another-helper} with the inequality \eqref{ineq: lin_conv_variance} yields that the log-determinant term of \eqref{ineq: risk_diff_bound} converges linearly}\bcgreen{, i.e.} 
\begin{align}\label{ineq: Lin_Conv_2}
   \frac{\bcred{2}\sigma^2}{\theta}\left|\log\det(X)-\log\det(Y)\right| \leq  L\Vert (I-\frac{\theta}{2\sigma^2}\Sigma_\infty^{1/2} Q \Sigma_{\infty}^{1/2})^{-1}\Vert_F \Vert A_Q^k\Vert_2^2 \Vert \Xi_0-\Xi_\infty\Vert_F.
\end{align}
}
}
Next, we are going to show that $M_k^{-1}$ converges linearly to $M_\infty^{-1}$. We can write the following equality for any invertible matrix $M$: 
\begin{small}
\begin{align*}
    I=(\frac{\theta}{\bcred{2}\sigma^2}I+M)(\frac{\theta}{\bcred{2}\sigma^2}I+M)^{-1}\implies M^{-1}=\frac{\theta}{\bcred{2}\sigma^2}M^{-1}(\frac{\theta}{\bcred{2}\sigma^2}I+M)^{-1}+(\frac{\theta}{\bcred{2}\sigma^2}I+M)^{-1}.
\end{align*}
\end{small}
Setting $M=M_k$ and $M=M_\infty$ yields the following relations,
\begin{align*}
    M_k^{-1}=\frac{\theta}{\bcred{2}\sigma^2}M_k^{-1}Q^{1/2}\Sigma_kQ^{1/2}+Q^{1/2}\Sigma_kQ^{1/2}\bcgreen{,}\\
    M_\infty^{-1}=\frac{\theta}{\bcred{2}\sigma^2}M_\infty^{-1}Q^{1/2}\Sigma_\infty Q^{1/2}+Q^{1/2}\Sigma_\infty Q^{1/2}\bcgreen{.}
\end{align*}
Therefore the difference $M_k^{-1}-M_\infty^{-1}$ is, 
$$
M_k^{-1}-M_\infty^{-1}= \frac{\theta}{2\sigma^2}M_k^{-1}Q^{1/2}\Sigma_kQ^{1/2}-\frac{\theta}{2\sigma^2}M_\infty^{-1} Q^{1/2}\Sigma_\infty Q^{1/2}+Q^{1/2}(\Sigma_k-\Sigma_\infty)Q^{1/2},
$$
Adding and subtracting \bcgreen{the term} $\frac{\theta}{\bcred{2}\sigma^2}M_\infty^{-1}Q^{1/2}\Sigma_k Q^{1/2}$ \bcgreen{leads to}
\begin{small}
\begin{align*}
    M_k^{-1}-M_\infty^{-1}= \frac{\theta}{2\sigma^2}(M_k^{-1}-M_\infty^{-1})Q^{1/2}\Sigma_kQ^{1/2}+\left( I+\frac{\theta}{\bcred{2}\sigma^2}M_\infty^{-1} \right) Q^{1/2}(\Sigma_k-\Sigma_\infty)Q^{1/2}.
\end{align*}
\end{small}
Hence, \bcgreen{solving this linear system for $M_k^{-1}$}, 
$$
M_{k}^{-1}-M_\infty^{-1}=\left(I+\frac{\theta}{2\sigma^2}M_\infty^{-1}\right)Q^{1/2}(\Sigma_k-\Sigma_\infty)Q^{1/2}\left( I-\frac{\theta}{2\sigma^2}Q^{1/2}\Sigma_kQ^{1/2}\right)^{-1}.
$$
Notice that the condition \bc{$\Sigma_k\preceq \Sigma_\infty $}  (from Lemma \ref{lem: X-and-Y-pos-def}) and $(\alpha,\beta,\gamma)\in\mathcal{F}_\theta$ imply \bc{$I-\frac{\theta}{\bcred{2}\sigma^2}Q^{1/2}\Sigma_k Q^{1/2}\succeq I-\frac{\theta}{2\sigma^2}Q^{1/2}\Sigma_\infty Q^{1/2}\succ 0$ }; hence, we get 
\begin{equation}\label{ineq: Lin_Conv_3}
\Vert M_{k}^{-1}-M_\infty \Vert_2 \leq D_\infty  \Vert \Sigma_k-\Sigma_\infty \Vert_2,
\end{equation}
where $D_\infty= \Vert A_Q\Vert_2\Vert I+\frac{\theta}{\bcred{2}\sigma^2}M_{\infty}^{-1}\Vert_2 \Vert I-\frac{\theta}{2\sigma^2}Q^{1/2}\Sigma_\infty Q^{1/2}\Vert^{-1}$.
\bcgreen{Combining everything together,} the bound \eqref{ineq: risk_diff_bound} with inequalities \eqref{ineq: lin_conv_mean}, \eqref{ineq: lin_conv_variance}, \eqref{ineq: Lin_Conv_2}, and \eqref{ineq: Lin_Conv_3} yields the desired result. 
\subsection{Proof of Proposition \ref{prop: quad-risk-meas-gauss-noise}}\label{app: quad-risk-meas}
\bcgreen{Consider the dynamical system \eqref{sys: TMM_dynmcal_quad_obs-1} that the transformed iterates $\bar{z}_k$ obey with the output vector $\bar{q}_k$ defined as,}
\begin{align}\label{sys: TMM_dynmcal_quad_obs}
\bar{\bc{z}}_{k+1}&= A_{\Lambda} \bar{\bc{z}}_k + B \bc{\veps_{k+1}}, \\ 
\bc{\bar{q}}_{k}&= S\bar{\bc{z}}_{k},
\end{align}
where $S=\frac{1}{\sqrt{2}} \left[\Lambda^{1/2}, 0_d \right]$. The entropic risk measure of \bcgreen{GMM} under standard Gaussian noise $\bc{\veps_{k+1}}\sim\mathcal{N}(0,\sigma^2I_d)$ is
\begin{equation*}
\mr_{\sigma^2}(\theta)= \frac{2\sigma^2}{\theta}\log \mathbb{E} \left[
e^{\frac{\theta}{2\sigma^2} [f(x_{\infty}) - f(x_*)]}
\right]=\frac{2\sigma^2}{\theta} 
 \log \mathbb{E} \left[
e^{\frac{\theta}{2\sigma^2} \bc{\bar{q}}_{\infty}^{\top}\bc{\bar{q}}_{\infty}}
\right],
\end{equation*}
where $\bc{\bar{q}}_{\infty}=\lim_{k\rightarrow \infty }\bc{\bar{q}}_k$ is a normally distributed random variable with mean \bcgreen{$0$} and the matrix \bc{$V_\infty:=\mathbb{E}[\bar{\bc{z}}_\infty\bar{\bc{z}}_\infty^\top]$} satisfies 
\begin{align*}
    \mathbb{E}[\bc{\bar{q}}_{\infty}\bc{\bar{q}}_{\infty}^{\top}] =\mathbb{E}[ S\bar{\bc{z}}_{\infty}\bar{\bc{z}}_{\infty}^{\top}S^{\top}]=S V_{\infty}S^{\top}.
\end{align*}
\bcgreen{Furthermore, the matrix} $V_{\infty}$ also satisfies the Lyapunov equation \eqref{eq: Quad_Lyapunov}; hence we can write 
\begin{equation*}
    \bcgreen{U}^\top V_{\infty}\bcgreen{U}= \bcgreen{U}^\top A_{\Lambda}\bcgreen{U} \left(\bcgreen{U}^\top  V_{\infty} \bcgreen{U} \right) \bcgreen{U}^\top A_{\Lambda}^\top\bcgreen{U}+\sigma^2 \bcgreen{U}^\top BB^\top \bcgreen{U},
\end{equation*}
where $\bcgreen{U}$ is the projection matrix given in \eqref{def: AQ-block-diag-form}\bc{which satisfies $\bcgreen{U}^\top \bcgreen{U}=I$. We \bcgreen{introduce} $\tilde{V}_\infty=\bcgreen{U}^\top V_\infty \bcgreen{U}$, $\tilde{B}=\bcgreen{U}^\top B$, and $\tilde{A}_\Lambda = \bcgreen{U}^\top A_\Lambda \bcgreen{U}$, and} rewrite the Lyapunov equation as 
\begin{equation}\label{eq: Lyap-tilde-V-sol}
\tilde{V}_{\infty} = \tilde{A}_{\Lambda}\tilde{V}_{\infty}\tilde{A}_{\Lambda}^{\top} + \sigma^2 \tilde{B}\tilde{B}^\top,
\end{equation}
where $\tilde{A}_{\Lambda}=\underset{i=1,..,d}{\mbox{Diag}}(\{M_i\})$ where 
$$ 
\bcgreen{
M_i :=\begin{bmatrix} (1+\beta) - \alpha (1+\gamma )\lambda_i  & -(\beta -\alpha\gamma \lambda_i) \\
    1 &  0 
    \end{bmatrix},}
$$
\bcgreen{and $\lambda_i:=\lambda_i(Q)$ drops the dependency of the eigenvalues to $Q$ for notational simplicity.} The solution $\tilde{V}_\infty=\underset{i=1,..,d}{\mbox{Diag}}\begin{bmatrix}\bar{x}_{\lambda_i} & \bar{y}_{\lambda_i}\\\bar{y}_{\lambda_i}& \bar{v}_{\lambda_i} \end{bmatrix}$ of the Lyapunov equation \eqref{eq: Lyap-tilde-V-sol} satisfies the system of equations:
\begin{align*}
    \bar{x}_{\lambda_i}&= c_{i}^2\bar{x}_{\lambda_i} + 2c_id_i \bar{y}_{\lambda_i} + d_i^2 \bar{v}_{\lambda_i} +\alpha^2\sigma^2,\\
    \bar{y}_{\lambda_i} &= c_i \bar{x}_{\lambda_i}+d_i\bar{y}_{\lambda_i},\\
    \bar{v}_{\lambda_i}&= \bar{x}_{\lambda_i},
\end{align*}
where $c_i=1+\beta-\alpha(1+\gamma)\lambda_i $ and $d_i=-\beta+\alpha\gamma\lambda_i$. The solution of \bcgreen{this} system is given as 
\begin{align*}
\bar{x}_{\lambda_i}= \frac{(1-d_i)\alpha^2 \sigma^2}{(1+d_{i})[(1-d_i)^2-c_i^2]},\;\;\bar{y}_{\lambda_i}=\frac{c_i\alpha^2 \sigma^2}{(1+d_{i})[(1-d_i)^2-c_i^2]}, \;\; \bar{v}_{\lambda_i}=\bar{x}_{\lambda_i},
\end{align*}
\bcgreen{From the definition of $\tilde{V}_\infty$, we have} $\lambda_{i}(SV_{\infty}S^\top)=\lambda_{i}(S\bcgreen{U}\tilde{V}_{\infty}\bcgreen{U}^\top S^\top)$ where the \bcgreen{latter} matrix can be written as
\bc{
\begin{align*}
    S\bcgreen{U}\tilde{V}_{\infty}\bcgreen{U}^\top S^\top&= \underset{i=1,..,d}{\mbox{Diag}}\left\{ \frac{1}{2}\lambda_i\bar{x}_{\lambda_i}\right\}.
\end{align*}

It directly follows that,
\begin{equation*} 
\lambda_{i}(SV_{\infty}S^{\top})=\frac{\lambda_i(1-d_i)\alpha^2\sigma^2}{2(1+d_i)[(1-d_i)^2-c_i^2].}
\end{equation*}
Recall that for all $(\alpha,\beta,\gamma)\in\mathcal{F}_\theta$, \bcgreen{by definition}, we already have the property, 
\begin{equation}\label{ineq: app-helper-1}
\bcgreen{\lambda_i(SV_\infty S^\top)=}
\frac{\lambda_i (1-d_i)\alpha^2 \sigma^2}{2(1+d_i)[(1-d_i)^2-c_i^2]}=\frac{\sigma^2}{2u_i} <\frac{\sigma^2}{\theta}.
\end{equation}
Therefore $\lambda_{\max}(SV_\infty S^\top)<\frac{\sigma^2}{\theta}$ from the condition, \bcgreen{where $\lambda_{\max}(.)$ denotes the largest eigenvalue}. 
On the other hand, $\bc{\bar{q}}_\infty$ is a Gaussian variable with covariance matrix $SV_\infty S^\top$; hence, it can be shown\footnote{\bcgreen{It suffices to choose} $m=0$, $\alpha^2(\lambda_P+L/2)=1$, and $\Sigma=SV_\infty S$ \bcgreen{in} \eqref{eq: Gauss-exp}.} using the proof of Lemma \ref{lem: Gaussian-expectation} that}
$$
\bc{
\mathbb{E}[e^{\frac{\theta}{2\sigma^2}\Vert \bc{\bar{q}}_\infty \Vert^2}]=\begin{cases}\det\left(I_d- \frac{\theta}{\sigma^2}SV_\infty S^\top\right)^{-1/2}, &\text{ if } \lambda_{\max}(SV_\infty S^\top)< \frac{\sigma^2}{\theta},\\ 
\infty, & \text{otherwise}.
\end{cases}}
$$
\bcgreen{where we used} $\lambda_{\max}(SV_\infty S^\top)< \frac{\sigma^2}{\theta}$ on $\mathcal{F}_{\theta}$, this yields the finite-horizon risk measure is finite if and only if $(\alpha,\beta,\gamma)\in\mathcal{F}_\theta$ and given as
\begin{equation} \label{eq: risk-meas-app-helper}
\mr_{\sigma^2}(\theta) = 
\frac{2\sigma^2}{\theta}\sum_{i=1}^{d} \log \left(1-\frac{\theta}{\sigma^2}\lambda_i \left(SV_{\infty}S^{\top} \right) \right)^{-1/2},
\end{equation}
\bcgreen{and we conclude from \eqref{ineq: app-helper-1}}.

\subsection{Proof of Theorem \ref{thm: quad-evar-bound}}\label{app: quad-evar-bound}
\bcgreen{Given $(\alpha,\beta,\gamma)\in \mathcal{S}_q$, by Proposition \ref{prop: quad-risk-meas-gauss-noise} and the definition of EV@R, }
\begin{equation*} 
EV@R_{1-\zeta}\bcred{[f(x_\infty)-f(x_*)]}=\bcgreen{\inf}_{\bcgreen{0<\theta}}  \;\mr_{\sigma^2}(\theta)+\frac{\sigma^2}{\theta}\log(1/\zeta)\bcgreen{\leq} \min_{0<\theta<2\bar{u}}\left\{ -\frac{\sigma^2}{\theta}\sum_{j=1}^{d}\log(1-\frac{\theta}{2u_i})+\frac{2\sigma^2}{\theta}\log(1/\zeta)\right\}.
\end{equation*}
Notice that \bcred{$\bar{u}\leq u_i$ \bcgreen{for every $i$ which} implies $-\log(1-\frac{\theta}{2u_i}) \leq -\log(1-\frac{\theta}{2\bar{u}})$}; therefore, we obtain \bcgreen{an} upper bound on the EV@R as follows: 

\begin{equation}\label{ineq: evar-min-quad}
\bcred{
EV@R_{1-\zeta}[f(x_\infty)-f(x_*)] \leq \min_{0<\theta<2\bar{u}}\left\{ h(\theta) \right\},
}
\end{equation}
\bcgreen{where} \bcred{ $h(\theta):=-\frac{\sigma^2d}{\theta}\log(1-\frac{\theta}{2\bar{u}})+\frac{2\sigma^2}{\theta}\log(1/\zeta)$} and its derivative $h'(\theta)$ is given as
\begin{align}\label{eq: h-derivative}
h'(\theta)&= \frac{\sigma^2d}{\theta^2}\left[ \log\left(1-\frac{\theta}{2 \bar{u}}\right)+\frac{2\bar{u}}{2\bar{u}-\theta}-1-\frac{2\log(1/\zeta)}{d}\right].
\end{align}
\bcgreen{A simple } explicit formula for $\theta_*$ \bcgreen{satisfying} $h'(\theta_*)=0$ does not \bcgreen{appear to} exist \bcgreen{which would be an optimizer of \eqref{ineq: evar-min-quad}}; therefore, we are going to approximate $h'(\theta)$ by using  $-\log(1-x)\approx x$ \bcgreen{for small $x$ to find an approximate root of \eqref{eq: h-derivative}:} 
$$
h'(\theta)\approx \tilde{h}'(\theta):= \frac{\sigma^2 d}{\theta^2}\left[-\frac{\theta}{2\bar{u}} + \frac{2 \bar{u}}{2\bar{u}-\theta}-(1+\frac{2\log(1/\zeta)}{d}) \right]. 
$$
The solution \bcgreen{$\tilde{ \theta}_0$ of $\tilde{ h}'(\tilde{\theta}_0)=0$ on $(0,2\bar{u})$ is the positive root of the polynomial}
$$
\theta^2+ \frac{4\bar{u}\log(1/\zeta)}{d}\theta-\frac{8\bar{u}^2\log(1/\zeta)}{d}=0,
$$
which is given as 
$$
\tilde{ \theta}_0= \frac{2\bar{u}\log(1/\zeta)}{d}\left[-1+\sqrt{1+\frac{2d}{\log(1/\zeta)}}\right].
$$
\bcred{Notice the following relation holds
\begin{align*}
    \sqrt{1+\frac{2d}{\log(1/\zeta)}}< 1+\frac{d}{\log(1/\zeta)} \implies \frac{\log(1/\zeta)}{d}\left[ -1+\sqrt{1+\frac{2d}{\log(1/\zeta)}}\right]<1.
\end{align*}
\bcgreen{Therefore $\tilde{\theta}_0=\theta_02\bar{u}< 2\bar{u}$}\bcgreen{; hence,}} we conclude from \eqref{ineq: evar-min-quad} that  \bcgreen{  $EV@R_{1-\zeta}\leq h(\tilde{\theta_0})=h(\theta_02\bar{u}) $} and obtain the inequality \eqref{eq: Approx_EVAR_Quad}.

{
\section{Proofs of Section \ref{sec: strongly-conv-obj}}\label{sec-app-b}}\bcgreen{We first provide a result from the literature about deterministic \bcgreen{GMM}.}
\begin{lemma}[\cite{hu2017dissipativity}]\label{lem: result-of-hu-lessard}Let $f\in\Sml$, and $\{\bc{z}_k\}_{k\in \mathbb{N}}$ be the iterates generated by the deterministic \bcred{\bcgreen{generalized} momentum method}, i.e. \eqref{Sys: TMM} with $\bc{\veps_{k}}=0$ \bcgreen{for every $k$}. Suppose there exists a symmetric positive semi-definite matrix $\tilde{P}$ and a rate $\rho<1$ for which the following matrix inequality holds 
\begin{equation}\label{MI}\tag{MI} 
\left(
\begin{array}{cc}
\tilde{A}^{\top}\tilde{P}\tilde{A}-\rho^2 \tilde{P} & \tilde{A}^{\top}\tilde{P}\tilde{B}
\\
\tilde{B}^{\top}\tilde{P}\tilde{A} & \tilde{B}^{\top}\tilde{P}\tilde{B}
\end{array}
\right)
-\tilde{X}\preceq 0,
\end{equation}
where $\tilde{X}=\tilde{X}_{1} +\rho^2 \tilde{X}_2 +(1-\rho^2)\tilde{X}_3$ with
\begin{align*}
    &\tilde{X}_1:=
    \frac{1}{2}
    \begin{bmatrix} 
    -L\delta^2 & L\delta^2 & -(1-\alpha L)\delta \\
    L\delta^2 & -L\delta^2 & (1-\alpha L)\delta \\
    -(1-\alpha L)\delta & (1-\alpha L)\delta & \alpha(2-\alpha L)
    \end{bmatrix},\;\;
    \tilde{X}_2:=\frac{1}{2}
    \begin{bmatrix} 
    \gamma^2\mu  & -\gamma^2\mu & -\gamma \\
    -\gamma^2\mu & \gamma^2\mu  &  \gamma \\
    -\gamma      & \gamma       &  0      
    \end{bmatrix},\\
    &\qquad\qquad \quad \quad \tilde{X}_3:=\frac{1}{2}
    \begin{bmatrix} 
    (1+\gamma)^2\mu  & -\gamma(1+\gamma)\mu & -(1+\gamma) \\
    -\gamma(1+\gamma)\mu & \gamma^2\mu      & \gamma      \\ 
    -(1+\gamma)          & \gamma           &0
    \end{bmatrix},
\end{align*}
and $\delta:=\beta-\gamma$. Then, the Lyapunov function $\mathcal{V}_P$ defined as \eqref{def: Lyapunov} with $P=\tilde{P}\otimes I_d$ decays with rate $\rho$ along the trajectories of $\{\bc{z}_k\}_{k\in\mathbb{N}}$, 
\begin{equation}\label{ineq: lyapunov-contraction}
\mV_P(\bc{z}_{k+1}) \leq \rho^{2} \mV_P(\bc{z}_k).
\end{equation}
\end{lemma}
\begin{proof}The result directly follows from \bcgreen{combining} \cite{hu2017dissipativity}[Lemma 5] \bcgreen{and} \cite{hu2017dissipativity}[Theorem 2].
\end{proof}
\begin{lemma}\label{lem: proof-of-non-empty-set}
The set \bcgreen{$\mathcal{S}_c$ defined in Theorem \ref{thm: TMM-MI-solution} is not empty. In particular,
$$\left\{(\vartheta,\psi) : \max\Big\{ \frac{1}{\kappa} \sqrt{\frac{\kappa^2}{4}+\kappa}-\frac{\kappa}{2}\leq \vartheta=\psi < 1\Big\}\right\}\subset \mathcal{S}_c.$$}
\end{lemma}
\begin{proof}
\bcgreen{Assume}
$\max\{\frac{1}{\kappa},\sqrt{\frac{\kappa^2}{4}+\kappa}-\frac{\kappa}{2}\}\leq \vartheta=\psi <1$. This would \bcgreen{imply} $\alpha=\frac{1}{L}$ (from the equality \eqref{cond: str-cnvx-alpha}) and we can write
\begin{align*}
\left[1-\sqrt{\frac{\psi}{\kappa}} \right]\left[1-\frac{\mu\psi^2-L(1-\psi)^2}{L\psi}\right]-\left(1-\frac{\psi}{\kappa} \right)^2&=\left[1-\sqrt{\frac{\psi}{\kappa}}\right]\left[1-\frac{\psi}{\kappa}+\frac{(1-\psi)^2}{\psi} -\left(1+\sqrt{\frac{\psi}{\kappa}}\right)\left(1-\frac{\psi}{\kappa}\right)\right],\\
& =\left[1-\sqrt{\frac{\psi}{\kappa}}\right]\left[\frac{(1-\psi)^2}{\psi} -\sqrt{\frac{\psi}{\kappa}}\left(1-\frac{\psi}{\kappa}\right)\right].
\end{align*}
From the fact that $\sqrt{\frac{\psi}{\kappa}}\geq \frac{\psi}{\kappa}$ and $\frac{\psi^2}{\kappa^2}\leq \frac{\psi^2}{\kappa}$, we get 
\begin{small}
\begin{eqnarray*}
\left[1-\sqrt{\frac{\psi}{\kappa}} \right]\left[1-\frac{\mu\psi^2-L(1-\psi)^2}{L\psi}\right]-\left(1-\frac{\psi}{\kappa} \right)^2\leq \left[1-\sqrt{\frac{\psi}{\kappa}} \right]\left[\frac{1-\psi}{\psi}\right]\left(1-\psi -\frac{\psi^2}{\kappa}\right).
\end{eqnarray*}
\end{small}
The condition $\sqrt{\frac{\kappa^2}{4}+\kappa}-\frac{\kappa}{2}\leq\psi$ implies $1-\psi-\frac{\psi^2}{\kappa}\leq 0$ showing that \bcgreen{$(\vartheta,\psi)\in \mathcal{S}_1$}:
\bcgreen{
$$
\left[1-\sqrt{\frac{(1-\vartheta)\vartheta}{\kappa(1-\psi)}}\right]\left[1-\frac{(1-\vartheta)(\mu\psi^2-L(1-\psi)^2)}{L(1-\psi)\vartheta}\right] \leq \left(1-\frac{(1-\vartheta)\psi}{\kappa(1-\psi)}\right)^2.
$$
}
Moreover, all $\frac{1}{\kappa}\geq \psi$ satisfy $\max\{2-\frac{1}{\psi}, \frac{1}{1+\kappa(1-\psi)}\}\leq\psi$ which proves that $\psi\in \mathcal{S}_{-}$. Hence, we conclude that 
$$
\left\{(\vartheta,\psi) ~\vline~ \max\left\{\frac{1}{\kappa},\sqrt{\frac{\kappa^2}{4}+\kappa}-\frac{\kappa}{2}\right\}\leq \vartheta=\psi <1\right\}\bcgreen{\subseteq (\mathcal{S}_{-} \cap \mathcal{S}_1) \subseteq \mathcal{S}_c.}
$$
\end{proof}

\subsection{Proof of Theorem \ref{thm: TMM-MI-solution}}
\bcred{The proof is based on \bcgreen{providing explicit solutions to} \eqref{MI}. Let us first show how \eqref{MI} can be used to derive an upper bound on the suboptimality, $f(x_k)-f(x_*)$, in Lemma \ref{lem: func-contr-prop}.
\begin{lemma}\label{lem: func-contr-prop} Consider the dynamical system \eqref{Sys: TMM} defined by \bcgreen{noisy} \bcgreen{GMM}. 
Suppose $(\alpha,\beta,\gamma)$ satisfies \eqref{MI} for some rate $\rho \in (0,1)$, and a \bcgreen{positive semi-definite} matrix $P$, then \bcgreen{we have} the inequality 
\begin{equation}\label{ineq: Lyapunov-contraction-2}
f(x_{k+1})-f(x_*) \leq \mathcal{V}_P(\bc{z}_{k+1}) \leq \rho^2 \mV_P(\bc{z}_k)
+ \bc{\veps_{k+1}}^\top m_k+ \alpha^2(\frac{L}{2}+\lambda_P)\Vert \bc{\veps_{k+1}}\Vert^2,
\end{equation}
where \bc{$m_k=\alpha \delta L(x_{k-1}-x_k)-\alpha(1-\alpha L)\nabla f(y_k)+ 2B^\top P (A (\bc{z}_k-\bc{z}_*)+B\nabla f(y_k))$ and $\lambda_P$ is the largest eigenvalue of \bcgreen{the} $d$-by-$d$ principle submatrix of P.}
\end{lemma}
\begin{proof}
\bcgreen{The proof is provided in Appendix \ref{app: func-contr-prop}.}
\end{proof}
\bcgreen{In order to prove \eqref{ineq: exp-subopt-str-cnvx-gaus}; we will provide explicit choice of $\rho$ and $P$ that satisfy \eqref{MI} and then build on Lemma \ref{lem: func-contr-prop}. We consider the left hand-side of \eqref{MI} and introduce 
$$
\textbf{L}=\begin{bmatrix}A^\top P A-\rho^2 P & A^\top P B\\ 
B^\top P A & B^\top P B\end{bmatrix}-X, 
$$
where $A=\tilde{A}\otimes I_d$ and $B=\tilde{B}\otimes I_d$ with 
\begin{align*}
    \tilde{A}= \begin{bmatrix}
    (1+\beta) & -\beta \\ 
    1 & 0
    \end{bmatrix}, \;\; \tilde{B}=\begin{bmatrix}
    -\alpha \\ 
    0
    \end{bmatrix}.
\end{align*}
}
}
By definition of the matrices $A,B, P$ and $X$, we can see that $\mathbf{L}=\tilde{L}\otimes I_d$ and $\mathbf{L}\preceq 0$ if and only if $\tilde{L}\preceq 0$. Hence in the rest of the proof, we are going to find the parameters $(\alpha,\beta,\gamma)$ for which $\tilde{L}$ satisfies \bcgreen{the} Sylvester criterion for negative semi-definite matrices. First, we set $\gamma=\psi\beta$ and $\tilde{P}=\tilde{p}\tilde{p}^\top$ where $\tilde{p}^\top= [\tilde{p}_0, -\tmp_0+\tilde{p}_{1}]$ \bcgreen{with $\tilde{p}_0,\tilde{p}_1\in\mathbb{R}$ that we will specify later}. The elements at $i$-th row and $j$-th column ($\tilde{L}_{ij}$) of $\tilde{L}$ can be computed as
\begin{align*}
\tilde{L}_{11}&= (\beta^2-\rho^2)\tmp_0^2+2\beta \tmp_0\tmp_1 +\tmp_1^2 + \frac{L\delta^2-\gamma^2 \mu}{2}-\frac{(1-\rho^2)(1+2\gamma)\mu}{2},\\
\tilde{L}_{12}&=(\rho^2-\beta^2)\tmp_0^2-(\beta+\rho^2)\tmp_0\tmp_1-\frac{L\delta^2-\gamma^2\mu}{2}+\frac{(1-\rho^2)\gamma\mu}{2},\\
\tilde{L}_{22}&=(\beta^2-\rho^2)\tmp_0^2-\rho^2(\tmp_1^2-2\tmp_1 \tmp_0)+\frac{L\delta^2-\gamma^2\mu}{2},\\
\tilde{L}_{32}&=\alpha\beta \tmp_0^2-\frac{\beta(1-\alpha L)+\alpha\gamma L}{2},\\
\tilde{L}_{31}&= -\alpha\beta \tmp_0^2-\alpha \tmp_1\tmp_0 +\frac{\beta(1-\alpha L)+\alpha \gamma L +1-\rho^2}{2}= -\tilde{L}_{32}+\frac{1-\rho^2}{2}-\alpha \tmp_1 \tmp_0,\\
\tilde{L}_{33}&= \alpha^2\tmp_0^2-\frac{\alpha(2-\alpha L)}{2}.
\end{align*}
\bcgreen{Choosing $\tmp_1=\sqrt{\frac{\mu}{2}}$, $\tmp_0=\sqrt{\frac{\vartheta}{2\alpha}}$, $\rho^2=1-\sqrt{\alpha \vartheta \mu}$, implies $1-\rho^2=2\alpha \tmp_0\tmp_1$ and we get $\tilde{L}_{31}=-\tilde{L}_{32}$}. \bcgreen{Note that \bcgreen{this} corresponds to choosing $P$ according to \eqref{def: P_Matrix}.} \bcgreen{Under this particular choice of $\tmp_0, \tmp_1$, and $\rho$; if $\alpha,\vartheta$ and $\gamma$ satisfy the condition
\begin{equation}\label{cond: str-cnvx-alpha}
\alpha L (1-\psi)-1+\vartheta=0,
\end{equation}
then $\tilde{L}_{31}=-\tilde{L}_{32}=0$. If we set $\beta= \frac{\rho^2}{1-\alpha\psi\mu}\left[ 1-\sqrt{\frac{\alpha \mu}{\vartheta}}\right]$ for $\alpha, \vartheta,\psi>0$ satisfying the condition \eqref{cond: str-cnvx-alpha} and $\rho$ as given before; then we can see that $\tilde{L}_{11}=\tilde{L}_{22}=-\tilde{L}_{12}$. Overall, our choice of parameters converts the }
matrix $\tilde{\mathbf{L}}$ \bcgreen{into} the following form
\begin{equation*}
\tilde{\textbf{L}}=
\begin{bmatrix} 
-\tilde{L}_{0} & \tilde{L}_{0} & 0\\
\tilde{L}_{0} & -\tilde{L}_{0} & 0 \\
0      & 0      &  \tilde{L}_{1}
\end{bmatrix}, 
\end{equation*}
where
\begin{align*}
    \tilde{L}_0&= \rho^2 \left[\sqrt{\frac{\vartheta}{2\alpha}}-\sqrt{\frac{\mu}{2}}\right]^2-\beta^2\left[\frac{\vartheta}{2\alpha}+\frac{L(1-\psi)^2-\psi^2\mu}{2}\right],\\
    \tilde{L}_1&=- \frac{\alpha (1-\psi \alpha L)}{2}.
\end{align*}
Notice that \eqref{cond: str-cnvx-alpha} uniquely determines the value of $\alpha$ if $\psi\neq 1$ and choosing $\psi=1$ directly implies $\vartheta=1$, since $\alpha>0$. In the case of $\psi=1=\vartheta$, we can \bcgreen{have any $\alpha \in(0,1/L)$ and satisfy \eqref{cond: str-cnvx-alpha}.} Therefore, we will study these two cases separately.

\textbf{The case $\psi\neq 1$:} In this case, the condition \eqref{cond: str-cnvx-alpha} \bcgreen{yields} $\alpha=\frac{1-\vartheta}{L(1-\psi)}$. Hence the definitions of $\alpha, \beta, \rho$ and $\tilde{L}_1$ imply that if 
$$
(\vartheta,\psi) \in \bcgreen{\tilde{S}}:= \left\{(\vartheta,\psi) ~\mid~ 0< \frac{1-\vartheta}{1-\psi} \leq\bcgreen{ \min \{\kappa \vartheta, \frac{1}{\psi},\frac{\kappa}{\vartheta}\}}\;\&\; \psi\neq1 \right\},
$$
then the inequalities
$\tilde{L}_1<0$, $\beta> 0$,  $\rho > 0$, and $ \alpha>0$ are satisfied. We can decompose $\bcgreen{\tilde{S}}$ as  $\bcgreen{\tilde{S}}=S_-\cup S_{+}$ where $S_{-}=\bcgreen{\tilde{S}}\cap \{\bcgreen{(\vartheta,\psi)} ~\mid~ \psi<1\}$ and $S_{+}:=\bcgreen{\tilde{S}}\cap \{\bcgreen{(\vartheta,\psi)} ~\mid~ \psi>1 \}$ are given as 
\begin{align*}
\mathcal{S}_{-}&= \left\{(\vartheta,\psi) ~\mid~ 0\leq \psi<1\,\&\, \max\left\{2-\frac{1}{\psi},\frac{1}{1+\kappa(1-\psi)}\right\}\leq \vartheta < 1 \right\},\\
\mathcal{S}_{+}&=\left\{(\vartheta,\psi) ~\mid~ \psi>1\, \&\, 1<\vartheta \leq \min\left\{2-\frac{1}{\psi}, \frac{1}{2}+\frac{1}{2}\sqrt{1+4\kappa(\psi-1)} \right\} \right\}.
\end{align*}
Next, we will simplify the conditions \bcgreen{in} $\mathcal{S}_+$. Consider the function \bcgreen{$p(\psi)=\frac{1}{2}\sqrt{4\psi-3}-\frac{3}{2}+\frac{1}{\psi}$ for $\psi>1$}. Notice that $p(1)=0$ and also the derivative $p'$ satisfies 
$$
p'(\psi)= \frac{1}{\sqrt{4\psi-3}}-\frac{1}{\psi^2}>0.
$$
for all $\psi>1$. Moreover, for all $\psi>1$ we have 
\begin{footnotesize}
\begin{equation*}
\sqrt{1+4(\psi-1)} \leq \sqrt{1+4\kappa(\psi-1)} \implies \frac{1}{2}\sqrt{4\psi-3}-\frac{3}{2}+\frac{1}{\psi} \leq \frac{1}{2}+\sqrt{1+4\kappa (\psi-1)}-2+\frac{1}{\psi}.
\end{equation*}
\end{footnotesize}
Hence we obtain,
$$
0< p(\psi) < \frac{1}{2}+\frac{1}{2}\sqrt{1+4\kappa(\psi-1)}-2+\frac{1}{\psi},
$$
for $\psi>1$ and $\kappa>1$. This simplifies the set $\mathcal{S}_+$ to the following 
$$
\mathcal{S}_{+}=\left\{(\vartheta,\psi)~\mid~ \psi>1\,\&\, 1< \vartheta \leq 2-\frac{1}{\psi} \right\}.
$$
The inequality $\tilde{L}_{0}\geq 0$, on the other hand, holds if and only if 
$\beta^2\left[\frac{\vartheta}{2\alpha}+\frac{L(1-\psi)^2-\psi^2\mu}{2}\right]\leq\rho^2 \left[\sqrt{\frac{\vartheta}{2\alpha}}-\sqrt{\frac{\mu}{2}}\right]^2$, \mg{which is true if and only if}
\begin{equation}\label{ineq: cond-on-L0}
\rho^2\left[1-\alpha\frac{\mu\psi^2-L(1-\psi)^2}{\vartheta}\right]\leq (1-\alpha\psi\mu)^2.
\end{equation}
\bc{Notice that above inequality equals to the inequality that defines the set  $\mathcal{S}_1$ by the definitions of the parameters $\rho^2$ and $\alpha$.} \bcgreen{Hence we can see that for $(\vartheta,\psi)\in\mathcal{S}_1$, we have $\tilde{L}_0\geq 0$} Therefore the parameters $(\vartheta,\psi)$ belonging to set $(\mathcal{S}_{-}\cup \mathcal{S}_{+})\cap \mathcal{S}_1$ 
implies that the matrix $\tilde{\mathbf{L}}$ is negative semi-definite from the Sylvester theorem.

\textbf{The case $\psi=1$.} The condition \eqref{cond: str-cnvx-alpha} implies that $\vartheta=1$, as well. Hence we would obtain
\begin{align*}
    \rho^2=1-\sqrt{\alpha\mu},\;\beta=\frac{1-\sqrt{\alpha\mu}}{1+\sqrt{\alpha\mu}}=\gamma.
\end{align*}
This satisfies $\rho>0$, $\beta>0$, $\tilde{L}_1\leq0$, and
$
\tilde{L}_0= \frac{(1-\alpha\mu)\sqrt{\alpha\mu}(1-\sqrt{\alpha\mu})^2}{2\alpha(1+\sqrt{\alpha\mu})}>0,
$
for $\alpha\leq\frac{1}{L}$. \bcgreen{Therefore, if $(\vartheta,\psi)\in\mathcal{S}_0$, \eqref{MI} is satisfied with the parameters $(\alpha_{\vartheta,\psi},\beta_{\vartheta,\psi},\gamma_{\vartheta,\psi})$ and rate $\rho_{\vartheta,\psi}$.}

\bcred{The Sylvester theorem implies that the matrix $\textbf{L}$ with given choice of parameters $(\alpha,\beta,\gamma)$, rate $\rho$, and the matrix $P$ satisfies \eqref{MI}. Hence the assumptions of Lemma \ref{lem: func-contr-prop} are satisfied and we can write the inequality \eqref{ineq: Lyapunov-contraction-2} for the iterates $x_k$ generated by \bcgreen{noisy GMM} with \bcgreen{parameters} $(\alpha_{\vartheta,\psi},\beta_{\vartheta,\psi},\gamma_{\vartheta,\psi})$: 
\bcred{
$$
f(x_{k+1})-f(x_*) \leq \mathcal{V}_P(\bc{z}_{k+1})\leq \rho^2 \mathcal{V}_P(\bc{z}_k)+\bc{\veps_{k+1}}^\top m_k + \frac{\alpha(L\alpha + \vartheta)}{2} \Vert \bc{\veps_{k+1}}\Vert^2,
$$
\bcgreen{where we used $\lambda_{P}=\frac{\vartheta}{2\alpha}$}.
}
If we take the conditional expectation of the above inequality with respect to the natural filtration $\mathcal{F}_k$ generated by the iterates $\{x_j\}_{j=0}^k$, use the tower property of the conditional expectation with the facts that $\mE[\Vert \bc{\veps_{k+1}}\Vert^2 | \mathcal{F}_k]\leq \sigma^2d$ \bcgreen{and $\mE[w_{k+1}|\mathcal{F}_k]=0$} (from Assumption \ref{Assump: Noise}), and recurse over $k$, we obtain the desired result.
}

\subsection{Proof of Proposition \ref{prop: risk-meas-bound-gaussian}}
\bcred{We are going to prove Proposition \ref{prop: risk-meas-bound-gaussian} by providing an upper bound on the finite-horizon risk measure, $\mr_{k,\sigma^2}(\theta)$. \bcgreen{Our proof strategy is to use} the contraction property of \bcgreen{the} Lyapunov function $\mV_P(\bc{z}_k)$ \bcgreen{and the fact that} $f(x_k)-f(x_*)\leq \mV_P(\bc{z}_k)$.}
\bc{For simplicity, \bcgreen{we} drop the \bcgreen{subscript} $P$ in $\mV_P$ and \bcgreen{instead} adopt $\mV$, \bcgreen{if it is clear from the context}. Similarly, \bcgreen{we} set $\alpha_{\vartheta,\psi}\rightarrow\alpha$, $\beta_{\vartheta,\psi} \rightarrow\beta$, $\gamma_{\vartheta,\psi}\rightarrow\gamma$, $\mathtt{v}_{\vartheta,\psi}\rightarrow\mathtt{v}$, $\bar{\rho}_{\vartheta,\psi}\rightarrow\bar{\rho}$, and $\rho_{\vartheta,\psi}\rightarrow \rho$. Recall that we have shown in the proof of Theorem \ref{thm: TMM-MI-solution} that this choice of parameters together with the matrix $P$ given in Theorem \ref{thm: TMM-MI-solution} satisfy \eqref{MI} \bcred{and the largest eigenvalue $\lambda_P$ of \bcgreen{the} $d$-by-$d$ principle submatrix of $P$ is $\frac{\vartheta}{2\alpha}$, that is $\lambda_P=\frac{\vartheta}{2\alpha}$}. This also implies the following inequality (see the inequality \eqref{ineq: lyap-ineq} \bcgreen{from} the Proof of Lemma \ref{lem: func-contr-prop}),
\begin{equation}\label{ineq: lyap-func-helper-1}
\mV(\bc{z}_{k+1})\leq \rho^{2}\mV(\bc{z}_k)+m_k^\top\bc{\veps_{k+1}}+\alpha^2(\frac{L}{2}+\lambda_P)\Vert\bc{\veps_{k+1}}\Vert^2,
\end{equation}
where $\lambda_P$ is the largest eigenvalue of $d$-by-$d$ principle submatrix of $P$ and $m_k$ is as given in Lemma \ref{lem: func-contr-prop}.} \bcgreen{We next} \bcred{show that the inequality \eqref{ineq: lyap-func-helper-1}, in fact, provides an upper bound on \bcgreen{the} finite-horizon \bcgreen{entropic} risk. By Assumption \ref{Assump: Noise}, the noise $\bc{\veps_{k+1}}=\tilde{\nabla}f(y_k)-\nabla f(y_k)$ is independent from \bcgreen{the} sigma-algebra, $\mathcal{F}_k$, generated by  $\{x_j\}_{j=0}^{k}$; therefore, by the tower property of conditional \bcgreen{expectations} and the inequality \eqref{ineq: lyap-func-helper-1},
\begin{align}
\mathbb{E}\left[e^{\frac{\theta}{2\sigma^2} \mV(\bc{z}_{k+1})}\right]&\leq \mathbb{E}\left[e^{\frac{\theta}{2\sigma^2}\left(\rho^2 \mV(\bc{z}_k)+ m_k^\top \bc{\veps_{k+1}} + \alpha^2(\frac{L}{2}+\lambda_P)\Vert \bc{\veps_{k+1}}\Vert^2\right)}\right],\nonumber\\
&=\mE\left[e^{\frac{\theta}{2\sigma^2}\rho^2 \mV(\bc{z}_k)}\mE\left[ e^{\frac{\theta}{2\sigma^2}\left(m_k^\top \bc{\veps_{k+1}}+ \alpha^2 (\frac{L}{2}+\lambda_P)\Vert \bc{\veps_{k+1}}\Vert^2\right)} ~\vline~ \mathcal{F}_k\right] \right].\label{eq: cond-expectation}
\end{align}

Since the noise is Gaussian, we can explicitly compute the conditional expectation at the right-hand side of the equation \eqref{eq: cond-expectation} \bcgreen{as follows.}


\begin{lemma}\label{lem: Gaussian-expectation}
Let $\alpha, \lambda_P, \theta, L$ be non-negative constants and $w \in\mathbb{R}^d$ be a Gaussian random vector satisfying Assumption \ref{Assump: Noise}. Then, the inequality 
\begin{equation} 
\mathbb{E}[e^{\frac{\theta}{2\sigma^2}( \bcred{m}^\top \veps+ \alpha^2 \bcred{(\lambda_P+\frac{L}{2})}\Vert \veps\Vert^2)}]\bcgreen{=} \frac{1}{(1-\theta\alpha^2\bcred{(\lambda_P+L/2)})^{d/2}}e^{\frac{\theta^2 \Vert m\Vert^2}{\bcred{4}\sigma^2 (\bcred{2}-\theta \alpha^2 \bcred{(2\lambda_P+L)})}},
\end{equation} 
holds for any vector $m\in\mathbb{R}^d$ and $\theta$ satisfying \bcgreen{for all} $\theta \alpha^2 \bcred{(2\lambda_P+L)}< 2$.
\end{lemma}
\begin{proof}
\bcgreen{The proof is provided in \ref{app: Gaussian-expectation}.}
\end{proof}
It is easy to check that $\theta_u^g=\frac{2(1-\rho^2)}{8\alpha^2 \mathtt{v}+(1-\rho^2) \alpha^2(L+2\lambda_P)}$\bc{, since $\lambda_P=\frac{\vartheta}{2\alpha}$,} and \bcgreen{by} \eqref{cond: cond-on-theta-gaus}, 
\begin{equation}\label{ineq: theta-bound-2}
\theta < \frac{2(1-\rho^2)}{8\alpha^2 \mathtt{v}+(1-\rho^2) \alpha^2(L+2\lambda_P)}< \frac{2}{\alpha^2 (L+2\lambda_P)}.
\end{equation}
Hence, the assumptions of Lemma \ref{lem: Gaussian-expectation} hold under the setting of Proposition \ref{prop: risk-meas-bound-gaussian} and setting $m=m_k$ and $\lambda_P$ as the largest eigenvalue of the $d$-by-$d$ principle submatrix of $P$ in Lemma \ref{lem: Gaussian-expectation} \bcgreen{gives} the conditional expectation \eqref{eq: cond-expectation} as 
\begin{equation}\label{ineq: bound-on-cond-exp}
\mE\left[ e^{\frac{\theta}{2\sigma^2} \left( m_k^\top \bc{\veps_{k+1}} +\alpha^2 (\frac{L}{2}+\lambda_P)\Vert \bc{\veps_{k+1}}\Vert^2 \right)} ~\vline~ \mF_k \right] \bcgreen{=} \Big(1-\theta \alpha^2 \big(\frac{L}{2}+\lambda_P\big)\Big)^{-d/2} e^{\frac{\theta^2 \Vert m_k\Vert^2}{4\sigma^2 (2- \theta \alpha^2 (2\lambda_P+L))}}.
\end{equation}
Next, we present an upper bound on the norm of the vector $m_k$ in Lemma \ref{lem: bound-on-mk}.
\bcred{
\begin{lemma}\label{lem: bound-on-mk}
The vector $m_k$ given in Lemma \ref{lem: func-contr-prop} under the premise of Proposition \ref{prop: risk-meas-bound-gaussian} admits the following inequality
\begin{equation}\label{ineq: m_k-bound}
     \Vert m_k \Vert^2 \leq 8\alpha^2 \mathtt{v}_{\alpha,\beta,\gamma}\left(\mV_P(\bc{z}_k)+\mV_P(\bc{z}_{k-1})\right),
\end{equation}
where $\mathtt{v}_{\alpha,\beta,\gamma}=\frac{2L^2 }{\mu}(2\delta^2+(1-\alpha L)^2(1+2\gamma+2\gamma^2))+\lambda_P \rho^2$, \bcgreen{$\rho$ and $P$ are as in Lemma \ref{lem: func-contr-prop} and $\lambda_P=\frac{\vartheta}{2\alpha}$ is }the largest eigenvalue of $d$-by-$d$ principle submatrix of $P$.
\end{lemma}
}
\bcgreen{
\begin{proof}
The proof is provided in Appendix \ref{app: bound-on-mk}.
\end{proof}
}
Notice that  $f(x_k)-f(x_*)\leq \mV(\bc{z}_k)$; therefore, we can write 
\begin{equation}\label{eq: exp-bound-helper-1}
\mE[e^{\frac{\theta}{2\sigma^2}(f(x_k)-f(x_*))}]\leq \mE[e^{\frac{\theta}{2\sigma^2}\mV(\bc{z}_k)}].
\end{equation}
\bcgreen{It follows from \eqref{ineq: lyap-func-helper-1}, \eqref{eq: cond-expectation}, \eqref{ineq: bound-on-cond-exp}, \eqref{ineq: m_k-bound}, and \eqref{eq: exp-bound-helper-1} that}
\begin{equation}\label{ineq: mom-gen-f-bound}
\mE[e^{\frac{\theta}{2\sigma^2}(f(x_k)-f(x_*))}]\\
\leq \big(1-\theta\alpha^2 (\frac{L}{2}+\lambda_P)\big)^{-d/2} \mE\left[e^{\frac{\theta}{2\sigma^2}\big( \rho^2\mV(\bc{z}_{k-1})+ \frac{4\alpha^2\mtv \bcgreen{\theta}}{2-\theta\alpha^2(2\lambda_P+L)}(\mV(\bc{z}_{k-1})+\mV(\bc{z}_{k-2}))\big)}\right],
\end{equation}
which also provides an upper bound on the finite horizon entropic risk measure  \bcgreen{after taking logarithms.}
However, the inequality \eqref{ineq: mom-gen-f-bound} is not very useful due to its dependency on the iterates $\bc{z}_{k-1}$ and $\bc{z}_{k-2}$. Therefore, in the following Lemma \ref{lem: risk-meas-bound-gauss}, we further analyze \eqref{ineq: mom-gen-f-bound} and derive an upper bound on the moment generating function of $f(x_k)-f(x_*)$ which depends only on \bcgreen{the} initialization $\bc{z}_0$. 

\begin{lemma}\label{lem: risk-meas-bound-gauss}
Consider the setting of Proposition \ref{prop: risk-meas-bound-gaussian} where $\theta < \theta_u^g$. \bcgreen{Consider the iterates $z_k=\begin{bmatrix} x_k^\top, x_{k-1}^\top\end{bmatrix}^\top$ of the noisy \bcgreen{GMM}. We have}
\begin{equation}\label{ineq: risk-meas-bound-helper-2}
    \mE[e^{\frac{\theta}{2\sigma^2}(f(x_k)-f(x_*))}]\leq \prod_{j=0}^{k-1}\big(1-\theta \bar{\rho}^{2j}_{\vartheta,\psi}\alpha^2_{\vartheta,\psi}(\lambda_P+ \frac{L}{2})\big)^{-d/2} e^{\frac{\theta}{2\sigma^2}\bar{\rho}^{2k}_{\vartheta,\psi}2\mV(\bc{z}_0)},
\end{equation}
where $\lambda_P=\frac{\vartheta}{2\alpha_{\vartheta,\psi}}$.
\end{lemma}
\begin{proof}
The proof is deferred to Appendix \ref{app: risk-meas-bound-gauss}.
\end{proof}
We are now ready to derive a bound on \bcgreen{the} finite-horizon entropic risk measure using Lemma \ref{lem: risk-meas-bound-gauss}:
$$
\mr_{k,\sigma^2}(\theta) \leq \frac{2\sigma^2}{\theta}\log \mE\left[ e^{\frac{\theta}{2\sigma^2}\mV(\bc{z}_{k})}\right] \leq \bar{\rho}^{2 k}2\mV(\bc{z}_0)-\frac{\sigma^2 d}{\theta}\sum_{j=0}^{k-1}\log\Big(1-\theta \bar{\rho}^{2j}\alpha^2\big(\lambda_P+\frac{L}{2}\big)\Big).
$$
The fact that \bcgreen{$-\log(x)\leq 1-\frac{1}{x}$ for all $x\in[0,1]$ implies}
$$
-\frac{d\sigma^2}{\theta}\log\Big( 1-\theta \bar{\rho}^{2j} \alpha^2 \big( \frac{L}{2}+\lambda_P\big)\Big) \leq \frac{\sigma^2 d \bar{\rho}^{2j}\alpha^2\big(\frac{L}{2}+\lambda_P \big)}{1-\theta \bar{\rho}^{2j}\alpha^2 \big( \frac{L}{2}+\lambda_P\big)}< \frac{\sigma^2 d \bar{\rho}^{2j}\alpha^2 \big(\frac{L}{2}+\lambda_P\big)}{1-\theta \alpha^2 \big( \frac{L}{2}+\lambda_P\big)},
$$
where the last inequality follows from $\bar{\rho}<1$ (see \bcgreen{also} the derivation of \eqref{ineq: mtk-prop-2}). Therefore, we get
$$
\mr_{k,\sigma^2}(\theta) < \left(\frac{\sigma^2 d \alpha^2 \big(\frac{L}{2}+\lambda_P\big)}{1-\theta \alpha^2 \big( \frac{L}{2}+\lambda_P\big)} \right) \frac{1-\bar{\rho}^{2k}}{1-\bar{\rho}^2}+ \bcgreen{2}\bar{\rho}^{2k}\mV(\bc{z}_0).
$$
\bcgreen{Then the} inequality \eqref{ineq: fin-risk-meas-bound-gaus} follows from the fact that $\lambda_P=\frac{\vartheta}{2\alpha}$ and the definitions of the rest of parameters as given in Theorem \ref{thm: TMM-MI-solution}.
}

\subsection{Proof of Theorem \ref{thm: Evar-TMM-str-cnvx-bound}}
\bcred{The proof builds upon the results of Proposition \ref{prop: risk-meas-bound-gaussian}. \bcgreen{During the proof, to simplify the notation, we will drop the subscripts in $\alpha_{\vartheta,\psi}, \rho_{\vartheta,\psi}, \mtv_{\vartheta,\psi},\bar{\rho}_{\vartheta,\psi}$ and $\bbrho_{\vartheta,\psi}$.}
\bcgreen{We will also drop the subscript $P$ in $\mV_P$ for simplicity. R}ecall that $\lambda_P$ is the largest eigenvalue of $d$-by-$d$ principle submatrix of $P$ \bcgreen{given in Theorem \ref{thm: TMM-MI-solution}} and \bcgreen{is} given as $\lambda_P=\frac{\vartheta}{2\alpha}$ . The inequality \eqref{ineq: fin-risk-meas-bound-gaus} implies that the EV@R of $f(x_k)-f(x_*)$ at each iteration $k$ satisfies the inequality,} 
\begin{align}
    EV@R_{1-\zeta}\left[f(x_k)-f(x_*)\right]&=\inf_{0<\theta} \mr_{k,\sigma^2}(\theta)+\frac{2\sigma^2}{\theta}\log(1/\zeta), \nonumber\\
    &\leq \inf_{0<\theta< \theta_u^g} \mr_{k,\sigma^2}(\theta)+ \frac{2\sigma^2}{\theta}\log(1/\zeta),\nonumber\\
    &\bc{\leq } \inf_{0<\theta<\theta_u^g}\frac{\sigma^2 d \alpha \big(\vartheta + \alpha L\big)}{(1-\bar{\rho}^2)(2-\theta \alpha \big( \vartheta+\alpha L\big))}+ \frac{2\sigma^2}{\theta}\log(1/\zeta) + \bar{\rho}^{2k} 2\mV(\bc{z}_0).\label{ineq: app-evar-bound-helper-1}
    \end{align}
Notice that $\bar{\rho}$ also depends on $\theta$, in particular, \bcgreen{if $\theta=\theta_u^g$ then we would have 
\begin{footnotesize}
\begin{align*}
&\theta = \frac{1-\rho^2}{8\alpha^2 \mtv+(1-\rho^2)\alpha^2(L+2\lambda_P)}\\
&\; \iff 8\alpha^2\mtv \theta = 2(1-\rho^2)-\theta(1-\rho^2)\alpha^2(L+2\lambda_P),\\
&\; \iff \frac{\theta 16\alpha^2 \mtv}{2-\theta \alpha^2(L+2\lambda_P)}=4(1-\rho^2)-\frac{\theta 16\alpha^2 \mtv}{2-\theta \alpha^2(L+2\lambda_P)},\\
&\; \iff \left(\rho^2 + \frac{\theta 4\alpha^2 \mtv}{2-\theta \alpha^2(L+2\lambda_P)}\right)^2+ \frac{\theta 16\alpha^2 \mtv}{2-\theta \alpha^2(L+2\lambda_P)}=\left( 2-\rho^2 -\frac{\theta 4\alpha^2 \mtv}{2-\theta \alpha^2(L+2\lambda_P)}\right)^2,\\
&\; \iff \sqrt{\left(\rho^2 +\frac{\theta 4\alpha^2 \mtv}{2-\theta \alpha^2(L+2\lambda_P)}\right)^2+\frac{\theta 16\alpha^2 \mtv}{2-\theta \alpha^2(L+2\lambda_P)}}=2-\rho^2 -\frac{\theta 4\alpha^2 \mtv}{2-\theta \alpha^2(L+2\lambda_P)},
\end{align*}
\end{footnotesize}which implies that $\bar{\rho}^2=1$. Therefore, we require the strict inequality $\theta< \theta_u^g$ so that \bcgreen{right-hand side of } \eqref{ineq: app-evar-bound-helper-1} \bcgreen{stays finite}. On the other hand, $\bar{\rho}$ depends on $\theta$ and this makes solving \bcgreen{the} minimization problem in \eqref{ineq: app-evar-bound-helper-1} \bcgreen{explicitly} non-trivial. \bcgreen{As we shall discuss next,} we can overcome this issue by shrinking the constraint from $(0, \theta_u^g)$ to $(0,\theta_\varphi^g]$ where $\theta_{\varphi}^g=\varphi \theta_u^g$ for a fixed $\varphi\in (0,1)$. \bc{Recall the notations 
\begin{align*}
   \mtl = \frac{ \theta4\alpha^2 \mtv}{2-\theta \alpha^2(2\lambda_P+L)},\text{ and }\mtk=\frac{1}{2}\sqrt{(\rho^2+\mtl)^2+ 4\mtl}-\frac{(\rho^2+\mtl)}{2},
\end{align*}
from the proof of Lemma \ref{lem: risk-meas-bound-gauss} which also gives $\bar{\rho}^2=\rho^2+\mtk+\mtl$. Notice that  we have the following inequality  for any $\theta\leq\theta_\varphi^g$,
$$
\mtl = \frac{\theta 4 \alpha^2 \mtv}{2-\theta \alpha (\vartheta+\alpha L)} \leq \frac{\theta_\varphi^g 4\alpha^2 \mtv}{2-\theta_\varphi^g \alpha(\vartheta+\alpha L)}=l_{\theta_{\varphi}^g},
$$
where we also used the fact that $\lambda_P=\frac{\vartheta}{2\alpha}$.
Since $l_{\theta}>0$, we can also write 
$$
\bar{\rho}^2 = \frac{1}{2}(\rho^2 + \mtl)+\frac{1}{2}\sqrt{(\rho^2+\mtl)^2+4\mtl}\bcgreen{\leq}\frac{1}{2}(\rho^2 + l_{\theta_\varphi^g})+\frac{1}{2}\sqrt{(\rho^2+l_{\theta_\varphi^g})^2+4l_{\theta_\varphi^g}}=\bbrho, 
$$
for all $\theta\in (0,\theta_\varphi^g]$.
}
Therefore, we can convert the inequality \eqref{ineq: app-evar-bound-helper-1} into 
\begin{footnotesize}
$$
EV@R_{1-\zeta}[f(x_k)-f(x_*)] < \min_{0<\theta \leq \theta_{\varphi}^{g}} \frac{\sigma^2 d \alpha (\vartheta +\alpha L)}{(1-\bbrho)(2-\theta \alpha (\vartheta+\alpha L))} + \frac{2\sigma^2}{\theta}\log(1/\zeta)+ (\bbrho)^{k}2\mV(\bc{z}_0).
$$
\end{footnotesize}
\bcgreen{It can be seen that}
$$
\bcgreen{\bar{h}}(\theta)=  \frac{\sigma^2 d\alpha (\vartheta + \alpha L)}{(1-\bbrho)(2-\theta \alpha (\vartheta+\alpha L))} +\frac{2\sigma^2}{\theta}\log(1/\zeta)+ (\bar{\bar \rho}^{\varphi})^{k}2\mV(\bc{z}_0),
$$
is \bc{convex} on $(0, \theta_u^g)$ for any $(\vartheta,\psi)$ as given in Theorem \ref{thm: TMM-MI-solution} (recall that $\alpha$ and $\bbrho$ also depend on $(\vartheta,\psi)$). 
Let us introduce the notations
\begin{align*}
    \mtc_0= \frac{d\alpha (\vartheta+ \alpha L)}{1-\bbrho},\;
    \mtc_1=\alpha (\vartheta+\alpha L),
\end{align*} 
so that the function $\bcgreen{\bar{h}}(\theta)$ simplifies to 
$$
\bcgreen{\bar{h}}(\theta)= \frac{\sigma^2\mtc_0}{2- \theta \mtc_1}+ \frac{2\sigma^2}{\theta}\log(1/\zeta) + (\bbrho)^{k}2\mV(\bc{z}_0),
$$
where its derivative, $\bcgreen{\bar{h}}'(\theta)$, is given as 
\begin{align}\label{eq: h'}
\bcgreen{\bar{h}}'(\theta)&= \frac{\sigma^2\mtc_0\mtc_1}{(2-\theta\mtc_1)^2}-\frac{2\sigma^2}{\theta^2}\log(1/\zeta) = \frac{\sigma^2 c_0c_1 \theta^2 - 2\sigma^2 (2-\theta \mtc_1)^2 \log(1/\zeta)}{\theta^2 (2-\theta\mtc_1)^2}.
\end{align}
Notice that we have
$$
\theta_u^g=\frac{2(1-\rho^2)}{8\alpha^2 \mathtt{v}+ \alpha (1-\rho^2)(\vartheta+ \alpha L)}< \frac{2}{\alpha (\vartheta+\alpha L)}=\frac{2}{\mtc_1},
$$
by the definition of $\theta_u^g$. Therefore, we have $\theta < \frac{2}{\mtc_1}$ under the condition \eqref{cond: cond-on-theta-gaus}. Hence, we can see that polynomial at the numerator in \eqref{eq: h'} has only one root $\theta_*$ on $[0, \frac{2}{\mtc_1}]$ which is given \bcgreen{by} 
$$
\theta_* = \frac{2}{ \mtc_1+\sqrt{\frac{\mtc_0\mtc_1}{2\log(1/\zeta)}}}<\frac{2}{\mtc_1}.
$$
Moreover, we can also observe that
\begin{align*}
\mtc_0\mtc_1\theta^2 -2(2-\theta \mtc_1)^2 \log(1/\zeta)
\leq 0, &\text{ if } 0\leq \theta \leq \theta_*,\\ 
\mtc_0\mtc_1\theta^2 -2 (2-\theta \mtc_1)^2 \log(1/\zeta)>0,  &\text{ if } \theta_* < \theta \leq \frac{2}{\mtc_1}.
\end{align*}
This shows us that the function $h$ is decreasing on $(0, \theta_*]$
and increasing on $(\theta_*,\frac{2}{\mtc_1})$; therefore, it is locally convex on the region $(0,\frac{2}{\mtc_1})$. So far, we obtain the following inequality
\begin{equation*}
EV@R_{1-\zeta} [f(x_k)-f(x_*)] \bc{\leq} \min_{0< \theta \leq \theta^g_\varphi}\underbrace{ \frac{\sigma^2 d\alpha (\vartheta + \alpha L)}{(1-\bbrho)(2-\theta \alpha (\vartheta+\alpha L))} +\frac{2\sigma^2}{\theta}\log(1/\zeta)+(\bbrho)^{k}2\mV(\bc{z}_0)}_{\bcgreen{\bar{h}(\theta)}},
\end{equation*}
where $\bcgreen{\bar{h}}(\theta)$ is \bc{convex} on $(0,\theta_{\varphi}^g]$ and $\varphi$ is a constant between $0$ and $1$. Therefore, the minimization problem on the right-hand side of the inequality above can be solved: 
\begin{equation}
\label{ineq: evar-helper}
EV@R_{1-\zeta}[f(x_k)-f(x_*)] \bc{\leq} \begin{cases} 
\bcgreen{\bar{h}}(\theta_*), & \text{ if } \theta_* \bcgreen{\leq} \theta_\varphi^g, \\
\bcgreen{\bar{h}}(\theta_\varphi^g), & \text{ otherwise }.
\end{cases}
\end{equation}

Notice that for any fixed constant $\varphi \in (0,1)$, the inequality \eqref{cond: conf-lev-gauss} implies
\begin{align*} 
\log(1/\zeta) &\leq \frac{d}{2(1-\bbrho)}\left(\frac{ \theta_\varphi^g\alpha (\vartheta+\alpha L)}{2-\theta_{\varphi}^g \alpha (\vartheta +\alpha L)}\right)^2 = \frac{\mtc_0\mtc_1}{2}\left( \frac{\theta_\varphi^g}{2-\theta_\varphi^g\alpha (\vartheta+\alpha L)}\right)^{2}\\
&\implies \frac{2-\theta_\varphi^g\alpha (\vartheta+\alpha L)}{\theta_\varphi^g} \leq \sqrt{\frac{\mtc_0\mtc_1}{2\log(1/\zeta)}} \\
&\implies \frac{2}{\theta_{\varphi}^g}\leq c_1+\sqrt{\frac{c_0 c_1}{2\log(1/\zeta)}}
\end{align*}
Recall that  $\mtc_1=\alpha(\vartheta+\alpha L)$, hence above inequalities imply \bcgreen{that under condition \eqref{cond: conf-lev-gauss} we have  }
$$
\theta_*=\frac{2}{\mtc_1+\sqrt{\frac{\mtc_0\mtc_1}{2\log(1/\zeta)}}}\leq \theta_{\varphi}^g,
$$
and we have
\begin{small}
\begin{align*}
h(\theta_*)&=\frac{\sigma^2 \mtc_0}{2-\theta_* \mtc_1}+\frac{2\sigma^2}{\theta_*}\log(1/\zeta)+ \big(\bbrho \big)^k 2\mV(\bc{z}_0),\\
&= \frac{\sigma^2 \sqrt{\mtc_0}}{2}\Big( \sqrt{2\mtc_1\log(1/\zeta)}+\sqrt{\mtc_0}\Big)\\
&\quad\quad\quad\quad\quad+\frac{\sigma^2}{2}\sqrt{2\mtc_1\log(1/\zeta)}\Big(\sqrt{2\mtc_1\log(1/\zeta)}+\sqrt{\mtc_0}\Big)+\big(\bbrho\big)^k 2\mV(\bc{z}_0),\\
&= \frac{\sigma^2}{2}\Big(\sqrt{2\mtc_1\log(1/\zeta)}+\sqrt{\mtc_0} \Big)^{2}+(\bbrho)^k2\mV(\bc{z}_0).
\end{align*}
\end{small}

\bcgreen{The} computations \bcgreen{above} show that \eqref{ineq: evar-helper} \bcgreen{is equivalent to} \eqref{ineq: fin-EVAR-bound-gauss-1} by the definitions of $\mtc_1$ and $\mtc_0$ \bcgreen{if} the condition \eqref{cond: conf-lev-gauss} \bcgreen{holds} and, similarly, the inequality \eqref{ineq: fin-EVAR-bound-gauss-2} directly follows from writing $h(\theta_{\varphi}^g)$ explicitly for the case $\theta_{\varphi}^g>\theta_{*}$ in \eqref{ineq: evar-helper}.
Lastly, we can directly derive the inequality \eqref{ineq: EVAR-bound-gauss} using our results so far. \bcgreen{In particular}, taking the limit \bcgreen{superior of} the inequality \eqref{ineq: fin-risk-meas-bound-gaus} gives us
\begin{align*}
    \underset{k\rightarrow\infty}{\lim\sup}\; \mr_{k,\sigma^2}(\theta) &\leq \lim\sup_{k\rightarrow \infty} \frac{\sigma^2 d \alpha \big(\vartheta + \alpha L\big)}{(1-\bar{\rho}^2)(2-\theta \alpha \big( \vartheta+\alpha L\big))}+ \bar{\rho}^{2k}2\mV(\bc{z}_0),\\
    \mr_{\sigma^2}(\theta)&\leq \frac{\sigma^2 d \alpha \big(\vartheta + \alpha L\big)}{(1-\bar{\rho}^2)(2-\theta \alpha \big( \vartheta+\alpha L\big))}.
\end{align*}
Therefore, the EV@R at $\zeta$-confidence level admits the bound
\begin{align*}
    \bcgreen{\underset{k\rightarrow\infty}{\lim\sup}}\; EV@R_{1-\zeta}[f(x_k)-f(x_*)]&= \inf_{0<\theta} \mr_{\sigma^2}(\theta)+\frac{2\sigma^2}{\theta}\log(1/\zeta),\\
    & \leq \min_{0<\theta\leq \theta_{\varphi}^g} \underbrace{\frac{\sigma^2 d\alpha (\vartheta + \alpha L)}{(1-\bbrho)(2-\theta \alpha (\vartheta+\alpha L))} +\frac{2\sigma^2}{\theta}\log(1/\zeta)}_{\tilde{h}(\theta)}.
\end{align*}
Notice that $\tilde{h}(\theta)=h(\theta)-(\bbrho)^{k}2\mV(\bc{z}_0)$ and share the same local minimum, i.e., $\tilde{h}'(\theta_*)=0$ as well; and $\tilde{h}(\theta)$ is also convex on $(0, \theta_\varphi^g]$. It follows from our discussion on $\theta_*$ that $\tilde{h}(\theta_*)=\min_{0<\theta\leq \theta_\varphi^g}\tilde{h}(\theta)$ which gives the desired inequality \eqref{ineq: EVAR-bound-gauss}.
}\bc{
\subsection{The Proof of Corollary \ref{cor: gaussian-tail-bound}}
We are going to show \eqref{ineq: tail-bound-1} using Chernoff bound for $f(x_k)-f(x_*)$ and implementing the inequality \eqref{ineq: fin-risk-meas-bound-gaus} provided at Proposition \ref{prop: risk-meas-bound-gaussian}. Particularly, Chernoff bound gives us
\begin{align*}
\mathbb{P}\{ f(x_k)-f(x_*)\geq t_\zeta \}\leq e^{-\frac{\theta t_\zeta}{2\sigma^2}}\mE[e^{\frac{\theta }{2\sigma^2}(f(x_k)-f(x_*))}].
\end{align*}
Since the function $\log(x)$ is increasing on $x>0$, we can convert above inequality into, 
\begin{align*}
\log(\mathbb{P}\{ f(x_k)-f(x_*)\geq t_\zeta \})\leq \log\mE[e^{\frac{\theta}{2\sigma^2}(f(x_k)-f(x_*))}]-\frac{\theta t_\zeta}{2\sigma^2}= \frac{\theta}{2\sigma^2}\mr_{k,\sigma^2}(\theta)-\frac{\theta t_\zeta}{2\sigma^2}.
\end{align*}
Moreover, we can see that the inequality \eqref{ineq: fin-risk-meas-bound-gaus} holds by the statement of Corollary \ref{cor: gaussian-tail-bound}. This yields the bound, 
$$
\log(\mathbb{P}\{f(x_k)-f(x_*)\geq t_{\zeta}\})\leq \frac{\theta d \alpha (\vartheta+\alpha L)}{2(1-\bar{\rho}^2)(2-\theta\alpha (\vartheta+\alpha L))}+ \frac{\theta \bar{\rho}^{2k}\mV_P(\bcred{z}_0)}{\sigma^2} - \frac{\theta t_{\zeta}}{2\sigma^2},
$$
where we drop $\vartheta,\psi$ dependency on the notations $\alpha_{\vartheta,\psi}$ and $ \bar{\rho}_{\vartheta,\psi}^2$ for simplicity. The result follows from taking the exponential power of both sides. The next inequality \eqref{ineq: tail-bound-2}, on the other hand, follows from \eqref{eq-tail-proba}, \eqref{ineq: fin-EVAR-bound-gauss-1}, and \eqref{ineq: fin-EVAR-bound-gauss-2}. Notice that the inequalities \eqref{ineq: fin-EVAR-bound-gauss-1} and \eqref{ineq: fin-EVAR-bound-gauss-2} imply 
$$
EV@R_{1-\zeta}[f(x_k)-f(x_*)] \leq \bar{E}_{1-\zeta}(\vartheta,\psi)+(\bbrho_{\vartheta,\psi})^k 2\mV_P(\bcred{z}_0),
$$
Let us denote $EV@R_{1-\zeta}[f(x_k)-f(x_*)]$ by $EV@R_{1-\zeta}$ with a slight abuse notation.Therefore, combining these inequalities with \eqref{eq-tail-proba} yields 
\begin{align*}
\zeta 
\geq \mathbb{P}\{ f(x_k)-f(x_*)\geq EV@R_{1-\zeta}\}\geq \mathbb{P}\{ f(x_k)-f(x_*)\geq \bar{E}_{1-\zeta}(\vartheta,\psi)+(\bbrho_{\vartheta,\psi})^k2\mV_P(\bcred{z}_0)\},
\end{align*}
\bcgreen{and we conclude.}
}
\subsection{Proof of Proposition  \ref{prop: gd-evar-bound-gauss}}

\bc{
Notice that the following inequality holds at each iterate $x_k$,
\begin{align*}
    \Vert x_{k+1}-x_*\Vert^2 &= \Vert x_{k}-x_*-\alpha \tn f(x_k) \Vert^2,\\
    &=\Vert x_k-x_*-\alpha \nabla f(x_k)\Vert^2\\
    &\quad\quad\quad-2\alpha (x_k-x_*-\alpha \nabla f(x_k))^\top \bc{\veps_{k+1}} + \alpha^2 \Vert \bc{\veps_{k+1}}\Vert^2.
\end{align*}
where $\bc{\veps_{k+1}}=\tilde{\nabla}f(y_k)-\nabla f(y_k)$. Let $\bcgreen{m_k^{GD}}=-2\alpha (x_k-x_*-\alpha \nabla f(x_k))$, then \bcgreen{we can write}, 
\begin{align}\label{ineq: app-gd-evar-helper-1}
    \Vert x_{k+1}-x_*\Vert^2 \leq \bcgreen{\rho_{GD}(\alpha)^2} \Vert x_k-x_*\Vert^2 +(\bcgreen{m_k^{GD}})^\top \bc{\veps_{k+1}} + \alpha^2  \Vert \bc{\veps_{k+1}}\Vert^2.
\end{align}
\bcgreen{where $\rho_{GD}(\alpha)=\max\{ |1-\alpha \mu|, |1-\alpha L|\}$, see e.g. \cite{lessard2016analysis}.} \bcgreen{For simplicity, we adopt the notation $\rho_{GD}:=\rho_{GD}(\alpha)$ in the rest of the proof.} Next, we will show the inequality \begin{equation}\label{ineq: gd-exp-bound-gauss}
\mE[e^{\frac{\theta L}{4\sigma^2 }\Vert x_{k+1}-x_*\Vert^2}] < \prod_{j=0}^{k}\left(1 -\theta\frac{\alpha^2 L}{2}\bar{\rho}_{GD}^{2j} \right)^{-d/2}e^{\frac{\theta L}{4\sigma^2}\bar{\rho}_{GD}^{2(k+1)}\Vert x_0-x_*\Vert^2},
\end{equation}
\bcred{where $\bar{\rho}_{GD}^2=\frac{2\rho^2_{GD}}{2-\theta \alpha^2 L}<1$} by our assumption on $\theta$. Since the inequality \eqref{ineq: app-gd-evar-helper-1} holds pointwise 
\begin{align*}
    \mE\left[e^{\frac{\theta L}{4\sigma^2}\Vert x_{k+1}-x_*\Vert^2}\right]&\leq \mE\left[e^{\frac{\theta L}{4\sigma^2}\left( \bcgreen{\rho^2_{GD}} \Vert x_{k}-x_*\Vert^2+(m_k^{GD})^\top \bc{\veps_{k+1}} + \alpha^2 \Vert \bc{\veps_{k+1}}\Vert^2\right)}\right],\\
   & =\mE\left[e^{ \frac{\theta L}{4\sigma^2}\bcgreen{\rho^2_{GD}} \Vert x_{k}-x_*\Vert^2} \mE\left[e^{
   \frac{\theta L}{4\sigma^2} \left( (\bcgreen{m_k^{GD}})^\top \bc{\veps_{k+1}} + \alpha^2 \Vert \bc{\veps_{k+1}}\Vert^2 \right)}~\vline~\mathcal{F}_k\right]\right].
\end{align*}
Notice that \bcred{$\theta<\frac{2(1-\rho^2_{GD})}{\alpha^2 L}< \frac{2}{\alpha^2 L}$} and the noise $\bc{\veps_{k+1}}$ is independent from $\mathcal{F}_k$ by Assumption \ref{Assump: Noise}; therefore, we can apply Lemma \ref{lem: Gaussian-expectation} to \bcgreen{the} conditional expectation\footnote{\bcgreen{By setting} $m\rightarrow m_k$, $w\rightarrow \bc{\veps_{k+1}}$, $\lambda_p+L/2\rightarrow 1$ and $\theta \rightarrow \frac{\theta L}{2}$ in Lemma \ref{lem: Gaussian-expectation}.}, 
\begin{equation*}
    \mE\left[e^{\frac{\theta L}{4\sigma^2}\Vert x_{k+1}-x_*\Vert^2 }\right]\leq \left(1-\theta\frac{\alpha^2L}{2}\right)^{-d/2}\mE\left[e^{\frac{\theta L}{4\sigma^2}\bcgreen{\rho_{GD}^2} \Vert x_k-x_*\Vert^2 + \frac{\theta^2 L^2\Vert \bcgreen{m_k^{GD}}\Vert^2}{16\sigma^2 (2-\theta \alpha^2 L)}}\right].
\end{equation*}
\bcgreen{By $L$-smoothness and $\mu$-strong convexity we can write,}
\begin{equation}\label{ineq: MI2}
\Vert \bcgreen{m_k^{GD}}\Vert^2 = 4\alpha^2 \Vert x_k-x_*-\alpha \nabla f(x_k)\Vert^2\leq 4\alpha^2 \bcgreen{\rho_{GD}^2} \Vert x_k-x_*\Vert^2.
\end{equation}
Hence, we obtain 
$$
\mE\left[e^{\frac{\theta L}{4\sigma^2}\Vert x_{k+1}-x_*\Vert^2 }\right]\leq \left(1-\theta  \frac{\alpha^2L}{2}\right)^{-d/2}\mE\left[e^{\frac{\theta L}{4\sigma^2}\bcgreen{\rho_{GD}^2}\left(1+\frac{\theta\alpha^2 L}{2-\theta \alpha^2 L} \right) \Vert x_k-x_*\Vert^2}\right].
$$
On the other hand, we can write
\begin{align*}
    \mE\left[e^{\frac{\theta L}{4\sigma^2} \bar{\rho}_{GD}^{2}\Vert x_k-x_*\Vert^2}\right] \leq \mE\left[e^{\frac{\theta L \bar{\rho}_{GD}^2}{4\sigma^2}\left( \bcgreen{\rho_{GD}^2} \Vert x_{k-1}-x_*\Vert^2 +(\bcgreen{m_{k-1}^{GD}})^\top \bc{w_{k}}+\alpha^2 \Vert \bc{w_{k}}\Vert^2\right)}\right],
\end{align*}
where $\bc{w_{k}}=\tilde{\nabla}f(y_{k-1})-\nabla f(y_{k-1})$. The fact that $\bc{w_{k}}$ is independent from $\mathcal{F}_{k-1}$ and $\theta \bar{\rho}_{GD}^2 < \theta < \frac{2}{\alpha^2 L}$ imply that we can again incorporate Lemma \ref{lem: Gaussian-expectation} using tower property of conditional \bcgreen{expectations} and obtain,\footnote{By setting $m\rightarrow m_k$, $w\rightarrow \bc{\veps_{k+1}}$, $\theta\rightarrow \frac{\theta L \bar{\rho}_{GD}^2}{2}$ and $\lambda_P+L/2\rightarrow 1$ in Lemma \ref{lem: Gaussian-expectation}.}
\begin{align*}
\mE\left[ e^{\frac{\theta L \bar{\rho}_{GD}^2}{4\sigma^2}\left( \bcgreen{\rho_{GD}^2} \Vert x_{k-1}-x_*\Vert^2 +\bcgreen{(m_{k-1}^{GD})}^\top \bc{w_{k}}+\alpha^2 \Vert \bc{w_{k}}\Vert^2\right)}\right]&=\mE\left[e^{\frac{\theta L \bar{\rho}_{GD}^2}{4\sigma^2}\bcgreen{\rho_{GD}^2} \Vert x_{k-1}-x_*\Vert^2} \mE\left[ e^{\frac{\theta L \bar{\rho}_{GD}^2}{4\sigma^2}\left(\bcgreen{(m_{k-1}^{GD})}^\top \bc{w_{k}}+\alpha^2 \Vert \bc{w_{k}}\Vert^2\right)}~\vline~ \mathcal{F}_{k-1}\right]\right],\\
&\bcgreen{=}\left(1-\theta\bar{\rho}_{GD}^2\frac{\alpha^2 L}{2}\right)^{-d/2} \mE\left[ e^{\frac{\theta L \bar{\rho}_{GD}^{2}}{4\sigma^2}\bcgreen{\rho^2_{GD}} \Vert x_{k-1}-x_*\Vert^2+\frac{\theta^2 \bar{\rho}^{4}_{GD}  L^2\Vert \bcgreen{m_{k-1}^{GD}}\Vert^2}{16\sigma^2 (2-\theta \bar{\rho}_{GD}^2 \alpha^2 L)}}\right].
\end{align*}
Since $\bar{\rho}_{GD}<1$, 
$$
\frac{\bar{\rho}_{GD}^4}{2-\theta \bar{\rho}_{GD}^2\alpha^2L}< \frac{1}{2-\theta \bar{\rho}_{GD}^2 \alpha^2 L} < \frac{1}{2-\theta \alpha^2 L}.
$$
\bcgreen{Using the inequality \eqref{ineq: MI2}}, we can see that $\Vert \bcgreen{m_{k-1}^{GD}}\Vert^2 \leq 4\alpha^2 \bcgreen{\rho^2_{GD}}\Vert x_{k-1}-x_*\Vert^2$. \bcgreen{Similarly,} This implies, 
\begin{align*}
\mE\left[ e^{\frac{\theta L \bar{\rho}_{GD}^2}{4\sigma^2}\left( \bcgreen{\rho^2_{GD}} \Vert x_{k-1}-x_*\Vert^2 +(\bcgreen{m_{k-1}^{GD}})^\top \bc{w_{k}}+\alpha^2 \Vert \bc{w_{k}}\Vert^2\right)}\right]
&\leq \left(1-\theta \bar{\rho}_{GD}^2\frac{\alpha^2 L}{2} \right)^{-d/2} \mE\left[ e^{\frac{\theta L \bar{\rho}_{GD}^2}{4\sigma^2}\bcgreen{\rho^2_{GD}} \left( 1+ \frac{\theta\alpha^2 L }{2-\theta \alpha^2 L}\right)\Vert x_{k-1}-x_*\Vert^2} \right],\\
&= \left(1-\theta \bar{\rho}_{GD}^2\frac{\alpha^2 L}{2} \right)^{-d/2} \mE\left[ e^{\frac{\theta L }{4\sigma^2}\bar{\rho}_{GD}^4\Vert x_{k-1}-x_*\Vert^2} \right].
\end{align*}
\bcgreen{Combining everything, we get}
$$
\mE[e^{\frac{\theta L}{4\sigma^2}\Vert x_{k+1}-x_*\Vert^2}]< \left(1-\theta \frac{\alpha^2 L}{2} \right)^{-d/2}\left(1-\theta \bar{\rho}_{GD}^2\frac{\alpha^2 L}{2} \right)^{-d/2}\mE\left[ e^{\frac{\theta L }{4\sigma^2}\bar{\rho}_{GD}^4\Vert x_{k-1}-x_*\Vert^2} \right].
$$
Repeating these arguments for each $x_{i}$ for $i=1,..,k-1$, we can show the desired inequality \eqref{ineq: gd-exp-bound-gauss}: 
$$
\mE[e^{\frac{\theta L}{4\sigma^2 }\Vert x_{k+1}-x_*\Vert^2}] < \prod_{j=0}^{k}\left(1 -\theta\big(\frac{\alpha^2 L}{2}\big)\bar{\rho}_{GD}^{2j} \right)^{-d/2}e^{\frac{\theta L}{4\sigma^2}\bar{\rho}_{GD}^{2(k+1)}\Vert x_0-x_*\Vert^2}.
$$
Recall that the function $f$ is $L$-smooth; therefore, it satisfies $f(x_k)-f(x_*)\leq \frac{L}{2}\Vert x_k-x_*\Vert^2$ and this directly gives us the bound on finite-horizon entropic risk measure:
\begin{align*}
    \mr_{k,\sigma^2}(\theta) & = \frac{2\sigma^2}{\theta}\log \mE[e^{\frac{\theta}{2\sigma^2}(f(x_k)-f(x_*))}]
    \leq \frac{2\sigma^2}{\theta}\log\mE[e^{\frac{\theta L}{4\sigma^2}\Vert x_{k}-x_*\Vert^2}],\\
    &< -\frac{\sigma^2 d}{\theta}\sum_{j=0}^{k-1}\log\left(1 -\theta(\frac{\alpha^2 L}{2})\bar{\rho}_{GD}^{2j} \right)+\frac{L\bar{\rho}_{GD}^{2k}}{2}\Vert x_0-x_*\Vert^2.
\end{align*}
\bcgreen{Using the property $\log(x)\geq 1-\frac{1}{x}$ for $x\in (0,1]$,}
\begin{align*}
\mr_{k,\sigma^2}(\theta) &< \sigma^2d \alpha^2 L \sum_{j=0}^{k-1}\frac{ \bar{\rho}_{GD}^{2j}}{2-\theta \alpha^2 L \bar{\rho}_{GD}^{2j}}+ \frac{L \bar{\rho}^{2k}_{GD}}{2}\Vert x_0-x_*\Vert^2,\\
&< \frac{\sigma^2 d\alpha^2 L}{(2-\theta \alpha^2 L)(1-\bar{\rho}_{GD}^{2})}+ \frac{L\bar{\rho}_{GD}^{2k}}{2}\Vert x_0-x_*\Vert^2,
\end{align*}
where the last inequality follows from $ 2-\theta\alpha^2 L<2-\theta\alpha^2 L \bar{\rho}^{2j}_{GD} $ for $\bar{\rho}_{GD}^2 <1$. Writing $\bar{\rho}^2_{GD}=\frac{2\bcgreen{\rho^2_{GD}}}{2-\theta\alpha^2L}$ back into the bound above gives the inequality \eqref{ineq: gd-risk-meas-bound-gauss}. Lastly, taking the limit superior on the right-hand side of the inequality \eqref{ineq: gd-risk-meas-bound-gauss} yields \eqref{ineq: gd-inf-risk-meas-bound-gauss}
}

\subsection{Proof of Theorem \ref{thm: gd-evar-bound-gaussian}} \bcgreen{Most} of the proof is similar to the proof of Theorem \ref{thm: Evar-TMM-str-cnvx-bound} \bcgreen{and} relies on the bounds provided in Proposition \ref{prop: gd-evar-bound-gauss}. Firstly notice that 
\begin{small}
\begin{align*}
EV@R_{1-\zeta}[f(x_k)-f(x_*)]&= \inf_{0<\theta} \mr_{k,\sigma^2}(\theta)+\frac{2\sigma^2}{\theta}\log(1/\zeta),\\
&\leq\inf_{0<\theta< \frac{2(1-\rho^2)}{\alpha^2L}} \frac{\sigma^2 d\alpha^2 L}{2(1-\rho^2)-\theta \alpha^2 L}+\frac{2\sigma^2}{\theta}\log(1/\zeta)+\left(\frac{\rho^2}{1-\theta \alpha^2 (L/2)}\right)^{k} \frac{L\Vert x_0-x_*\Vert^2}{2}.
\end{align*}
\end{small}
\bcgreen{We} fix $\varphi\in [\frac{1}{1+\sqrt{\frac{d}{2\log(1/\zeta)}}},1)$ and define $\theta_\varphi:=\frac{2\varphi (1-\rho^2)}{\alpha^2 L}$. \bcgreen{Then,} we can write 
$$
\frac{\rho^2}{1-\theta\alpha^2(L/2)}\leq\frac{\rho^2}{1-\theta_\varphi \alpha^2 (L/2)},
$$
for all $\theta \leq \theta_{\varphi}$. This implies 
\begin{equation}\label{ineq: gd-evar-app-helper}
EV@R_{1-\zeta}[f(x_k)-f(x_*)]\\ \leq \min_{0<\theta\leq \theta_{\varphi}} \frac{\sigma^2 d \alpha^2 L}{2(1-\rho^2)-\theta \alpha^2 L}+\frac{2\sigma^2}{\theta}\log(1/\zeta)+ \left( \frac{\rho^2}{1-\theta_\varphi \alpha^2 (L/2)}\right)^{k}\frac{L\Vert x_0-x_*\Vert^2}{2}.
\end{equation}
Let us introduce $\bar{\rho}_{GD}^{(\varphi)} = \frac{\rho^2}{1-\theta_\varphi \alpha^2 (L/2)}$ for simplicity and denote the function that depends on $\theta$ in the inequality above by $\bcgreen{\bar{h}_{GD}}(\theta)$. That is
$$
\bcgreen{\bar{h}_{GD}}(\theta):=\frac{\sigma^2d \alpha^2 L}{2(1-\rho^2)-\theta \alpha^2 L} + \frac{2\sigma^2}{\theta}\log(1/\zeta)+ \big( \bar{\rho}_{GD}^{(\varphi)}\big)^k\frac{L\Vert x_0-x_*\Vert^2}{2}.
$$
Its derivative $\bcgreen{\bar{h}_{GD}'}(\theta)$ is given as 
$$
\bcgreen{\bar{h}_{GD}'}(\theta)=\frac{\sigma^2 d (\alpha^2 L)^2}{(2(1-\rho^2)-\theta \alpha^2 L)^2}- \frac{2\sigma^2}{\theta^2}\log(1/\zeta),
$$
\bcgreen{whose root} $\theta_*$ \bcgreen{is given as}:
$$
\theta_*= \frac{2(1-\rho^2)}{\alpha^2 L (1+\sqrt{\frac{d}{2\log(1/\zeta)}})}\leq\theta_{\varphi},
$$
where the last inequality follows from the choice of $\varphi$ and $\theta_*=\theta_\varphi$ if $\varphi=\frac{1}{1+\sqrt{\frac{d}{2\log(1/\zeta)}}}=\varphi_{GD}$. On the other hand, we can see that $\bcgreen{\bar{h}_{GD}'}(\theta)<0$ on $(0,\theta_*)$ and $\bcgreen{\bar{h}_{GD}'}(\theta)\geq0$ on $[\theta_*,\theta_{\varphi}]$, which shows that the solution of the minimization problem on the right-hand side of the inequality \eqref{ineq: gd-evar-app-helper} is $\bcgreen{\bar{h}_{GD}}(\theta_*)$. The desired result follows directly from setting $\varphi=\varphi_{GD}$ \bcgreen{using $EV@R_{1-\zeta}[f(x_k)-f(x_*)]\leq \bar{h}_{GD}(\theta_*)$.}

\section{Proofs of Section \ref{sec: Subgaussian}}
We start with a helper lemma about the properties of sub-Gaussian distributions before proceeding to the proof of Proposition \ref{prop: SubOpt-Subgaussian}. \bcgreen{The following results in Lemma \ref{lem: SubGauss-prop} are well-known (see \cite{vershynin2018high}). We provide them for the sake of completeness.}
\begin{lemma}\label{lem: SubGauss-prop}
The following inequalities hold,
\begin{enumerate}
\item  \bcgreen{Under Assumption \ref{Assump: Gen-Noise}, we have the moment bounds}
\begin{equation} \label{ineq: SubGaus_Moment}
\mathbb{E}[\Vert \bc{\veps_{k+1}} \Vert^n \bcred{~|~\mF_k}]\leq (2\sigma^2)^{n/2}n\Gamma(\frac{n}{2}),
\end{equation} 
\bcgreen{for any $n\geq 1$, }where $\Gamma$ is the Gamma function. 
\item The moment generating function of $\Vert \bc{\veps_{k+1}} \Vert$ admits the following bound
\begin{equation} \label{ineq: SubGaus_MGF}
\mathbb{E}[e^{\theta \Vert \bc{\veps_{k+1}} \Vert}\bcred{~|~\mF_k}]\leq e^{4\sigma^2 \theta^2}, \;\;\;\text{ \bcgreen{for any }}\theta\geq 0.
\end{equation} 
\item Suppose \bcgreen{$0\leq \theta \leq \frac{1}{8\sigma^2}$}, then the moment generating function of \bcred{$\Vert \bc{\veps_{k+1}} \Vert^2$} has the following property 
\begin{equation}\label{ineq: SubGauss2_MGF} 
\mathbb{E}[e^{\theta \Vert \bc{\veps_{k+1}} \Vert^2} \bcred{~|~\mF_k}]\leq e^{8\sigma^2\theta},
\end{equation} 
\end{enumerate} 
\end{lemma}
\bcgreen{
\begin{proof}
We are first going to prove (1). Notice that integral representation of expectation together with Condition \eqref{cond: SubGaussian} yield 
\begin{align*}
    \mathbb{E}[\Vert \bc{\veps_{k+1}}\Vert^n \bcred{~|~\mF_k}]&=\int_{0}^{\infty} \mathbb{P}\{\Vert \bc{\veps_{k+1}} \Vert^n \geq t \bcred{~|~\mF_k}\}dt,\\
    &= \int_0^{\infty}
\mathbb{P}\{\Vert \bc{\veps_{k+1}} \Vert \geq t^{1/n}\bcred{~|~\mF_k}\}dt\leq 2\int_0^{\infty} e^{-\frac{t^{2/n}}{2\sigma^2}}dt.
\end{align*}
Applying change of variable $u=\frac{t^{2/n}}{2\sigma^2}$ yields
\begin{align*}
    \bcred{\mathbb{E}[\Vert \bc{\veps_{k+1}}\Vert^n \bcred{~|~\mF_k}]\leq \int_0^{\infty} n(2\sigma^2)^{n/2}e^{-u}u^{n/2-1}du=(2\sigma^2)^{n/2}n \Gamma(\frac{n}{2}).}
\end{align*}
In order to show (2) we are going to use \eqref{ineq: SubGaus_Moment} and Taylor expansion of exponential function. The dominated convergence theorem implies 
\bc{
\begin{align*}
    \mathbb{E}[e^{\theta \Vert \bc{\veps_{k+1}} \Vert}\bcred{~|~\mF_k}]&=1+\sum_{n=1}^{\infty} \frac{\theta^n}{n!} \mathbb{E}[\Vert \bc{\veps_{k+1}} \Vert^{n}\bcred{~|~\mF_k}]\leq 1+\sum_{n=1}^{\infty} \frac{\theta^n (2\sigma^2)^{n/2}n\Gamma(\frac{n}{2})}{n!},\\
    & \leq 1+\sum_{n=1}^{\infty}\frac{(2\theta^{2}\sigma^2)^{n}2n\Gamma(n)}{(2n)!}+\sum_{n=0}^{\infty}\frac{(2\theta^2\sigma^2)^{n+1/2}(2n+1)\Gamma(n+1/2)}{(2n+1)!},\\
    & \leq 1+ 2\sum_{n=1}^{\infty}\frac{(2\sigma^2\theta^2)^{n}n!}{(2n)!}+\sqrt{2\sigma^2\theta^2}\sum_{n=0}^{\infty}\frac{(2\sigma^2\theta^2)^{n}\Gamma(n+1/2)}{(2n)!}
\end{align*}
}
Recall that $2(k!)^2\leq (2k)!$ and duplication formula, $\Gamma(k+1/2)=\frac{(2k)!}{4^k k!} \sqrt{\pi}$ (see \cite[6.1.18]{abramowitzhandbook}). Hence, we obtain 
\begin{align*}
    \mathbb{E}[e^{\theta\Vert \bc{\veps_{k+1}} \Vert}\bcred{~|~\mF_k}]&\leq  1+ \sum_{n=1}^{\infty}\frac{(2\sigma^2\theta^2)^n}{n!}+ \sqrt{2\sigma^2\theta^2\pi}\sum_{n=1}^{\infty}\frac{1}{n!}(\frac{\sigma^2\theta^2}{2})^{n},\\
    &= e^{2\sigma^2\theta^2}+\sqrt{2\sigma^2\theta^2\pi}(e^{\sigma^2\theta^2/2}-1).
\end{align*}
Notice that $\sqrt{x}\leq e^{\frac{x}{4}}$ for all $x>0$; therefore, $\sqrt{2\sigma^2\theta^2\pi}(e^{\sigma^2\theta^2/2}-1)\leq e^{2\sigma^2\theta^2}(e^{\sigma^2\theta^2/2}-1)$ which yields \eqref{ineq: SubGaus_MGF}. Lastly, we show (3): 
\begin{align*}
    \mathbb{E}[e^{\theta \Vert \bcred{\bc{\veps_{k+1}}} \Vert^2}\bcred{~|~\mF_k}]=1+\sum_{\bcred{n}=1}^{\infty}\frac{\theta^\bcred{n}}{\bcred{n}!}\mathbb{E}[\Vert \bcred{\bc{\veps_{k+1}}} \Vert^{2\bcred{n}}] \leq 1+ \sum_{\bcred{n}=1}^{\infty}\frac{\theta^{\bcred{n}}(2\sigma^2)^{\bcred{n}}(2\bcred{n})\Gamma(\bcred{n})}{\bcred{n}!}=1+\sum_{\bcred{n}=1}^{\infty}(4\sigma^2\theta)^\bcred{n}.
\end{align*}
Note that for $x\leq\frac{1}{2}$ the inequality $\frac{1}{1-x}\leq e^{2x}$ holds and the condition $\theta\leq \frac{1}{8\sigma^2}$ implies $4\sigma^2 \theta \leq \frac{1}{2}$ which gives the desired result. 
\end{proof}
}

\bcred{
\subsection{Proof of Proposition \ref{prop: SubOpt-Subgaussian}}\label{app: SubOpt-Subgaussian}
\bcgreen{For simplicity of the presentation, as before, we drop the dependancy to the parameters $(\vartheta,\psi)$.} Since the parameters are as given in Theorem \ref{thm: TMM-MI-solution}, they satisfy \eqref{MI}. This implies that Lemma \ref{lem: func-contr-prop} holds for \bcgreen{the} Lyapunov function $\mV_P(\bc{z}_k)$. Hence, we can write 
$$
\mE[f(x_k)-f(x_*)]\leq\mE[\mV(\bc{z}_k)]\leq \rho^2 \mE[\mV(\bc{z}_{k-1})]+ \bcgreen{\alpha^2}(\frac{L}{2}+\lambda_P)\mE[\Vert \bc{w_{k}}\Vert^2],
$$
where we used the fact that $\mE[\bc{\veps_{k}} | \bcgreen{\mathcal{F}_{k-1}}]=0$. On the other hand, the inequality \eqref{ineq: SubGaus_Moment} \bcgreen{yields}  
$$
\mE[\Vert \bc{w_{k}}\Vert^2]=\mE[\mE[\Vert \bc{w_{k}}\Vert^2 | \mathcal{F}_{k-1}]]\bcgreen{\leq} 2\sigma^2.
$$
Therefore, we obtain
$$
\mE[f(x_k)-f(x_*)] \leq \rho^2\mE[\mV(\bc{z}_k-1)]+\alpha (\vartheta+\alpha L) \sigma^2 \leq \rho^{2k}\mV(\bc{z}_0)+\sigma^2\bcgreen{\alpha^2 (\frac{L}{2}+\lambda_P)}\frac{1-\rho^{2k}}{1-\rho^2}.
$$
This completes the proof.
}

\subsection{Proof of Proposition \ref{prop: entr-risk-meas-str-cnvx-subgaus}}\label{app: entr-risk-meas-str-cnvx-subgaus}
\bcgreen{Similar to before, }we drop the $(\vartheta,\psi)$ dependency in the notations for simplicity. Denote $\mV_P(\bc{z}_k)$ by $\mV(\bc{z}_k)$, and also denote the largest eigenvalue of $d$-by-$d$ principle submatrix of the matrix $P$ by $\lambda_P$ which can be computed as $\lambda_P=\frac{\vartheta}{2\alpha}$ under the premise of Theorem \ref{thm: TMM-MI-solution}. Similar to the proof of Proposition \ref{prop: risk-meas-bound-gaussian}, we are first going to provide a bound on moment generating function of $f(x_k)-f(x_*)$. \bcgreen{More specifically, we will show that}
\begin{equation}\label{ineq: exp-bound-subgauss}
    \mE[e^{\frac{\theta}{2\sigma^2}(f(x_k)-f(x_*))}] \leq e^{2\theta\alpha^2\big( L+2\lambda_P\big)\sum_{j=0}^{k-1}\hat{\rho}^{2j}+\frac{\theta}{2\sigma^2}\hat{\rho}^{2k} 2\mV(\bc{z}_0)+ 2\theta\hat{\rho}^{2(k-1)}\alpha^2(\frac{L}{2}+\lambda_P)},
\end{equation}
\bcgreen{For this purpose, we} introduce the notation 
\begin{align*}
\hat{k}_\theta=\frac{1}{2}\left[ \sqrt{(\rho^{2}+32\alpha^2\mtv\theta)^2+128\alpha^2\mtv \theta}-(\rho^2 +32\alpha^2\mtv\theta)\right],
\end{align*}
which satisfies the following relation 
\begin{equation}\label{prop: hat-k-cond-1}
(\hat{k}_\theta)^2+\hat{k}_\theta(\rho^2+32\alpha^2\mtv \theta)=32\alpha^2 \mtv\theta \implies \hat{k}_\theta=\frac{32\alpha^2\mtv \theta}{\rho^2+\hat{k}_\theta+32\alpha^2\mtv\theta}.
\end{equation}
The rest of the proof is similar to the proof of Lemma \ref{lem: risk-meas-bound-gauss} where \bcgreen{our aim will be to} show the inequality 
\begin{equation}\label{ineq: exp-ineq-subgaus-helper}
\mE\left[e^{\frac{\theta}{2\sigma^2}\big(\mV(\bc{z}_{k+1})+\hat{k}_\theta\mV(\bc{z}_{k})\big)}\right] < e^{2\theta\alpha^2\big( L+2\lambda_P\big)\sum_{j=0}^{k-1}\hat{\rho}^{2j}} \mE[e^{\frac{\theta}{2\sigma^2}\hat{\rho}^{2k}\big( \mV(\bc{z}_1)+\hat{k}_\theta\mV(\bc{z}_0)\big)}].
\end{equation}
By Assumption \ref{Assump: Gen-Noise}, the noise $\bc{\veps_{k+1}}=\tilde{\nabla}f(y_k)-\nabla f(y_k)$ is independent from $\mathcal{F}_k$. Moreover, since we already adopted the parameters of Theorem \ref{thm: TMM-MI-solution}, they satisfy \eqref{MI} with matrix $P$ where $\lambda_P=\frac{\vartheta}{2\alpha}$. Therefore, we can use Lemma \ref{lem: func-contr-prop} to derive the following inequality using \bcgreen{the tower property of conditional expectations,}
\begin{small}
\begin{align*}
    \mE[e^{\frac{\theta}{2\sigma^2}(\mV(\bc{z}_{k+1})+\hat{k}_\theta\mV(\bc{z}_{k}))}]&\leq \mE\left[e^{\frac{\theta}{2\sigma^2}(\rho^2 + \hat{k}_\theta)\mV(\bc{z}_k)} \mE\left[ e^{\frac{\theta}{2\sigma^2}\big(m_k^\top \bc{\veps_{k+1}} + \alpha^2(\frac{L}{2}+\lambda_P)\Vert \bc{\veps_{k+1}}\Vert^2\big)} ~\vline~ \mathcal{F}_k \right] \right].
\end{align*}
\end{small}
\bcgreen{
The risk-averseness parameter $\theta$ satisfies \eqref{cond: cond-theta-subgauss}, in particular $\theta < \frac{1}{4\alpha (\vartheta+\alpha L)} = \frac{1}{8\alpha^2 (\frac{L}{2}+\lambda_P)}$; therefore, we can see that $\frac{\theta \alpha^2 (\frac{L}{2}+\lambda_P)}{\sigma^2} < \frac{1}{8\sigma^2}$.} 
\bcgreen{
Applying inequality \eqref{ineq: SubGauss2_MGF} with Cauchy-Schwarz and H\"older inequalities, the \bcgreen{following} conditional expectation can be bounded as,
\begin{align*}
    \mE\left[e^{\frac{\theta}{2\sigma^2}\big( m_k^\top \bc{\veps_{k+1}} + \alpha^2 (\frac{L}{2}+\lambda_P)\Vert \bc{\veps_{k+1}}\Vert^2 \big)} ~\vline~ \mathcal{F}_k\right]&\leq \mE\left[ e^{\frac{\theta}{2\sigma^2}\big(  \Vert m_k \Vert \Vert \bc{\veps_{k+1}}\Vert + \alpha^2(\frac{L}{2}+\lambda_P)\Vert \bc{\veps_{k+1}}\Vert^2\big)}~\vline~ \mathcal{F}_k\right],\\
    &\leq \left(\mE[e^{\frac{\theta}{\sigma^2}\Vert m_k \Vert \Vert \bc{\veps_{k+1}}\Vert}~\vline~ \mathcal{F}_k]\right)^{1/2} \left(\mE[e^{\frac{\theta}{\sigma^2}\alpha^2(\frac{L}{2}+\lambda_P)\Vert \bc{\veps_{k+1}}\Vert^2}~\vline~\mathcal{F}_k] \right)^{1/2}.
\end{align*}
\bcgreen{Applying the inequalities} \eqref{ineq: SubGaus_MGF} and \eqref{ineq: SubGauss2_MGF} \bcgreen{yields}
$$
\mE\left[e^{\frac{\theta}{2\sigma^2}\big(m_k^\top \bc{\veps_{k+1}} + \alpha^2(\frac{L}{2}+\lambda_P)\Vert \bc{\veps_{k+1}}\Vert^2\big)}~\vline~ \mathcal{F }_k\right]\leq e^{\frac{2\theta^2 \Vert m_k\Vert^2}{\sigma^2}+2\theta \alpha^2(L+2\lambda_P)}.
$$ 
Recall $\Vert m_k\Vert^2\leq 8\alpha^2\mtv (\mV(\bc{z}_k)+\mV(\bc{z}_{k-1}))$ from Lemma \ref{lem: bound-on-mk}. Therefore, 
\begin{align*}
\mE[e^{\frac{\theta}{2\sigma^2}(\mV(\bc{z}_{k+1})+\hat{k}_\theta\mV(\bc{z}_{k}))}]&\leq \mE\left[e^{\frac{\theta}{2\sigma^2}\big( (\rho^2 + \hat{k}_\theta+ 32\alpha^2\mtv \theta)\mV(\bc{z}_k) +32\alpha^2\mtv\theta \mV(\bc{z}_{k-1})\big)+2\theta \alpha^2(L+2\lambda_P)} \right],\\
&\leq e^{2\theta \alpha^2 (L+2\lambda_P)}\mE\left[ e^{\frac{\theta}{2\sigma^2}\hat{\rho}^2\big(\mV(\bc{z}_k)+\hat{k}_\theta \mV(\bc{z}_{k-1}) \big)}\right],
\end{align*}
where the last inequality follows from \eqref{prop: hat-k-cond-1} and 
\begin{equation*}
(\rho^2 + \hat{k}_\theta+32\alpha^2\mtv\theta)\mV(\bc{z}_k)+32\alpha^2\mtv \theta \mV(\bc{z}_{k-1})\\= (\rho^2+\hat{k}_\theta+32\alpha^2 \mtv\theta)\big( \mV(\bc{z}_{k})+\frac{32\alpha^2\mtv\theta}{\rho^2+\hat{k}_\theta+32\alpha^2\mtv\theta}\mV(\bc{z}_{k-1})\big).
\end{equation*}

Next, we study the expectation on the right-hand side of the above inequality \bcgreen{further}. Again following similar arguments, \bcgreen{it is straighforward to show}: 
\begin{equation*}
\mE\left[ e^{\frac{\theta}{2\sigma^2}\hat{\rho}^2\big(\mV(\bc{z}_k)+\hat{k}_\theta \mV(\bc{z}_{k-1}) \big)}\right]\\\leq \mE\left[ e^{\frac{\theta}{2\sigma^2}\hat{\rho}^2(\rho^2 + \hat{k}_\theta)\mV(\bc{z}_{k-1})}\mE\left[ e^{\frac{\theta \hat{\rho}^2}{2\sigma^2}\big(m_{k-1}^\top \bc{w_{k}}+\alpha^2(\frac{L}{2}+\lambda_P)\Vert \bc{w_{k}}\Vert^2\big)} ~\vline~ \mathcal{F}_{k-1}\right] \right].
\end{equation*}
\bcgreen{Proceeding in a similar fashion,}
\begin{equation*}
    \mE\left[ e^{\frac{\theta \hat{\rho}^2}{2\sigma^2}\big(m_{k-1}^\top \bc{w_{k}}+\alpha^2(\frac{L}{2}+\lambda_P)\Vert \bc{w_{k}}\Vert^2\big)} ~\vline~ \mathcal{F}_{k-1}\right]\leq \left( \mE\left[e^{\frac{\theta \hat{\rho}^2}{\sigma^2}\Vert m_{k-1}\Vert \Vert \bc{w_{k}}\Vert} ~\vline~ \mathcal{F}_{k-1}\right]\right)^{1/2} \left( \mE\left[e^{\frac{\theta \hat{\rho}^2}{\sigma^2}\alpha^2(\frac{L}{2}+\lambda_P)\Vert \bc{w_{k}}\Vert^2}~\vline~ \mathcal{F}_{k-1}\right]\right)^{1/2}.
\end{equation*}
Using the inequalities \eqref{ineq: SubGaus_MGF} and \eqref{ineq: SubGauss2_MGF} together with Lemma \ref{lem: bound-on-mk} one more time \bcgreen{for the} conditional expectations on the right-hand side of the above inequality: 
\begin{align*}
    \mE\left[ e^{\frac{\theta \hat{\rho}^2}{2\sigma^2}\big(m_{k-1}^\top \bc{w_{k}}+\alpha^2(\frac{L}{2}+\lambda_P)\Vert \bc{w_{k}}\Vert^2\big)} ~\vline~ \mathcal{F}_{k-1}\right]\leq e^{\frac{\theta^2}{\sigma^2}16\hat{\rho}^4\alpha^2\mtv (\mV(\bc{z}_{k-1})+\mV(\bc{z}_{k-2})) +2\theta\hat{\rho}^2\alpha^2 (L+2\lambda_P)}.
\end{align*}
Since $\hat{\rho}^2<1$, we obtain 
\begin{align*}
\mE[e^{\frac{\theta}{2\sigma^2}\hat{\rho}^2 ( \mV(\bc{z}_{k})+\hat{k}_\theta \mV(\bc{z}_{k-1}))}]&\leq \mE\left[e^{\frac{\theta}{2\sigma^2}\hat{\rho}^2 \big( (\rho^2 +\hat{k}_\theta+32\alpha^2\mtv \theta )\mV(\bc{z}_{k-1})+32\alpha^2\mtv\theta \mV(\bc{z}_{k-2})\big)+2\theta \hat{\rho}^2\alpha^2(L+2\lambda_P)}\right],\\
& \leq e^{2\theta\hat{\rho}^2 \alpha^2(L+2\lambda_P)}\mE[e^{\frac{\theta}{2\sigma^2}\hat{\rho}^4 \big(\mV(\bc{z}_{k-1})+\hat{k}_\theta\mV(\bc{z}_{k-2})\big)}].
\end{align*}
Notice that this inequality gives us 
$$
\mE\left[ e^{\frac{\theta}{2\sigma^2}\big(\mV(\bc{z}_{k+1})+\hat{k}_\theta\mV(\bc{z}_k)\big)} \right]\leq e^{2\theta \alpha^2(1+\hat{\rho}^2)(L+2\lambda_P)}\mE[e^{\frac{\theta}{2\sigma^2}\hat{\rho}^4 \big(\mV(\bc{z}_{k-1})+\hat{k}_\theta\mV(\bc{z}_{k-2})\big)}],
$$
and continuing this \bcgreen{recursion} for $k$-steps yields the inequality \eqref{ineq: exp-ineq-subgaus-helper}. Next, we provide \bcgreen{a} bound on the moment generating function of $\mV(\bc{z}_1)+\hat{k}_\theta\mV(\bc{z}_0)$. By Lemma \ref{lem: func-contr-prop}, we have 
\begin{align*}
    \mE\left[e^{\frac{\theta \hat{\rho}^{2k}}{2\sigma^2}(\mV(\bc{z}_1)+\hat{k}_\theta\mV(\bc{z}_0))}\right]\leq \mE[e^{\frac{\theta \hat{\rho}^{2k}}{2\sigma^2}\big((\rho^2+\hat{k}_\theta) \mV(\bc{z}_0)+m_0^\top \bc{w_1} + \alpha^2 (\frac{L}{2}+\lambda_P)\Vert \bc{w_1}\Vert^2 \big)}].
\end{align*}
\bcgreen{Here}, $\bc{z}_0$ is the initialization of the algorithm \bcgreen{and we take it to be deterministic}. Also, using the Cauchy-Schwarz inequality together with \bcgreen{the} H\"older inequality and the inequalities \eqref{ineq: SubGaus_MGF} and \eqref{ineq: SubGauss2_MGF} yields,
\begin{align*}
     \mE\left[e^{\frac{\theta \hat{\rho}^{2k}}{2\sigma^2}(\mV(\bc{z}_1)+\hat{k}_\theta\mV(\bc{z}_0))}\right] \leq e^{\frac{\theta}{2\sigma^2}\hat{\rho}^{2k}(\rho^2+\hat{k}_\theta)\mV(\bc{z}_0)}\mE[e^{\frac{\theta \hat{\rho}^{2k}}{2\sigma^2}\big(m_0^\top \bc{w_1} + \alpha^2 (\frac{L}{2}+\lambda_P)\Vert \bc{w_1}\Vert^2 \big)}],\\
     \leq e^{\frac{\theta}{2\sigma^2}\hat{\rho}^{2k}(\rho^2+\hat{k}_\theta)\mV(\bc{z}_0)}e^{\frac{2\theta^2}{\sigma^2}\hat{\rho}^{4k}\Vert m_0\Vert^2 + 2\theta \hat{\rho}^{2k}\alpha^2 (L+2\lambda_P)}.
\end{align*}
Recall from the proof of Lemma \ref{lem: risk-meas-bound-gauss} that $\Vert m_0\Vert^2\leq 16 \alpha^2\mtv \mV(\bc{z}_0)$ and $\hat{\rho}^{4k}<\hat{\rho}^{2k}$. Therefore, 
\begin{align}
\mE\left[e^{\frac{\theta \hat{\rho}^{2k}}{2\sigma^2}(\mV(\bc{z}_1)+\hat{k}_\theta\mV(\bc{z}_0))}\right]&\leq e^{\frac{\theta}{2\sigma^2}\hat{\rho}^{2k}(\rho^2 + \hat{k}_\theta)\mV(\bc{z}_0)+\frac{ \theta^2 \hat{\rho}^{2k}}{\sigma^2}32\alpha^2 \mtv \mV(\bc{z}_0)+2\theta \hat{\rho}^{2k}\alpha^2 (L+2\lambda_P)},\nonumber\\
&\leq e^{\frac{\theta}{2\sigma^2}\hat{\rho}^{2(k+1)} 2\mV(\bc{z}_0)+ 2\theta\hat{\rho}^{2k}\alpha^2(L+2\lambda_P)}.\label{ineq: exp-bound-subgauss-helper-2}
\end{align}
Combining \eqref{ineq: exp-bound-subgauss-helper-2} with \eqref{ineq: exp-ineq-subgaus-helper} yields 
$$
\mE[e^{\frac{\theta}{2\sigma^2}(\mV(\bc{z}_{k+1})+\hat{k}_\theta\mV(\bc{z}_{k}))}]\leq  e^{2\theta\alpha^2\big( L+2\lambda_P\big)\sum_{j=0}^{k}\hat{\rho}^{2j}+\frac{\theta}{2\sigma^2}\hat{\rho}^{2(k+1)} 2\mV(\bc{z}_0)}.
$$
Lastly, we can use the fact that the inequality $$
e^{\frac{\theta}{2\sigma^2}(f(x_{k+1})-f(x_*))}\leq e^{\frac{\theta}{2\sigma^2}\mV(\bc{z}_{k+1})}<e^{\frac{\theta}{2\sigma^2}(\mV(\bc{z}_{k+1})+\hat{k}_\theta\mV(\bc{z}_k))},$$ 
holds pointwise \bcgreen{and by taking expectations}
\begin{align*}
\mE\left[e^{\frac{\theta}{2\sigma^2}(\mV(\bc{z}_{k+1})+\hat{k}_\theta\mV(\bc{z}_k))}\right]
\geq\mE\left[ e^{\frac{\theta}{2\sigma^2}\mV(\bc{z}_{k+1})} \right] \geq \mE\left[e^{\frac{\theta}{2\sigma^2}(f(x_{k+1})-f(x_*))} \right].
\end{align*}
This gives the desired inequality \eqref{ineq: exp-bound-subgauss}. \bcgreen{Our} main result \eqref{ineq: fin-hor-risk-meas-str-cnvs-subgaus} \bcgreen{then} directly follows from \eqref{ineq: exp-bound-subgauss} and the definition of \bcgreen{the} finite-horizon entropic risk:
\begin{align*}
\mr_{k,\sigma^2}(\theta)&=\frac{2\sigma^2}{\theta}\log \mE\left[ e^{\frac{\theta}{2\sigma^2}(f(x_k)-f(x_*))}\right],
\\
&\bcgreen{<} \frac{2\sigma^2}{\theta} \log \exp\left\{2\theta\alpha^2\big( L+2\lambda_P\big)\sum_{j=0}^{k-1}\hat{\rho}^{2j}+\frac{\theta}{2\sigma^2}\hat{\rho}^{2k} 2\mV(\bc{z}_0)\right\}.
\end{align*}
\bcgreen{This completes the proof.}
}
\bcred{
\subsection{Proof of Theorem \ref{thm: evar-bound-str-cnvx-subgaus}}\label{app: evar-bound-str-cnvx-subgaus}
Throughout the proof we are going to drop the $\vartheta$ and $\psi$ in the notations $\alpha_{\vartheta,\psi},\beta_{\vartheta,\psi},\gamma_{\vartheta,\psi}$ and $\rho_{\vartheta,\psi}$, and also drop $(\alpha,\beta,\gamma)$ in \bcgreen{$\hat{\mtv}_{\alpha,\beta,\gamma}, \htv_{\alpha,\beta,\gamma}, \hat{\rho}_{\alpha,\beta,\gamma}$, and $\hat{\hat \rho}^{(\varphi)}_{\alpha,\beta,\gamma}$ for simplicity.} 
The proof of the theorem is similar to the proof of Theorem \ref{thm: Evar-TMM-str-cnvx-bound} \bcgreen{and} incorporates the bound \eqref{ineq: fin-hor-risk-meas-str-cnvs-subgaus} provided in Proposition \ref{prop: entr-risk-meas-str-cnvx-subgaus} to derive an upper bound on EV@R. First notice that under the setting of Theorem \ref{thm: TMM-MI-solution}, if we set $\theta<\theta_u^{sg}$ then the conditions of Proposition \ref{prop: entr-risk-meas-str-cnvx-subgaus} are satisfied and the inequality \eqref{ineq: fin-hor-risk-meas-str-cnvs-subgaus} can be used as
\begin{align}\label{ineq: app-evar-bound-1}
    EV@R_{1-\zeta}[f(x_k)-f(x_*)] &=\inf_{0<\theta}\{ \mr_{k,\sigma^2}(\theta) +\frac{2\sigma^2}{\theta}\log(1/\zeta)\},\nonumber\\
    &\bc{\leq} \inf_{0< \theta < \theta_u^{sg}}\left\{\frac{4\sigma^2 \bcgreen{\alpha^2 (L+2\lambda_P)}}{1-\hat{\rho}^2}+\hat{\rho}^{2k}2\mV(\bc{z}_0)+\frac{2\sigma^2}{\theta}\log(1/\zeta)\right\},
\end{align}
where strict inequality $\theta <\theta_u^{sg}$ is required for $\hat{\rho}^2<1$. Let us fix a constant $\varphi\in (0,1)$ and set $\theta_{\varphi}^{sg}=\varphi \theta_u^{sg}$, then we can see that $\hat{\rho}^2$ is increasing with respect to $\theta$ and the following inequality holds for all $\theta< \theta_{\varphi}^{sg}$, 
\begin{align*}
\hat{\rho}^2 &= \frac{1}{2}\big( \rho^2 + 32\alpha^2\bcgreen{\hat{\mtv}} \theta \big) + \frac{1}{2}\sqrt{(\rho^2+32\alpha^2\bcgreen{\hat{\mtv}}\theta)^2+128\alpha^2\bcgreen{\hat{\mtv}} \theta},\\
&< \frac{1}{2}\big( \rho^2 + 32\alpha^2 \bcgreen{\hat{\mtv}} \theta_{\varphi}^{sg}\big)^2+\frac{1}{2}\sqrt{(\rho^2+32\alpha^2\bcgreen{\hat{\mtv}}\theta_{\varphi}^{sg})^2+128\alpha^2\bcgreen{\hat{\mtv}}\theta_{\varphi}^{sg}}=\hhrho.
\end{align*}
Similarly, we can also write
\begin{small}
\begin{align*}
    \rho^2 &= \frac{1}{2}\big( \rho^2 + 32\alpha^2\bcgreen{\hat{\mtv}} \theta \big) + \frac{1}{2}\sqrt{(\rho^2+32\alpha^2\bcgreen{\hat{\mtv}}\theta)^2+128\alpha^2\bcgreen{\hat{\mtv}}\theta},\\
    &< \frac{1}{2}\big( \rho^2 + 32\alpha^2 \bcgreen{\hat{\mtv}} \theta\big)+\frac{1}{2}\sqrt{(\rho^2+32\alpha^2\bcgreen{\hat{\mtv}}\theta_{\varphi}^{sg})^2+128\alpha^2\bcgreen{\hat{\mtv}}\theta_{\varphi}^{sg}}=\frac{1}{2}(\rho^2+32\alpha^2\bcgreen{\hat{\mtv}}\theta)+\frac{1}{2}\htv,
\end{align*}
\end{small}
on the region $(0,\theta_{\varphi}^{sg}]$. Therefore, the inequality \eqref{ineq: app-evar-bound-1} becomes 
\begin{align*}
    EV@R_{1-\zeta} &\bc{\leq} \inf_{0<\theta<\theta_u^{sg}} \left\{ \frac{4\sigma^2 \bcgreen{\alpha^2 (L+2\lambda_P)}}{1-\hat{\rho}^2}+\frac{2\sigma^2}{\theta}\log(1/\zeta) +\hat{\rho}^{2k}2\mV(\bc{z}_0)\right\},\\
    & \bc{\leq} \min_{0<\theta \leq \theta_{\varphi}^{sg}} \left\{ \frac{8\sigma^2 \bcgreen{\alpha^2 (L+2\lambda_P)}}{2-\rho^2-\htv- 32\alpha^2\bcgreen{\hat{\mtv}} \theta } +\frac{2\sigma^2}{\theta}\log(1/\zeta)+(\hhrho)^{k} 2\mV(\bc{z}_0)\right\}.
\end{align*}
\bcgreen{It can be seen that the function}
$$
\hat{h}(\theta)= \frac{8\sigma^2 \bcgreen{\alpha^2 (L+2\lambda_P)}}{2-\rho^2-\htv- 32\alpha^2\bcgreen{\hat{\mtv}} \theta }+\frac{2\sigma^2}{\theta}\log(1/\zeta) + (\hhrho)^k 2\mV(\bc{z}_0),
$$
is \bcgreen{convex} on the region $(0,\theta_{\varphi}^{sg}]$ \bcgreen{with a unique minimum}. With a slight abuse of notation, let us introduce
\begin{align*}
    \mtc_0=8\bcgreen{\alpha^2 (L+2\lambda_P)}, \;\; \mtc_1=2-\rho^2 - \hat{t}^{(\varphi)}, \;\; \mtc_2=32\alpha^2\bcgreen{\hat{\mtv}}.
\end{align*}
Since $\hat{\hat \rho}$ and $\hat{t}^{(\varphi)}$ do not depend on $\theta$, we can see that the derivative $\hat{h}'(\theta)$ can be written as 
$$
\hat{h}'(\theta)=\frac{\sigma^2 \mtc_0\mtc_2 \theta^2 -2\sigma^2 (\mtc_1-\theta \mtc_2)^2\log(1/\zeta)}{\theta^2 (\mtc_1-\theta\mtc_2)^2}.
$$
The root, $\hat{\theta}_*$, of the polynomial at the nominator satisfying $\hat{\theta}_*\in (0,\frac{\mtc_1}{\mtc_2}]$ is given as 
$$
\hat{\theta}_*=\frac{\mtc_1}{\mtc_2+ \sqrt{\frac{\mtc_0\mtc_2}{2\log(1/\zeta)}}}.
$$
We can see that 
\begin{align*}
    \mtc_0\mtc_2\theta^2 -2(\mtc_1-\theta\mtc_2)^2\log(1/\zeta) \leq 0& \;\; \text{ on } (0, \hat{\theta}_*],\\
    \mtc_0\mtc_2\theta^2 -2(\mtc_1-\theta\mtc_2)^2\log(1/\zeta) > 0& \;\; \text{ on } (\hat{\theta}_*, \frac{\mtc_1}{\mtc_2}].
\end{align*}
Therefore, we can conclude that the function $\hat{h}(\theta)$ is decreasing on $(0,\hat{\theta}_*]$ and increasing on $(\hat{\theta}_*, \frac{\mtc_1}{\mtc_2}]$.
So far, we have shown that
$$
EV@R_{1-\zeta}[f(x_k)-f(x_*)] \leq \min_{0<\theta \leq \theta_\varphi^{sg}} \hat{h}(\theta),
$$
where $h$ is \bcgreen{convex} on $(0,\theta_\varphi^{sg}]$ and the derivative $\hat{h}'$ diminishes at $\hat{\theta}_*$. We obtain
\begin{equation}\label{ineq: appendix-evar-bound-subauss-2}
EV@R_{1-\zeta}[f(x_k)-f(x_*)]\bc{\leq} \min_{0<\theta\leq \theta_{\varphi}^{sg}} \hat{h}(\theta)= \begin{cases} 
\hat{h}(\hat{\theta}_*), & \text{ if } \hat{\theta}_* \leq \theta_{\varphi}^{sg},\\
\hat{h}(\theta_\varphi^{sg}), & \text{otherwise}.
\end{cases}
\end{equation}
\bc{ By the definitions of $\mtc_0, \mtc_1$, and $\mtc_2$, the condition \eqref{cond: conf-lev-subgauss} can be written as
\begin{align*}
\log(1/\zeta) \bc{\leq} \frac{\mtc_0 \mtc_2}{2}\Big(\frac{\theta_{\varphi}^{sg}}{\mtc_1-\theta_{\varphi}^{sg}\mtc_2}\Big)^2 \implies \frac{\mtc_1-\theta_{\varphi}^{sg}\mtc_2}{\theta_{\varphi}^{sg}} \bc{\leq} \sqrt{\frac{\mtc_0\mtc_2}{2\log(1/\zeta)}}. 
\end{align*}
Hence, 
\begin{equation*}
\text{Condition \eqref{cond: conf-lev-subgauss}} \implies \hat{\theta}_*= \frac{\mtc_1}{\mtc_2+\sqrt{\frac{\mtc_0\mtc_2}{2\log(1/\zeta)}}}\leq \theta_{\varphi}^{sg}.
\end{equation*}
Moreover, we can calculate $\hat{h}(\theta_*)$ as
\begin{small}
\begin{align*} 
\hat{h}(\hat{\theta}_*)&=\frac{\sigma^2\mtc_0}{\mtc_1-\hat{\theta}_*\mtc_2}+\frac{2\sigma^2}{\hat{\theta}_*}\log(1/\zeta)+(\hhrho)^k 2\mV(\bc{z}_0),\\
&= \frac{\sigma^2}{\mtc_1}\sqrt{\mtc_0}\Big(\sqrt{2\mtc_2\log(1/\zeta)}+\sqrt{\mtc_0}\Big),\\
&\quad\quad+ \frac{\sigma^2}{\mtc_1}\sqrt{2\mtc_2\log(1/\zeta)}\Big(\sqrt{2\mtc_2\log(1/\zeta)}+\sqrt{\mtc_0}\Big)+(\hhrho)^k2\mV(\bc{z}_0),\\
&=\frac{\sigma^2}{\mtc_1}\Big(\sqrt{2\mtc_2\log(1/\zeta)}+\sqrt{\mtc_0}\Big)^2 + (\hhrho)^k2\mV(\bc{z}_0).
\end{align*}
\end{small}
}
Therefore, if condition \eqref{cond: conf-lev-subgauss} holds then we have $\hat{\theta}_*< \theta_{\varphi}^{sg}$ and the inequality \eqref{ineq: appendix-evar-bound-subauss-2} yields the inequality \eqref{ineq: fin-step-evar-bound-subgauss-1} from the calculation of $\hat{h}(\hat{\theta}_*)$ above. The inequality \eqref{ineq: fin-step-evar-bound-subgauss-2}, also directly follows from \eqref{ineq: appendix-evar-bound-subauss-2} if $\theta_\varphi^{sg}>\hat{\theta}_*$. We can derive the inequality \eqref{ineq: inf-evar-bound-subgauss} by simply taking limit superior of both \eqref{ineq: fin-step-evar-bound-subgauss-1} and \eqref{ineq: fin-step-evar-bound-subgauss-2}.
}
\section{Proofs of Supplementary Results}

\subsection{Proof of Lemma \ref{lem: sol-lyap-eq-quad}}\label{app: sol-lyap-eq-quad}
\bcgreen{We} \bc{denote the covariance matrix $\mathbb{E}[(\bc{z}_k-\mathbb{E}[\bc{z}_k])(\bc{z}_k-\mathbb{E}[\bc{z}_k])^\top]$ by $\Xi_k$. Since the noise $w_k$ is unbiased, we have $
\mathbb{E}[\bc{z}_{k+1}]=A_Q \mathbb{E}[\bc{z}_k]
$. Therefore, we can write the following recursion, 
\begin{align}
    \Xi_{k+1}&=\mathbb{E}[(\bc{z}_{k+1}-\mathbb{E}[\bc{z}_{k+1}])(\bc{z}_{k+1}-\mathbb{E}[\bc{z}_{k+1}])^\top] \nonumber,\\
    &=\mathbb{E}[(A_Q(\bc{z}_k-\mathbb{E}[\bc{z}_k])+Bw_k)(A_Q(\bc{z}_k-\mathbb{E}[\bc{z}_k])+Bw_k)^\top]\nonumber,\\
    &=A_Q \Xi_k A_Q^\top + \sigma^2 BB^\top.\label{eq: recurs-Vk-lemma2}
\end{align}
\bcgreen{By the statement of Proposition \ref{prop: quad_qisk_meas_convergence}, we have $(\alpha,\beta,\gamma)\in \mathcal{F}_\theta$, which also implies $(\alpha,\beta,\gamma)\in\mathcal{S}_q$ by Lemma \ref{lem: feas-set-stab-set}. Therefore, we have $\rho(A_Q)<1$ and }the following discrete time Lyapunov equation holds for the limiting covariance matrix $\Xi_\infty=\lim_{k\rightarrow \infty}\Xi_k$: 
\begin{equation}\label{eq: lyp-eq-helper}
\Xi_\infty= A_Q \Xi_\infty A_Q^\top + \sigma^2 BB^\top.  
\end{equation}
\bcred{Let $Q=\bcgreen{U}_Q\Lambda \bcgreen{U}_Q^\top$ be the eigendecomposition of the matrix $Q$, then we can write the following for $A_Q$,
\begin{align*}
A_Q&=\begin{bmatrix}
(1+\beta)I_d-\alpha(1+\gamma)\bcgreen{U}_Q \Lambda \bcgreen{U}_Q^\top & -(\beta - \alpha\gamma \bcgreen{U}_Q\Lambda\bcgreen{U}_Q^\top)\\ 
I_d & 0_d
\end{bmatrix},\\
&= 
\begin{bmatrix} 
\bcgreen{U}_Q  & 0 \\
0 & \bcgreen{U}_Q
\end{bmatrix}
\begin{bmatrix} 
(1+\beta)-\alpha(1+\gamma)\Lambda & -(\beta -\alpha \gamma \Lambda)\\
I_d & 0_d
\end{bmatrix}\begin{bmatrix} 
\bcgreen{U}_Q^\top & 0_d \\
0_d & \bcgreen{U}_Q^\top
\end{bmatrix}.
\end{align*}
\bcred{Define $\bcgreen{U}_\pi$ as the permutation matrix associated with permutation $\pi$ over $\{1,2,...,2d\}$ such that $\pi(i)=2i-1$ for $i\in [1,d]$ and $\pi(i)=2(i-d)$ for $i \in [d+i, 2d]$. It is well known that the permutation matrices are invertible and their inverses satisfy $\bcgreen{U}_{\pi}^{-1}=\bcgreen{U}_\pi^{\top}=\bcgreen{U}_{\pi^{-1}}$. Hence, we can write $A_Q$ as 
$$
A_Q = \underbrace{\begin{bmatrix} 
\bcgreen{U}_Q & 0_d\\ 
0_d & \bcgreen{U}_Q
\end{bmatrix}\bcgreen{U}_{\pi}}_{\bcgreen{U}} \tilde{A}_Q \underbrace{\bcgreen{U}_{\pi}^{\top}\begin{bmatrix}
\bcgreen{U}_Q^\top & 0_d \\
0_d & \bcgreen{U}_Q^\top
\end{bmatrix}}_{\bcgreen{U}^{\top}},
$$
where $\tilde{A}_Q:=\text{blkdiag}( \{M_i\}_{i=1}^n)$ as given in \eqref{def: AQ-block-diag-form}}
Therefore, the equation \eqref{eq: lyp-eq-helper} can be written as
\begin{equation*}
    \bcgreen{U^\top}\Xi_{\infty}\bcgreen{U}= \bcgreen{U}^\top A_{Q}\bcgreen{U}\left(\bcgreen{U}^\top \Xi_{\infty} \bcgreen{U}\right) \bcgreen{U}^\top A_{Q}^\top\bcgreen{U}+\sigma^2 \bcgreen{U}^\top BB^\top \bcgreen{U}.
\end{equation*}
Introducing \bcgreen{$\tilde{\Xi}_\infty=U^\top \Xi_\infty U$, $\tilde{A}_Q=U^\top A_Q U$, and $\tilde{B}=U^\top B$,}
we get
$$
\tilde{\Xi}_\infty = \tilde{A}_Q \tilde{\Xi}_\infty \tilde{A}_Q^\top +\sigma^2 \tilde{B}\tilde{B}^\top.
$$
\bcgreen{It follows from this equation} that the solution $\tilde{\Xi}_\infty = \underset{i=1,..,d}{\text{Diag}}\begin{bmatrix} \tilde{x}_{\lambda_i}& \tilde{y}_{\lambda_i}\\
\tilde{y}_{\lambda_i} & \tilde{v}_{\lambda_i}\end{bmatrix}$ satisfies the following system of equations
\begin{align*}
    \tilde{x}_{\lambda_i}&= c_{i}^2\tilde{x}_{\lambda_i} + 2c_id_i \tilde{y}_{\lambda_i} + d_i^2 \tilde{v}_{\lambda_i} +\alpha^2\sigma^2,\\
    \tilde{y}_{\lambda_i} &= c_i \tilde{x}_{\lambda_i}+d_i\tilde{y}_{\lambda_i},\\
    \tilde{v}_{\lambda_i}&= \tilde{x}_{\lambda_i},
\end{align*}
where $c_i=1+\beta-\alpha(1+\gamma)\lambda_i $ and $d_i=-\beta+\alpha\gamma\lambda_i$, and its solution is given as 
\begin{align*}
\tilde{x}_{\lambda_i}= \frac{(1-d_i)\alpha^2 \sigma^2}{(1+d_{i})[(1-d_i)^2-c_i^2]},\;\;\tilde{y}_{\lambda_i}=\frac{c_i\alpha^2 \sigma^2}{(1+d_{i})[(1-d_i)^2-c_i^2]}, \;\; \tilde{v}_{\lambda_i}=\tilde{x}_{\lambda_i}.
\end{align*}
}\bcred{Define \bcgreen{$\Pi=[I_d, 0_d]\in \mathbb{R}^{d\times 2d}$} with $d$-by-$d$ zero matrix $0_d$. The matrix $Q^{1/2}\Sigma_\infty Q^{1/2}$ is given by $$Q^{1/2}\Sigma_\infty Q^{1/2}=\bcred{Q^{1/2}} \Pi \Xi_\infty \Pi^\top \bcred{Q^{1/2}}= \bcred{Q^{1/2}}\Pi \bcgreen{U} \tilde{\Xi}_\infty \bcgreen{U}^\top \Pi^\top \bcred{Q^{1/2}}.$$ 
Therefore, its eigenvalues are  
$$
\lambda_i(Q^{1/2}\Sigma_\infty Q^{1/2})=\lambda_i(\bcred{Q^{1/2}}\Pi \bcgreen{U} \tilde{\Xi}_\infty \bcgreen{U}^\top \Pi^\top \bcred{Q^{1/2}})=\bcgreen{\frac{\lambda_i(Q)(1-d_i)\alpha^2 \sigma^2}{(1+d_i)[(1-d_i)^2 -c_i^2]},}
$$
\bcgreen{for $i=1,2,..,d$.}
}

\subsection{Proof of Lemma \ref{lem: X-and-Y-pos-def}}\label{app: X-and-Y-pos-def}
We are first going to show that the covariance matrices of $\bc{z}_\infty$ and $\bc{z}_k$ satisfy $\Xi_\infty\succeq \Xi_k$ for all $k\in \mathbb{N}$. \bcgreen{Since the initialization is deterministic, we have $\Xi_0=0$.} Then the induction on $k$ at the recursion \eqref{eq: recurs-Vk-lemma2} implies
$$
\Xi_k= A_Q \Xi_{k-1}A_Q^\top +\sigma^2 BB^\top = \sigma^2\sum_{j=0}^{k-1}(A_Q^j)BB^\top (A_Q^\top)^j.
$$
\bcgreen{By Lemma \ref{lem: feas-set-stab-set}, the assumption $(\alpha,\beta,\gamma)\in\mathcal{F}_\theta$ implies $(\alpha,\beta,\gamma)\in\mathcal{S}_q$ and $\rho(A_Q)<1$. Therefore,}
the solution $\Xi_\infty=\lim_{k\rightarrow \infty}\Xi_k$ \bcgreen{exists} and has the form 
$$\Xi_\infty=\sigma^2\sum_{j=0}^{\infty}(A_Q^j)BB^\top (A_Q^{\top})^j.$$ 
\bcgreen{The matrix $A_Q^{j}BB^\top (A_Q^\top)^j$ is a positive \bcgreen{semi-definite} matrix for any $j\in \mathbb{N}$,} 
$$
\Xi_\infty = \sigma^2 \sum_{j=0}^{\infty}(A_Q^j)BB^\top (A_Q^\top)^j \succeq \sigma^2 \sum_{j=0}^{k-1}(A_Q^j)BB^\top (A_Q^\top)^j = \Xi_k,
$$
for any $k\in\mathbb{N}$. Therefore, we can conclude that $\Xi_\infty-\Xi_{k}\succeq 0$. Notice that the first $d$-by-$d$ leading principle submatrix of $\Xi_\infty-\Xi_k$ is $\Sigma_\infty-\Sigma_k$ \bcgreen{by definition}; therefore, $\Xi_\infty-\Xi_k$ is positive semidefinite if $\Sigma_\infty-\Sigma_k$ and its Schur complement in $\Xi_\infty-\Xi_k$ are also positive semidefinite. This shows that $\Sigma_\infty \succeq \Sigma_k$ for all $k\in \mathbb{N}$. Next, we will show $X=I-\frac{\theta}{2\sigma^2}\Sigma_k^{1/2}Q\Sigma_{k}^{1/2}$ and  $Y=I-\frac{\theta}{2\sigma^2}\Sigma_\infty^{1/2}Q\Sigma_{\infty}^{1/2}$ are positive definite. Notice that $X\succeq Y$ since $\Sigma_\infty\succeq \Sigma_k$; so, it is sufficient to show that $Y\succ 0$ to complete the proof. The result directly follows from the fact that the eigenvalues of $\Sigma_\infty^{1/2}Q\Sigma_{\infty}^{1/2}$ and $Q^{1/2}\Sigma_\infty Q^{1/2}$ are equal and given as $\frac{\sigma^2}{u_i}$ (from Lemma \ref{lem: sol-lyap-eq-quad}). Therefore, for each $i=1,..,d$  $\lambda_i(\Sigma_\infty^{1/2}Q\Sigma^{1/2}_\infty)\prec \frac{2\sigma^2}{\theta}I_d$ on $\mathcal{F}_\theta$ which implies that $Y\succ 0$. This gives the desired result.


\subsection{Proof of Lemma \ref{lem: func-contr-prop}}\label{app: func-contr-prop}
Recall that $y_k=x_k+\gamma(x_k-x_{k-1})$ and $\bcgreen{w_{k+1}}=\tilde{\nabla} f(y_k)-\nabla f(y_k)$. \bcgreen{T}hen the $L$-smoothness of $f$ implies 
\begin{align*}
    f(y_k)-f(x_{k+1}) &\geq \nabla f(y_k)^\top \big( \delta(x_{k-1}-x_{k})+\alpha (\nabla f(y_k)+\bcgreen{w_{k+1}}) \\
    & \quad\quad\quad -\frac{L}{2}\Vert \delta (x_{k-1}-x_k)+\alpha(\nabla f(y_k)+\bcgreen{w_{k+1}})\Vert^2,\\
    & =\nabla f(y_k)^\top\big( \delta(x_{k-1}-x_k)+\alpha \nabla f(y_k) \big) - \frac{L}{2}\Vert \delta (x_{k-1}-x_k)+\alpha \nabla f(y_k)\Vert^2\\
    &\quad\quad +\alpha(1-\alpha L) \bcgreen{w_{k+1}}^\top \nabla f(y_k) -\alpha \delta L \bcgreen{w_{k+1}^\top}  (x_{k-1}-x_k) - \frac{L\alpha^2}{2}\Vert \bcgreen{w_{k+1}}\Vert^2,\\
    &= \begin{bmatrix} 
    \bc{z}_k-\bc{z}_*\\
    \nabla f(y_k)
    \end{bmatrix}^\top
    (\tilde{X}_1 \otimes I_d) \begin{bmatrix} 
    \bc{z}_k-\bc{z}_*\\
    \nabla f(y_k)
    \end{bmatrix}^\top - \bcgreen{w_{k+1}^\top} n_k -\frac{L\alpha^2}{2}\Vert \bcgreen{w_{k+1}}\Vert^2,
\end{align*}
where \bcgreen{$\delta=\beta-\gamma$} and $n_k=\alpha \delta L(x_{k-1}-x_k)-\alpha (1-\alpha L) \nabla f(y_k)$. Similarly, we can use the $\mu$-strong convexity to obtain the inequalities
\begin{align*}
    f(x_k)-f(y_k)&\geq \nabla f(y_k)^\top(x_k-y_k)+\frac{\mu}{2}\Vert x_k-y_k \Vert^2,\\
    &=\gamma\nabla f(y_k)^\top \big(x_{k-1}-x_k\big) + \frac{\mu \gamma^2}{2}\Vert x_{k-1}-x_k\Vert^2,\\
    &= \begin{bmatrix}
    \bc{z}_k-\bc{z}_*\\
    \nabla f(y_k)
    \end{bmatrix}^\top 
    (\tilde{X}_2\otimes I_d )
    \begin{bmatrix} 
    \bc{z}_k-\bc{z}_*\\
    \nabla f(y_k)
    \end{bmatrix},
\end{align*}
and 
\begin{small}
\begin{align*}
    f(x_*)-f(y_k) & \geq \nabla f(y_k)^\top (x_*-y_k) + \frac{\mu}{2}\Vert x_*-y_k\Vert^2,\\
                  & =-\nabla f(y_k)^\top \big( (1+\gamma)(x_k-x_*)-\gamma (x_{k-1}-x_*) \big) \\
                  & \quad\quad\quad \quad\quad\quad\quad\quad+ \frac{\mu}{2}\Vert (1+\gamma)(x_{k}-x_*)-\eta (x_{k-1}-x_*) \Vert^2,\\
                  &= \begin{bmatrix} 
                  \bc{z}_k-\bc{z}_*\\
                  \nabla f(y_k)
                  \end{bmatrix}^\top 
                  (\tilde{X}_3\otimes I_d )
                  \begin{bmatrix} 
                  \bc{z}_k-\bc{z}_*\\
                  \nabla f(y_k)
                  \end{bmatrix}.
\end{align*}
\end{small}
Thus we get 
\begin{small}
\begin{align*}
    \begin{bmatrix} 
    \bc{z}_k-\bc{z}_*\\
    \nabla f(y_k)
    \end{bmatrix}^\top \big( (\tilde{X}_1+\tilde{X}_2)\otimes I_d \big) \begin{bmatrix} 
    \bc{z}_k-\bc{z}_*\\
    \nabla f(y_k)
    \end{bmatrix} &\leq f(x_k)-f(x_{k+1}) + \bcgreen{w_{k+1}^\top} n_k + \frac{L\alpha^2}{2}\Vert \bcgreen{w_{k+1}}\Vert^2,\\
    \begin{bmatrix} 
    \bc{z}_k-\bc{z}_*\\
    \nabla f(y_k)
    \end{bmatrix}^\top 
    \big( (\tilde{X}_1+\tilde{X}_3)\otimes I_d \big) \begin{bmatrix}
    \bc{z}_k-\bc{z}_*\\
    \nabla f(y_k)
    \end{bmatrix} &\leq f(x_*)  - f(x_{k+1}) + \bcgreen{w_{k+1}^\top} n_k + \frac{L\alpha^2}{2}\Vert \bcgreen{w_{k+1}}\Vert^2.
\end{align*}
\end{small}
Notice that above inequality gives us 
\begin{equation}\label{ineq: lyap-ineq-helper-1}
f(x_{k+1})-f(x_*) \leq \rho^2 \left( f(x_k)-f(x_*) \right)- \begin{bmatrix} \bc{z}_{k}-\bc{z}_*\\ \nabla f(y_k)\end{bmatrix}^\top (\tilde{X}\otimes I_d )
\begin{bmatrix}
\bc{z}_k-\bc{z}_*\\ 
\nabla f(y_k)
\end{bmatrix}+ \bcgreen{w_{k+1}^\top} n_k + \frac{L\alpha^2}{2}\Vert \bcgreen{w_{k+1}}\Vert^2.
\end{equation}
\bcgreen{We notice that}
\begin{align*}
(\bc{z}_{k+1}-z_*)^\top P (\bc{z}_{k+1} - z_*)&=\begin{bmatrix} 
\bc{z}_{k}-\bc{z}_* \\
\nabla f(y_k) + \bcgreen{w_{k+1}}
\end{bmatrix}
\begin{bmatrix} 
A^\top P A & A^\top P B \\
B^\top P A & B^\top P B
\end{bmatrix}
\begin{bmatrix}
\bc{z}_k-\bc{z}_*\\
\nabla f(y_k)+\bcgreen{w_{k+1}}
\end{bmatrix}\\
&\;=\begin{bmatrix} 
\bc{z}_k-\bc{z}_*\\
\nabla f(y_k)
\end{bmatrix}^\top 
\begin{bmatrix} 
A^\top P A & A^\top PB \\
B^\top P A & B^\top P B
\end{bmatrix}
\begin{bmatrix} 
\bc{z}_k-\bc{z}_*\\
\nabla f(y_k)
\end{bmatrix}\\ 
&\quad\quad\quad\quad+ 2\bcgreen{w_{k+1}^\top} B^\top P(A (\bc{z}_k-\bc{z}_*)+ B \nabla f(y_k))+\bcgreen{w_{k+1}^\top} B^\top P B\bcgreen{w_{k+1}}.
\end{align*}
Let us define $n'_k=B^\top P (A(\bc{z}_k-\bc{z}_*)+B\nabla f(y_k))$. Since the inequality \eqref{MI} holds:
\begin{multline}\label{ineq: lyap-ineq-helper-2}
    (\bc{z}_{k+1}^\top -z_*) P (\bc{z}_{k+1} - z_*) \\ \leq \rho^2 (\bc{z}_{k}-\bc{z}_*)^\top P (\bc{z}_k-\bc{z}_*)
    + \begin{bmatrix} 
    \bc{z}_k-\bc{z}_* \\ 
    \nabla f(y_k)
    \end{bmatrix}
    (\tilde{X}\otimes I_d) \begin{bmatrix} 
    \bc{z}_k-\bc{z}_*\\
    \nabla f(y_k)
    \end{bmatrix}
    + 2\bcgreen{w_{k+1}^\top} n'_k + \bcgreen{w_{k+1}^\top} B^\top P B\bcgreen{w_{k+1}}.
\end{multline}
Finally, recall that $B=[-\alpha I_d, 0_d]^\top$; therefore, if we combine inequalities \eqref{ineq: lyap-ineq-helper-1} and \eqref{ineq: lyap-ineq-helper-2}, and the definition of the matrix $B$, we obtain 
\begin{small}
\begin{multline}\label{ineq: lyap-ineq}
f(x_{k+1})-f(x_*) + (\bc{z}_{k+1}-\bc{z}_*)^\top P (\bc{z}_{k+1}-\bc{z}_*)\leq \rho^2 (f(x_k)-f(x_*)+(\bc{z}_k-\bc{z}_*)^\top P (\bc{z}_{k}-\bc{z}_*))\\+\bcgreen{w_{k+1}^\top} (n_k + 2n'_k)+ \alpha^2(\frac{L}{2}+\lambda_P)\Vert \bcgreen{w_{k+1}}\Vert^2,
\end{multline}
\end{small}
where $\lambda_P$ is the largest eigenvalue of first $d$-by-$d$ principle submatrix of $P$. This completes the proof.

\subsection{Proof of Lemma \ref{lem: Gaussian-expectation}}\label{app: Gaussian-expectation}
\bcgreen{We introduce} $\tilde{\theta}:=\frac{\theta}{2\sigma^2}$. If $\ttheta < \frac{1}{\sigma^2\alpha^2\bcred{(2\lambda_P+L)}}$ holds, we can write 
\begin{align*}
&\mathbb{E}\left[ e^{\tilde{\theta}(\bcred{m}^\top \veps+\alpha^2\bcred{(\lambda_P+L/2)}\Vert \veps\Vert^2)}\right]\\
&\quad\quad =\frac{1}{(2\pi)^{d/2}\det(\Sigma_w)^{1/2}}\int e^{\ttheta (\bcred{m}^\top \veps + \alpha^2\bcred{(\lambda_P+L/2)} \Vert \veps \Vert^2) -\frac{1}{2}\veps^\top \Sigma^{-1}_w\veps} d\veps,\\
&\quad \quad= \frac{e^{\frac{\ttheta^2}{2}\bcred{m}^\top (\Sigma^{-1}_w-\ttheta \alpha^2\bcred{(2\lambda_P+L)}I_d)^{-1}\bcred{m}}}{(2\pi)^{d/2}\det(\Sigma_w)^{1/2}}\int e^{-\frac{1}{2}(\veps-\veps_*)^\top (\Sigma^{-1}_w-\ttheta\alpha^2\bcred{(2\lambda_P+L)} I_d)(\veps-\veps_*)}d\veps,
\end{align*}
where $\veps_*=\ttheta(\Sigma^{-1}_w-\ttheta\alpha^2\bcred{(2\lambda_P+L)}I_d)^{-1}\bcred{m}$ \bcgreen{and $\Sigma_w=\sigma^2 I_d$}. Notice that $\Sigma^{-1}= \frac{1}{\sigma^2} I_d$ hence the inverse exists from the fact that $\ttheta<\frac{1}{\sigma^2\alpha^2\bcred{(2\lambda_P+L)}}$. The right hand-side is \bcgreen{a} Gaussian integral and yields: 
\begin{equation}\label{eq: Gauss-exp}
    \mathbb{E}\left[ e^{\tilde{\theta}(\bcred{m}^\top \veps^g+\alpha^2\bcred{(\lambda_P+L/2)}\Vert \veps^g\Vert^2)}\right]= \frac{1}{\det (I-\ttheta \alpha^2\bcred{(2\lambda_P+L)} \Sigma_w)^{1/2}}e^{\frac{\ttheta^2}{2} \bcred{m}^\top (\Sigma^{-1}_w-\ttheta\alpha^2\bcred{(2\lambda_P+L)}I_d)^{-1}\bcred{m}}.
\end{equation}

\bcgreen{Using} $\Sigma_w = \sigma^2 I_d$ and writing $\ttheta= \frac{\theta}{2\sigma^2}$ back yields the desired result.
\\

\subsection{Proof of Lemma \ref{lem: bound-on-mk}}\label{app: bound-on-mk}
We are first going to bound the norms of vectors $ n_k=\alpha\delta L(x_{k-1}-x_k)-\alpha(1-\alpha L)\nabla f(y_k)$ and $n'_k:=B^\top P (A(\bc{z}_k-\bc{z}_*)+B\nabla f(y_k))$ given in the proof of Lemma \ref{lem: func-contr-prop}. Notice that $m_k=n_k+2n'_k$; so, these bounds directly provides an upper bound on $\Vert m_k\Vert^2$. \bcgreen{We will first prove}
\begin{equation}\label{ineq: n_k-bound}
    \Vert n_k \Vert^2 \leq  \frac{4\alpha^2L^2}{\mu} \left(2\delta^2 +(1-\alpha L)^2 (1+2\gamma(1+\gamma))\right) \left(\mV_P(\bc{z}_k)+\mV_P(\bc{z}_{k-1})\right).
\end{equation}
\bcgreen{By Cauchy-Schwarz inequality}, 
\begin{align}\label{ineq: n_k-bound-helper-1}
    \Vert n_k\Vert^2 &= \Vert \alpha \delta L(x_{k-1}-x_k)-\alpha (1-\alpha L) \nabla f(y_k)\Vert^2\nonumber,\\
    &\leq 2(\alpha L\delta)^2 \Vert x_{k-1}-x_k\Vert^2 + 2\alpha^2(1-\alpha L)^2\Vert \nabla f(y_k)\Vert^2.
\end{align}
Recall that $\nabla f(x_*)=0$ and $f$ is $L$-smooth: 
\begin{align*}
\Vert \nabla f(y_k)\Vert^2&=\Vert \nabla f(y_k)-\nabla f(x_*)\Vert^2\leq L^2\Vert y_k-x_*\Vert^2,\\ &= L^2\Vert (1+\gamma)(x_k-x_*)-\gamma (x_{k-1}-x_*)\Vert^2,\\
&=L^2(\bc{z}_k-\bc{z}_*)^\top \left(\begin{bmatrix} 
(1+\gamma)^2 & -\gamma (1+\gamma)\\
-\gamma (1+\gamma) & \gamma^2
\end{bmatrix}\otimes I_d\right)
(\bc{z}_k-\bc{z}_*).
\end{align*}
The eigenvalues of the matrix on the right-hand side of \bcgreen{this} equality are $0$ and $(1+\gamma)^2+\gamma^2$ which implies, 
\begin{equation}\label{ineq: n_k-bound-helper-2}
\Vert \nabla f(y_k) \Vert^2 \leq L^2 (1+2\gamma(1+\gamma)) \Vert \bc{z}_k-\bc{z}_*\Vert^2.
\end{equation}
On the other hand, we have the inequality $\Vert x_{k-1}-x_k\Vert^2 \leq 2\Vert x_{k}-x_*\Vert^2 + 2\Vert x_{k-1}-x_*\Vert^2 = 2\Vert \bc{z}_k-\bc{z}_*\Vert^2$. Combining the bounds \eqref{ineq: n_k-bound-helper-1} and \eqref{ineq: n_k-bound-helper-2} yields,
\begin{equation}\label{ineq: nk-bound-app-1}
\Vert n_k\Vert^2 \leq 2\alpha^2L^2 \left(2\delta^2 +(1-\alpha L)^2 (1+2\gamma(1+\gamma))\right)\Vert \bc{z}_k-\bc{z}_*\Vert^2.
\end{equation}
\bc{
The $\mu$-strong convexity and the \bcgreen{the positive semi-definiteness of $P$} imply   
\begin{align*}
\Vert \bc{z}_k-\bc{z}_*\Vert^2 &= \Vert x_k-x_*\Vert^2 + \Vert x_{k-1}-x_*\Vert^2 \leq \frac{2}{\mu}\big(f(x_k)-f(x_*)\big)+\frac{2}{\mu}\big( f(x_{k-1})-f(x_*)\big),\\
&\leq \frac{2}{\mu}\left( \mV_P(\bc{z}_k) +\mV_P(\bc{z}_{k-1})\right),
\end{align*}
which \bcgreen{together with \eqref{ineq: nk-bound-app-1}} shows \eqref{ineq: n_k-bound}.}
Next, we show $n'_k$ admits \bcgreen{the bound}
\begin{equation}\label{ineq: n'_k-bound}
    \Vert n'_k\Vert^2 \leq \alpha^2 \lambda_p\rho^2 \mV_P(\bc{z}_k).
\end{equation}
Remember that $P$ is positive semidefinite matrix; therefore, there exists a square root matrix $P^{1/2}$ such that $P^{1/2}P^{1/2}=P$. Hence,
\begin{align*}
\Vert n'_{k} \Vert^2&=\Vert (B^\top P^{1/2})(P^{1/2} [A(\bc{z}_{k}-\bc{z}_*)+B \bcgreen{\nabla f(y_k)}])\Vert^2,\\
&\leq \alpha^2\lambda_P\begin{bmatrix} 
\bc{z}_k-\bc{z}_* \\ 
\nabla f(y_k) 
\end{bmatrix}^\top 
\begin{bmatrix} 
A^\top P A & A^\top P B \\
B^\top P A & B^\top P B
\end{bmatrix} 
\begin{bmatrix} 
\bc{z}_k-\bc{z}_*\\
\nabla f(y_k)
\end{bmatrix}.
\end{align*} 
The matrix inequality \eqref{MI} implies
\begin{align*}
    \Vert n'_k\Vert^2 \leq \alpha^2\lambda_P \left[\rho^2 (\bc{z}_k-\bc{z}_*)^\top P (\bc{z}_k-\bc{z}_*)+ 
    \begin{bmatrix} 
    \bc{z}_k-\bc{z}_* \\
    \nabla f(y_k)
    \end{bmatrix}^\top
    (\tilde{X}\otimes I_d) 
    \begin{bmatrix} 
    \bc{z}_k-\bc{z}_* \\ 
    \nabla f(y_k)
    \end{bmatrix}\right].
\end{align*}
Because the function $f$ is $L$-smooth and $\mu$-strongly convex, the matrix $\tilde{X}\otimes I_d$ satisfies the following inequality (see \cite[Lemma 5]{hu2017dissipativity})
$$
\begin{bmatrix} 
    \bc{z}_k-\bc{z}_* \nonumber \\
    \nabla f(y_k)
    \end{bmatrix}
    (\tilde{X}\otimes I_d)
    \begin{bmatrix} 
    \bc{z}_k-\bc{z}_* \nonumber\\ 
    \nabla f(y_k)
    \end{bmatrix}\leq \rho^2 (f(x_k)-f(x_*)),
$$
which yields the desired result:
$$
\Vert n'_k \Vert^2 \leq \alpha^2 \lambda_P\rho^2 \left[(\bc{z}_k-\bc{z}_*)^\top P(\bc{z}_k-\bc{z}_*) + f(x_k)-f(x_*)  \right]=\alpha^2 \lambda_P \rho^2 \mV(\bc{z}_k).
$$
\bcgreen{By applying Cauchy-Schwarz again,}
$$
\Vert m_k \Vert^2 = \Vert n_k+2n'_k\Vert^2 \leq 2 \Vert n_k\Vert^2 + 8 \Vert n'_k\Vert^2.
$$
Therefore, \bcgreen{combining everything}
$$
\Vert m_k \Vert^2 \leq 8\alpha^2 \left( \frac{L^2 \mathtt{v}'_{\alpha,\beta,\gamma}}{\mu}+\lambda_P \rho^2\right) \mV_P(\bc{z}_k)+ \frac{8\alpha^2 L^2 \mathtt{v}'_{\alpha,\beta,\gamma}}{\mu}\mV_P(\bc{z}_{k-1}),
 $$
where $\mathtt{v}'_{\alpha,\beta,\gamma}=2\delta^2+(1-\alpha L)^2(1+2\gamma(1+\gamma))$. Notice that $\mathtt{v}_{\alpha,\beta,\gamma}=\left( \frac{2L^2 \mathtt{v}'_{\alpha,\beta,\gamma}}{\mu}+\lambda_P \rho^2\right)$, which completes the proof.
}
\bc{
\subsection{Proof of Lemma \ref{lem: risk-meas-bound-gauss}}\label{app: risk-meas-bound-gauss}
Let us drop the $\vartheta,\psi$ dependency on the notations $\alpha_{\vartheta,\psi},\beta_{\vartheta,\psi}, \gamma_{\vartheta,\psi}, \mathtt{v}_{\vartheta,\psi}$, $\bar{\rho}_{\vartheta,\psi}$, and $\rho_{\vartheta,\psi}$, and use $\alpha, \beta,\gamma, \mathtt{v}, \bar{\rho}$, and $\rho$ instead for simplicity. Also let us denote $\mV_P(\bc{z})$ by $\mV(\bc{z})$. Recall that $\lambda_{P}$ is the largest eigenvalue of the $d$-by-$d$ principle submatrix $P$ and is given as $\lambda_{P}=\frac{\vartheta}{2\alpha}$ by Theorem \ref{thm: TMM-MI-solution}. Define the variables 
\begin{align*}
\mtl = \frac{ \theta4\alpha^2 \mtv}{2-\theta \alpha^2(2\lambda_P+L)},\text{ and }\mtk=\frac{1}{2}\sqrt{(\rho^2+\mtl)^2+ 4\mtl}-\frac{(\rho^2+\mtl)}{2}.
\end{align*}
Notice that $\bar{\rho}^2=\rho^2 + \mtk+\mtl$ by definition \eqref{def: bar-rho}. Before starting the proof, we point out some properties of $\mtk$ and $\bar{\rho}$ which will be useful to obtain our results. First of all by definition of $\rho^2$ and $\lambda_P$ given in Theorem \ref{thm: TMM-MI-solution}, the variable $\theta_u^g$ can be written as 
\bcred{
$$
\theta_u^g = \frac{2(1-\rho^2)}{8\alpha^2\mtv+\alpha^2 (1-\rho^2)(2\lambda_P+L)},
$$
}
therefore, the condition \eqref{cond: cond-on-theta-gaus} \bcred{yields
\begin{align*}
    \theta < \theta_u^g &\implies \theta 8\alpha^2\mtv+\theta \alpha^2 (1-\rho^2)(2\lambda_P+L) < 2(1-\rho^2),\nonumber\\
    & \implies \theta 8\alpha^2\mtv < 2(1-\rho^2)-\theta \alpha^2 (1-\rho^2)(2\lambda_P+L),\\
    & \implies  \frac{\theta 4\alpha^2\mtv}{2-\theta \alpha^2(2\lambda_P+L)}< \frac{1-\rho^2}{2}.
\end{align*}
}
Hence, this implies
\begin{align}
1-\rho^2>2\mtl &\implies 4(1-\rho^2)-4\mtl > 4\mtl, \nonumber\\
&\implies 4-4(\rho^2+\mtl)+(\rho^2+\mtl)^2 > (\rho^2+\mtl)^2+4\mtl, \nonumber\\
&\implies (2-(\rho^2+\mtl))^2>(\rho^2+\mtl)^2+4\mtl \nonumber,\\ 
&\implies  1-\rho^2 -\mtl > \frac{1}{2}\left(\sqrt{(\rho^2+\mtl)^2+4\mtl}-(\rho^2 + \mtl)\right)\nonumber,\\
&\implies 1> \rho^2 + \mtk + \mtl = \bar{\rho}^2\label{ineq: mtk-prop-2}.
\end{align}
Hence, we can see that $\bar{\rho}<1$. Secondly, the definition of $\mtk$ shows us the following relation
\begin{equation}\label{eq: mtk-prop-1}
    (\mtk)^2 + \mtk(\rho^2+\mtl)-\mtl=0 \implies \mtk = \frac{\mtl}{\rho^2 + \mtk + \mtl}.
\end{equation}
After presenting these properties of $\bar{\rho}$ and $\mtk$, we now focus on our proof. In order to show \eqref{ineq: risk-meas-bound-helper-2},  we are first going to prove the following inequality,
\begin{equation}\label{ineq: induct-principle}
\mE[e^{\frac{\theta}{2\sigma^2}\big(\mV(\bc{z}_{k+1})+\mtk \mV(\bc{z}_k) \big)}]\leq \prod_{j=0}^{k-1}\big(1-\theta \bar{\rho}^{2j}\alpha^2(\lambda_P+ \frac{L}{2})\big)^{-d/2} \mE[e^{\frac{\theta}{2\sigma^2}\bar{\rho}^{2k}(\mV(\bc{z}_1)+\mtk\mV(\bc{z}_0))}],
\end{equation}
where $\theta <\frac{2}{\alpha^2 (2\lambda_P+L)}$. Let us show that \eqref{ineq: induct-principle} holds for $k=1$ to have an insight. Notice that if $\theta$ satisfies condition \eqref{cond: cond-on-theta-gaus} then it also satisfies $\theta < \frac{2}{\alpha^2(2\lambda_P+L)}$ (see the \eqref{ineq: theta-bound-2}) and by the statement of Lemma \ref{lem: risk-meas-bound-gauss} the conditions of Lemma \ref{lem: func-contr-prop} also hold. Therefore, we can write the following using Lemmata \ref{lem: func-contr-prop}-\ref{lem: Gaussian-expectation}-\ref{lem: bound-on-mk} and the fact that the noise $w_1$ is independent from $\mathcal{F}_1$,
\begin{align*}
&\mE[e^{\frac{\theta}{2\sigma^2}\big(\mV(\bc{z}_2)+\mtk\mV(\bc{z}_1)\big)}] \\
&\qquad\leq \mE\left[e^{\frac{\theta}{2\sigma^2}\big((\rho^2+\mtk) \mV(\bc{z}_1)+ m_1^\top w_1 + \alpha^2 (\frac{L}{2}+\lambda_P)\Vert w_1\Vert^2\big)}\right],\\
&\qquad=\mE\left[e^{\frac{\theta}{2\sigma^2}(\rho^2+\mtk) \mV(\bc{z}_1)}\mE\left[e^{\frac{\theta}{2\sigma^2}\big(m_1^\top w_1 + \alpha^2 (\frac{L}{2}+\lambda_P)\Vert w_1\Vert^2\big)}\mid \mathcal{F}_1\right]\right],\\
&\qquad\leq \big(1-\theta \alpha^2 (\frac{L}{2}+\lambda_P)\big)^{-d/2} \mE\left[ e^{\frac{\theta}{2\sigma^2} \big( (\rho^2+\mtk) \mV(\bc{z}_1) + \frac{ \theta4\alpha^2 \mtv}{2-\theta\alpha^2 (2\lambda_P+L)}(\mV(\bc{z}_1)+\mV(\bc{z}_0))} \right],\\
&\qquad= \big(1-\theta \alpha^2 (\frac{L}{2}+\lambda_P)\big)^{-d/2} \mE\left[e^{\frac{\theta}{2\sigma^2}(\rho^2 + \mtk + \mtl)\big(\mV(\bc{z}_1)+ \frac{\mtl}{\rho^2+\mtk+\mtl}\mV(\bc{z}_0)\big)}\right].
\end{align*}
We obtain the inequality \eqref{ineq: induct-principle} for $k=1$ using the equality \eqref{eq: mtk-prop-1} and the definition of $\bar{\rho}^2$:
$$
\mE[e^{\frac{\theta}{2\sigma^2}\big( \mV(\bc{z}_2)+\mtk\mV(\bc{z}_1) \big)}] \leq \Big(1-\theta \alpha^2 \big( \frac{L}{2}+\lambda_P \big) \Big)^{-d/2}\mE[e^{\frac{\theta}{2\sigma^2}\bar{\rho}^2 \big( \mV(\bc{z}_1)+\mtk \mV(\bc{z}_0)\big)}].
$$
\bcred{We have already shown that $\theta<\frac{2}{\alpha^2 (2\lambda_P+L)}$ and assumed the conditions of Lemma \ref{lem: func-contr-prop}; so, we can use Lemmata \ref{lem: func-contr-prop}-\ref{lem: Gaussian-expectation}-\ref{lem: bound-on-mk} and we can write the following inequality for $k+1$ by following similar approach, 
\begin{align*}
\mE[e^{\frac{\theta}{2\sigma^2}\big( \mV(\bc{z}_{k+1})+\mtk \mV(\bc{z}_k)\big)}]\leq \big(1-\theta \alpha^2 ( \frac{L}{2}+\lambda_P) \big)^{-d/2} \mE[e^{\frac{\theta}{2\sigma^2}\bar{\rho}^2 \big( \mV(\bc{z}_k)+\mtk \mV(\bc{z}_{k-1})\big)}].
\end{align*}
Moreover, since $\theta\bar{\rho}^2<\theta < \frac{2}{\alpha^2(2\lambda_P+L)}$, we can also use Lemmas \ref{lem: func-contr-prop}-\ref{lem: Gaussian-expectation}-\ref{lem: bound-on-mk} on the integral on the right-hand side of the inequality above:
\begin{equation}\label{ineq: expect-helper-2}
\mE[e^{\frac{\theta}{2\sigma^2}\bar{\rho}^2(\mV(\bc{z}_k)+\mtk \mV(\bc{z}_{k-1}))}]\leq (1-\theta \bar{\rho}^{2}\alpha^2 (\frac{L}{2}+\lambda_P))^{-d/2} \mE[e^{\frac{\theta}{2\sigma^2}\bar{\rho}^2\big((\rho^2+ k_\theta)\mV(\bc{z}_{k-1})+ \frac{\theta \bar{\rho}^{2}4\alpha^2\mtv}{2-\theta\bar{\rho}^2\alpha^2 (2\lambda_P+L)}(\mV(\bc{z}_{k-1})+\mV(\bc{z}_{k-2}))\big)}].
\end{equation}
Notice that $l_{\theta}$ is increasing with respect to $\theta>0$, that is, 
$$
l_{\theta \bar{\rho}^2} = \frac{\theta\bar{\rho}^2 4\alpha^2\mtv}{2-\theta \bar{\rho}^2 \alpha^2 (2\lambda_P+L)}< \frac{\theta 4\alpha^2 \mtv}{2-\theta \alpha^2 (2\lambda_P+L)}=l_{\theta}.
$$
Since $\mV(\bc{z}_{k-1})$ and $\mV(\bc{z}_{k-2})$ are all non-negative, we can convert \eqref{ineq: expect-helper-2} into following using our observation on $l^{\theta}$, 
\begin{align*}
&\mE[e^{\frac{\theta}{2\sigma^2}\bar{\rho}^2\big( \mV(\bc{z}_{k})+\mtk \mV(\bc{z}_{k-1})\big)}]\\
&\quad\quad\quad\leq (1-\theta \bar{\rho}^{2}\alpha^2 (\frac{L}{2}+\lambda_P))^{-d/2}\mE[e^{\frac{\theta}{2\sigma^2}\bar{\rho}^2\big( (\rho^2 + \mtk+ \mtl)\mV(\bc{z}_{k-1})+ \mtl\mV(\bc{z}_{k-2})\big)}],\\
&\quad\quad\quad \leq \big(1-\theta \bar{\rho}^{2}\alpha^2 (\frac{L}{2}+\lambda_P)\big)^{-d/2}\mE[e^{\frac{\theta}{2\sigma^2}(\bar{\rho}^2)^2\big(\mV(\bc{z}_{k-1})+ \mtk\mV(\bc{z}_{k-2})\big)}].
\end{align*}
So far we obtain, 
\begin{equation*}
\mE[e^{\frac{\theta}{2\sigma^2}\big( \mV(\bc{z}_{k+1})+\mtk \mV(\bc{z}_{k})\big)}]
\leq \big(1-\theta\alpha^2 (\frac{L}{2}+\lambda_P)\big)^{-d/2}\big(1-\theta \bar{\rho}^{2}\alpha^2 (\frac{L}{2}+\lambda_P)\big)^{-d/2}\mE\Big[e^{\frac{\theta}{2\sigma^2}\bar{\rho}^4\big(\mV(\bc{z}_{k-1})+ \mtk\mV(\bc{z}_{k-2})\big)}\Big],
\end{equation*}
and repeating these steps for $k$-times gives the desired inequality \eqref{ineq: induct-principle}.
}
Next, we are going to study $\mE[e^{\frac{\theta}{2\sigma^2}\bar{\rho}^{2k}\mV(\bc{z}_1)}]$. Again, Lemmata \ref{lem: func-contr-prop} and \ref{lem: Gaussian-expectation} yield
\begin{align*}
\mE\left[e^{\frac{\theta}{2\sigma^2}\bar{\rho}^{2k}\mV(\bc{z}_1)}\right]
&\leq\mE\left[e^{\frac{\theta \bar{\rho}^{2k}}{2\sigma^2}\big( \rho^2 \mV(\bc{z}_0)+m_0^\top \bc{w_1} + \alpha^2 (\frac{L}{2}+\lambda_P)\Vert \bc{w_1}\Vert^2\big)}\right],\\
&\leq \big( 1-\theta \bar{\rho}^{2k}\alpha^2 (\lambda_P+\frac{L}{2})\big)^{-d/2}\mE\left[e^{\frac{\theta \bar{\rho}^{2k}}{2\sigma^2}\big( \mV(\bc{z}_0)+ \frac{\theta \bar{\rho}^{2k}\Vert m_0\Vert^2}{2(2-\theta \bar{\rho}^{2k} \alpha^2 (2\lambda_P+L))}\big)}\right].
\end{align*}
On the other hand, we can see from the proof of Lemma \ref{lem: bound-on-mk} that $\Vert m_0\Vert^2 < 16\alpha^2\mtv \mV(\bc{z}_0)$ and 
$$
\frac{\theta \bar{\rho}^{2k}}{2-\theta \bar{\rho}^{2k}\alpha^2 (2\lambda_P+L)} \leq \frac{\theta }{2-\theta \alpha^2(2\lambda_P+L)}.
$$
Hence above inequality becomes 
\begin{align}
\mE\left[e^{\frac{\theta}{2\sigma^2}\bar{\rho}^{2k}\mV(\bc{z}_1)}\right]&\leq \big( 1-\theta \bar{\rho}^{2k}\alpha^2 (\lambda_P+\frac{L}{2})\big)^{-d/2}\mE\left[e^{\frac{\theta \bar{\rho}^{2k}}{2\sigma^2}\big( \mV(\bc{z}_0)+ \frac{\theta 8\alpha^2 \mtv \mV(\bc{z}_0)}{(2-\theta \alpha^2 (2\lambda_P+L))}\big)}\right],\nonumber\\
&\leq \big( 1-\theta \bar{\rho}^{2k}\alpha^2 (\lambda_P+\frac{L}{2})\big)^{-d/2}e^{\frac{\theta }{2\sigma^2}(\rho^2+\mtl)2\mV(\bc{z}_0)}.\label{ineq: exp-V1-helper-1}
\end{align}
Combining the inequality \eqref{ineq: induct-principle} with \eqref{ineq: exp-V1-helper-1} yields
\begin{equation}
    \mE[e^{\frac{\theta}{2\sigma^2}\big(\mV(\bc{z}_{k+1})+\mtk \mV(\bc{z}_k) \big)}]\leq \prod_{j=0}^{k}\big(1-\theta \bar{\rho}^{2j}\alpha^2(\lambda_P+ \frac{L}{2})\big)^{-d/2} e^{\frac{\theta}{2\sigma^2}\bar{\rho}^{2(k+1)}2\mV(\bc{z}_0)},
\end{equation}
Let $k\geq 1$. Notice that $\mV(\bc{z}_k)\geq0$ for any $\bc{z}_{k}\in \mathbb{R}^{2d}$ and $\theta, \mtk>0$; hence, the inequality $e^{\frac{\theta}{2\sigma^2}(\mV(\bc{z}_{k})+\mtk \mV(\bc{z}_{k-1}))}>e^{\frac{\theta}{2\sigma^2}\mV(\bc{z}_k)}$ holds pointwise. Hence,
\begin{align*}
\mE[e^{\frac{\theta}{2\sigma^2}(\mV(\bc{z}_{k})+\mtk \mV(\bc{z}_{k-1}))}]
\geq \mE[e^{\frac{\theta}{2\sigma^2}\mV(\bc{z}_k)}] \geq \mE[e^{\frac{\theta}{2\sigma^2}(f(x_k)-f(x_*))}].
\end{align*}
This gives the desired result.
}

\section{Relationship between the sets $\mathcal{F}_\theta$  and $S_q$}\label{lemma-Ftheta-vs-Sq}
\begin{lemma} \label{lem: feas-set-stab-set}
\bc{
The $\theta$-feasible set satisfies $\mathcal{F}_\theta \subset \mathcal{S}_q$; moreover, we have $\lim_{\theta\rightarrow 0}\mathcal{F}_\theta=\mathcal{S}_q=\{(\alpha,\beta,\gamma)~|~ |c_i|\leq |1-d_i|\;\text{and}\; 0<u_i\}$}.
\end{lemma}
\begin{proof}\bc{ We are first going to show $\mathcal{F}_\theta \subset \mathcal{S}_q$.
Let $(\alpha,\beta,\gamma)\in\mathcal{F}_\theta$ for a given $\theta>0$, then $|c_i|<|1-d_i|$ and $u_i >\frac{\theta}{2}$ for all $u_i$. Since $\theta>0$, this also implies $u_i>0$ and the inequality $|c_i|<|1-d_i|$ gives $(1-d_i)^2-c_i^2>0$. Recall that $Q$ is positive definite; hence, we can write  
\begin{equation}\label{feas-set-cond-1}
u_i=\frac{(1+d_i)[(1-d_i)^2-c_i^2]}{\lambda_i(Q)(1-d_i)\alpha^2}>\theta>0 \implies \frac{1+d_i}{1-d_i}>0 \implies 1>|d_i|.
\end{equation}
The stable is defined as $\mathcal{S}_q:= \{ (\alpha,\beta,\gamma)\in \mathbb{R}_{++}\times \mathbb{R}_+\times \mathbb{R}_+ ~|~ \rho(A_Q)=\max\{\rho_i\}<1\}$, where $\rho_i$ is given as 
$$
\rho_i=\begin{cases}
\frac{1}{2}|c_i| +  \frac{1}{2}\sqrt{c_i^2+4d_i},& \text{ if } c_i^2+4d_i>0, \\
\sqrt{|d_i|}, & \text{otherwise}.
\end{cases}
$$
The inequality \eqref{feas-set-cond-1} already shows that $\rho_i<1$ for $c_i^2+4d_i<0$. Next, we will show $|c_i|+\sqrt{c_i^2+4d_i}<2$ for $c_i^2+4d_i$, which in turn shows $\rho_i<1$. Suppose $(\alpha,\beta,\gamma)\in \mathcal{F}_\theta$ and $c_i^2+4d_i>0$. Since, we have already shown $|d_i|<1$, we can use $1-d_i>|c_i|$ and write
\begin{align}
    1-d_i > |c_i| &\implies 4(1-d_i) > 4|c_i| \implies 4-4|c_i| +c_i^2 > c_i^2 + 4d_i, \nonumber\\
    &\implies (2-|c_i|)^2 > c_i^2+4d_i \implies 2> |c_i|+ \sqrt{c_i^2+4d_i} \label{feas-set-cond-2}.
\end{align}
This shows that $\rho_i<1$ in the case $c_i^2+4d_i>0$ for each $i=1,...,d$.Therefore, \eqref{feas-set-cond-1} and \eqref{feas-set-cond-2} show us that for all $(\alpha,\beta,\gamma)\in \mathcal{F}_{\theta}$ then $(\alpha,\beta,\gamma)\in\mathcal{S}_q$ proving $\mathcal{F}_\theta \subset \mathcal{S}_q$. Notice that these inequalities are derived by using the property $\theta>0$. Matter of fact, we can show all of the results only using $u_i>0$. This also yields  $\underset{\theta>0}{\cup}\mathcal{F}_\theta \subset \mathcal{S}_q$. Lastly, let us show $\mathcal{S}_q\subset \underset{\theta>0}{\cup}\mathcal{F}_\theta$, so that we can prove $\underset{\theta>0}{\cup}\mathcal{F}_\theta=\mathcal{S}_q$. In order to show this, we are going to use properties of Lyapunov stability. Particularly, since $\rho(A_Q)<1$ on $\mathcal{S}_q$ the following Lyapunov equation \eqref{eq: Lyap-tilde-V-sol} holds 
\begin{equation}
\tilde{V}_{\infty} = \tilde{A}_{\Lambda}\tilde{V}_{\infty}\tilde{A}_{\Lambda}^{\top} + \sigma^2 \tilde{B}\tilde{B}^\top,
\end{equation}
for a matrix $\tilde{V}_\infty=\underset{i=1,..,d}{\mbox{Diag}}\begin{bmatrix}\bar{x}_{i} & \bar{y}_{i}\\\bar{y}_{i} & \bar{v}_{i} \end{bmatrix}$ with entries
\begin{align*}
\bar{x}_{\lambda_i}= \frac{(1-d_i)\alpha^2 \sigma^2}{(1+d_{i})[(1-d_i)^2-c_i^2]},\;\;\bar{y}_{\lambda_i}=\frac{c_i\alpha^2 \sigma^2}{(1+d_{i})[(1-d_i)^2-c_i^2]}, \;\; \bar{v}_{\lambda_i}=\bar{x}_{\lambda_i},
\end{align*}
Since the solution exists, it follows from Lyapunov theory that the matrix $\tilde{V}_\infty$ is positive semidefinite. This holds if and only if
$$
\bar{x}_{\lambda_i}=\frac{(1-d_i)\alpha^2 \sigma^2}{(1+d_{i})[(1-d_i)^2-c_i^2]}>0,
$$
and the determinant is non-negative 
$$
\bar{x}_{\lambda_i}^2-\bar{y}_{\lambda_i}^2 = \frac{[(1-d_i)^2 -c_i^2] \alpha^4\sigma^4}{(1+d_i)^2 [(1-d_i)^2-c_i^2]^2}>0.
$$
for all $i=1,..,d$.These yield the conditions
$$
u_i> 0 \text{ and } (1-d_i)^2> c_i^2,
$$
which in turn implies $(\alpha,\beta,\gamma)\in \mathcal{F}_\theta$. This gives the desired result $\lim_{\theta\rightarrow 0}\mathcal{F}_\theta=\mathcal{S}_q$, \bcgreen{where we used the fact that the sets $\mathcal{F}_\theta$ are monotonically increasing as $\theta$ decreases, i.e. $\mathcal{F}_{\theta_1}\subset \mathcal{F}_{\theta_2}$ for the $\theta_1 > \theta_2$. Finally, we can compute $\lim_{\theta \rightarrow 0}\mathcal{F}_\theta$ by inserting $\theta=0$ in \eqref{cond: var-on-Sigma-quad}}.} 
\end{proof}

\end{document}